\documentclass[10pt,twoside]{amsart}
\usepackage[hidelinks]{hyperref}
\usepackage{amssymb}
\usepackage{amscd}
\usepackage{amsmath}
\usepackage[all]{xy}
\usepackage{mathrsfs}
\usepackage{graphicx}
\usepackage{amsmath,amsthm,amssymb,amsfonts,tikz-cd, xcolor,inputenc}
\usepackage{soul}

\newtheorem{thm}{Theorem}[section]
\newtheorem{prop}[thm]{Proposition}
\newtheorem{proposition}[thm]{Proposition}
\newtheorem{lemma}[thm]{Lemma}

\newtheorem{example}[thm]{Example}

\newtheorem{lemma-def}[thm]{Lemma-Definition}
\newtheorem{theorem}{Theorem}[section]
\newtheorem{corollary}[thm]{Corollary}
\newtheorem{definition}[thm]{Definition}
\newtheorem{remark}[thm]{Remark}
\newtheorem{remarks}[thm]{Remarks}

\newtheorem{As}[thm]{Assumption}

\newtheorem{mainthm}{Theorem}

  1

\newcommand{\NN}{\mathbb N}
\newcommand{\QQ}{\mathbb Q}

\newcommand{\ZZ}{\mathbb Z}

\newcommand{\UU}{\mathbb U}
\newcommand{\F}{\mathcal F}

\newcommand{\co}{\mathcal O}
\newcommand{\cU}{\mathcal U}
\newcommand{\cT}{\mathcal T}

\newcommand{\cN}{\mathcal N}

\newcommand{\B}{\mathfrak B}

\newcommand{\cH}{\mathcal H}

\newcommand{\cm}{\mathcal M}

\newcommand{\cW}{\mathcal{W}}

\newcommand{\calL}{\mathcal L}
\newcommand{\p}{\mathfrak{p}}
\newcommand{\q}{\mathfrak{q}}

\newcommand{\calP}{\mathcal P}
\newcommand{\calK}{\mathcal K}

\newcommand{\Mat}{\mathrm{M}}
\newcommand{\Fq}{{\mathbb F}_q}
\newcommand{\Fl}{{\mathbb F}_\ell}
\newcommand{\expo}{{\rm exp}^{-1}_E(E(\co_K))}
\newcommand{\expm}{{\rm exp}^{-1}_E(E(\cm))}

\newcommand{\FiG}{\F_\infty[G]}
\newcommand{\Fi}{\F_\infty}

\DeclareMathOperator{\Fr}{Fr}

\DeclareMathOperator{\GL}{GL}

\DeclareMathOperator{\Nrd}{Nrd}

\DeclareMathOperator{\Spec}{MSpec}

 \DeclareMathOperator{\Fit}{Fit}

 \DeclareMathOperator{\im}{im}
 \DeclareMathOperator{\Gal}{Gal}
 
\DeclareMathOperator{\id}{id}

\DeclareMathOperator{\Hom}{Hom} 
\DeclareMathOperator{\Ext}{Ext} 
 
\DeclareMathOperator{\cok}{cok}

\renewcommand{\det}{\text{det}}

\DeclareMathOperator{\ev}{ev}
\newcommand{\evt}{\ev_{t^{-1}}}
\DeclareMathOperator{\KFit}{KFit}

\def\YEAR{\year}\newcount\VOL\VOL=\YEAR\advance\VOL by-1995
\def\firstpage{1}

\makeatletter
\def\magnification{\afterassignment\m@g\count@}
\def\m@g{\mag=\count@\hsize6.5truein\vsize8.9truein\dimen\footins8truein}
\makeatother

\makeatletter
\@specialpagefalse\headheight=8.5pt

\newcommand{\Addresses}{{% additional braces for segregating \footnotesize
  \bigskip
  \footnotesize
  \parindent0pt

  \textsc{Universität Bonn, Mathematisches Institut, Endenicher Allee 62, 53115 Bonn (Germany).}\\
  \textit{E-mail address}: \texttt{midff@math.uni-bonn.de}

  \medskip

  \textsc{Instituto de Ciencias Matem\'aticas, Calle Nicol\'as Cabrera 13-15, Campus de Cantoblanco, 28049 Madrid (Spain)}.\\
  \textit{E-mail address}: \texttt{daniel.macias@icmat.es}
  
  \medskip

  \textsc{Departamento de Matem\'aticas, Universidad Aut\'onoma de Madrid, Campus de Cantoblanco, 28049 Madrid (Spain); and Instituto de Ciencias Matem\'aticas, 28049 Madrid (Spain)}.\\
  \textit{E-mail address}: \texttt{daniel.martinezm@uam.es}
}}

 \newcommand\blfootnote[1]{%
     \begingroup
     \renewcommand\thefootnote{}\footnote{#1}%
     \addtocounter{footnote}{-1}%
      \endgroup
    }

\begin{document}

\title[The refined class number formula]{The refined class number formula \\ for Drinfeld modules}
\author[de Frutos-Fern\'andez]{Mar\'ia In\'es de Frutos-Fern\'andez}
\author[Mac\'ias Castillo]{Daniel Mac\'ias Castillo}
\author[Mart\'inez Marqu\'es]{Daniel Mart\'inez Marqu\'es}

\begin{abstract}
Let $K/k$ be a finite Galois extension of global function fields. Let $E$ be a Drinfeld module over $k$. We state and prove an equivariant refinement of Taelman's analogue of the analytic class number formula for $(E,K/k)$, and derive explicit consequences for the Galois structure of the Taelman class group of $E$ over $K$.

\end{abstract}

\maketitle


\vspace{-10pt}
\tableofcontents
\section{Introduction}\label{intro}

In the seminal article \cite{TAM}, Taelman proved an analogue of the analytic class number formula for Drinfeld modules. He stated that `it should be possible to formulate and prove an equivariant version' of this formula (cf. Remark 10 in loc. cit.).

Let $\ell$ be a prime number and let $q$ be a power of $\ell$. Let $k$ be a finite separable extension of the rational function field $\Fq(t)$. Let $K$ be a finite Galois extension of $k$ with group $G:={\rm Gal}(K/k)$.

Let $E$ be a Drinfeld $\Fq[t]$-module defined over the ring of integers $\co_k$ of $k$. We may then also regard $E$ as being defined over $\co_K$, and $G$ obviously acts on the $\Fq[t]$-modules $\co_K$ and $E(\co_K)$. 
The exponential map 
of $E$ is Galois-equivariant and this fact implies that $G$ acts naturally on the $\Fq[t]$-modules 
$\UU(E/\co_K)$ and $H(E/\co_K)$
that are studied by Taelman \cite{TMA} as respective analogues for $E/\co_K$ of the unit group and of the ideal class group 
of the ring of integers of a global field. 

An equivariant, or `refined', version of Taelman's class number formula for $(E,K/k)$ should encode the action of $G$ on these arithmetic $\Fq[t]$-modules. In the main result of this article, Theorem \ref{MT}, we prove such a formula, which specialises to recover Taelman's for $K=k$.

In the case that the Galois group $G$ of $K/k$ is abelian, a refined class number formula has been recently obtained by Ferrara, Green, Higgins and Popescu \cite{fghp}, who refer to it as an `equivariant Tamagawa number formula'. Our formula also specialises to recover it for such abelian extensions.

In the sequel we set $A:=\Fq[t]$, $\F:=\Fq(t)$ and $\F_\infty:=\Fq(\!(t^{-1})\!)$, the completion of $\F$ at the place at infinity (i.e., of uniformiser $t^{-1}$). We interpret our given Drinfeld $A$-module $E/\co_k$ as a functor that assigns $\co_k[G]$-modules $M$ a new $A$-action, thus a new $A[G]$-module structure $E(M)$.

\subsection{The refined trace formula}\label{11}
\subsubsection{Characteristic classes}
Taelman interprets the special $L$-value of $E/\co_K$ as an Euler product whose factors compare the characteristic polynomials of the finite $A$-modules $\co_K/\p$ and $E(\co_K/\p)$, as $\p$ runs through the maximal ideals of $\co_K$.

At the moment there is however no suitable notion of `reduced determinant' or `reduced norm' over general group rings in positive characteristic (but see also \S \ref{13} below for a more detailed discussion of this issue). 

As a key component of our approach and a generalisation of the notion of characteristic polynomial we are led to define, for suitable finite $A[G]$-modules $M$, the `characteristic class' $c_G(M)$
of $M$ to be the class of the automorphism
$$(t\otimes 1)-(1\otimes t)$$
of $\FiG\otimes_{\Fq[G]}M$, regarded as an element of the Whitehead group $K_1(\FiG)$ of the group ring $\FiG$ (see Definition \ref{ccdef} for details).

We observe that if $G$ is abelian, then the determinant of the characteristic class of $M$ is simply a generator of the Fitting ideal of $M$ in $A[G]$ while, if $\ell$ does not divide $|G|$, then its reduced norm is a generator of the non-commutative Fitting invariant of $M$, as defined by Nickel \cite{nickel} or by Burns and Sano \cite{nagm} (see Lemma \ref{Fitchar}).

In the general case, characteristic classes allow us to define equivariant special $L$-values for $(E,K/k)$ as Euler products in the abelian group $K_1(\FiG)$.

\subsubsection{Convergence and the trace formula}
As an essential intermediate result towards the proof of his main theorem, Taelman proves a `trace formula' in $\F_\infty^\times$ for the Euler products that arise from global `nuclear' automorphisms (see \cite[Thm. 3]{TAM}). This result is a generalisation of the trace formula of Anderson \cite{Anderson}.

As a special case, Taelman deduces the convergence in $\F_\infty^\times$ of the special $L$-value of $E$ to the (infinite-dimensional) determinant of a certain global action defined by $E$. 

We endow $K_1(\FiG)$ with the quotient topology of the natural group topology in $\FiG^\times$ and prove a refined trace formula 
in this topological group (see Theorem \ref{tf}). As far as we are aware, the idea of studying the convergence of Euler products directly in an algebraic $K$-group is completely new.

In Corollary \ref{tfcor} we then deduce a convergence formula for the special $L$-values for $(E,K/k)$. 
For $K=k$, this result corresponds to Taelman's via the determinant map $K_1(\F_\infty)\cong \F_\infty^\times$. 
More generally, if $G$ is abelian, then it corresponds via the determinant $K_1(\F_\infty[G])\cong \F_\infty[G]^\times$ to \cite[Cor. 3.0.3]{fghp}.
One should thus think of the refined trace formula as a direct relation between the relevant $A[G]$-actions. 

As a key step towards the refined trace formula, we prove a topological extension of the structural result \cite[Prop. 1.5.1]{fukaya-kato} of Fukaya and Kato for Whitehead groups of adic rings (see Proposition \ref{151}).

\subsection{The refined class number formula}\label{12}
\subsubsection{The lattice of units and the class group}
In \cite{TMA}, Taelman studies the kernel $\UU(E/\co_K)$ and the cokernel $H(E/\co_K)$ of the exponential map
\begin{equation}\label{Kinfty}K_\infty\,\longrightarrow\, E(K_\infty)/E(\co_K)\end{equation}
of $E$. Here we have set $K_\infty:=\F_\infty\otimes_{\F}K$.
He proves that $\UU(E/\co_K)$ is a finitely generated $A$-submodule of $K_\infty$, while $H(E/\co_K)$ is a finite $A$-module. 

As an equivariant version of these results, we prove in Theorem \ref{PBT} that appropriate modifications of (\ref{Kinfty}), which we refer to as `complexes of units' for $(E,K/k)$, are \emph{perfect} complexes of $A[G]$-modules: that is, isomorphic in the derived category to a bounded complex of finitely generated, projective modules. 

\subsubsection{The refined Euler characteristic of the complex of units}
Taelman's class number formula for $E/\co_K$ in \cite{TAM} states that the ratio of co-volumes in $\F_\infty^\times$ of the $A$-lattices $\co_K$ and $\UU(E/\co_K)$, multiplied by the characteristic polynomial of $H(E/\co_K)$, coincides with the special $L$-value of $E/\co_K$. To be more precise, this first term is defined through the determinant that measures the relative positions of $\co_K$ and $\UU(E/\co_K)$ in $K_\infty$, but see also Remarks 8 and 14 in loc. cit.

Our main result, the refined class number formula for $(E,K/k)$ (see Theorem \ref{thm:mainA}, and Theorem \ref{MT} for a precise version thereof), compares the special $L$-value for $(E,K/k)$ to the refined Euler characteristic of the complex of units for $(E,K/k)$. Specifically, we consider the image of the special $L$-value for $(E,K/k)$ under the canonical homomorphism $$K_1(\FiG)\,\longrightarrow \,K_0(A[G],\FiG)$$ to the relative algebraic $K_0$-group of the ring inclusion $A[G]\subset\FiG$ (see Remark \ref{impartial} for the explicit definition of this homomorphism).

Relative $K_0$-groups of inclusions of group rings are by now regarded as the natural environment for equivariant special $L$-value formulas (see for example \cite{bdals,mcrc,bufl99,bk,cfksv,kakde,RWnew}). In particular, $K_0(A[G],\FiG)$ is the natural environment for the refined Euler characteristic of a pair $(C,\mu)$ comprising: a perfect complex $C$ of $A[G]$-modules (that satisfies an additional condition as in Definition \ref{Dpc}); and an $\Fi$-trivialisation $\mu$ of $C$ (in the sense of \cite[\S 2.3.2]{omac}). The adjective `refined' refers here to the fact that the refined Euler characteristic of $(C,\mu)$ maps to the classical Euler characteristic of $C$ in the Grothendiek group $K_0(A[G])$.

Such refined Euler characteristics can be defined by mimicking the use of Deligne's determinant functor (with values in the category of virtual objects) that is made by Breuning and Burns in \cite{Breuning Burns} (see Remark \ref{ECRemark}). They provide a conceptual and general interpretation of the ratios of co-volumes that occur in \cite{TAM,fghp}.

For simplicity, we use instead a completely explicit definition in the special case of interest for this article (see Definition \ref{ECex}). In particular, other than to state Theorem \ref{thm:mainA}, we do not use the term `trivialisation' in the rest of this article.

Our refined class number formula (Theorem \ref{thm:mainA}, Theorem \ref{MT}) specialises to recover the equivariant Tamagawa number formula of Ferrara, Green, Higgins and Popescu \cite{fghp} for abelian extensions, and thus also Taelman's main result in \cite{TAM} for $K=k$ (see Remark \ref{impartial}, and Remark \ref{forearlierref} for further details).

\subsubsection{Taming modules}
In the above discussion we have referred to special $L$-values and to complexes of units for $(E,K/k)$, both in plural, so as to avoid mentioning a specific technical issue, which we now address explicitly to state Theorem \ref{thm:mainA}.

By Noether's Theorem, 
$\co_K$ is projective as an $\co_k[G]$-module if and only if the ring extension $\co_K/\co_k$ is tamely ramified. In fact, given a maximal ideal $\p$ of $\co_k$, the quotient $\co_K/\p\co_K$ is $\Fq[G]$-projective if and only if $\p$ is tamely ramified in $K/k$, but the notion of characteristic class only suits finite $A[G]$-modules that are $\Fq[G]$-projective.

In a similar way, it follows easily from Noether's Theorem that if $\co_K/\co_k$ is not tamely ramified, then the complex of $A[G]$-modules (\ref{Kinfty}) cannot itself be perfect.

To get around these technical difficulties, we use the notion of `taming module' introduced by Ferrara--Green--Higgins--Popescu \cite[Def. 1.3.2]{fghp}. 
Although it is discussed in the abelian setting in loc. cit., 
taming modules $\cm$ exist for general Galois extensions $K/k$ (see Definition \ref{defTM} and Remark \ref{TMremark3}),
with the following basic properties: \begin{itemize}\item $\cm$ is a projective $A[G]$-submodule of $\co_K$ with finite index, and
\item for each maximal ideal $\p$ of $\co_k$ the $\Fq[G]$-module $\cm/\p\cm$ is free,  
with $$\cm/\p\cm=\co_K/\p\co_K$$ whenever $\p$ is tamely ramified in $K/k$.\end{itemize}
In fact, if $\co_K/\co_k$ is tamely ramified, then the only taming module for $K/k$ is $\cm=\co_K$.

For a given taming module $\cm$, one can thus define a $K$-theoretic Euler factor at a wildly ramified prime $\p$ of $\co_k$ as the difference between the characteristic classes $c_G(\cm/\p\cm)$ and $c_G(E(\cm/\p\cm))$ of the $A[G]$-modules $\cm/\p\cm$ and $E(\cm/\p\cm)$. Regarding $\cm$ as fixed, we are thus led to define the `$\cm$-modified $L$-value' $\Theta^{\cm}_{E,K/k}$ of $(E,K/k)$ as the corresponding Euler product, which converges in $K_1(\FiG)$ (Corollary \ref{tfcor}) as a consequence of our refined trace formula (Theorem \ref{tf}).

In a similar way, for our fixed $\cm$, we define the `$\cm$-modified complex of units' of $(E,K/k)$ to be the complex $C^{\cm}_{E,K/k}$ of $A[G]$-modules $$K_\infty\,\xrightarrow{({\rm exp}_E,0)}\,\left(E(K_\infty)/E(\cm)\right)\,\oplus\,\cm,$$ with the first term in degree one and ${\rm exp}_E$ induced by the exponential map of $E$.

Our main result, Theorem \ref{MT}, may then be roughly stated as follows.

\begin{mainthm}\label{thm:mainA} Let $E$ be a Drinfeld module defined over $\co_k$ and let $\cm$ be a taming module for $K/k$. Then the following claims are valid.
\begin{itemize}
    \item[(1)]The infinite product
$$\Theta^\cm_{E,K/k}:=\prod_{\p\in\Spec(\co_k)}\left(c_G\bigl(\cm/\p\cm\bigr)\cdot c_G\bigl(E(\cm/\p\cm)\bigr)^{-1}\right)$$
converges to a unique element of $K_1(\Fi[G])$.
\item[(2)]The complex of $A[G]$-modules $C^{\cm}_{E,K/k}$ is perfect.
\item[(3)]The canonical image of $\Theta^{\cm}_{E,K/k}$ in $K_0(A[G],\F_\infty[G])$ is equal to the inverse of the refined Euler characteristic of the pair $(C^{\cm}_{E,K/k},\mu)$, where
$\mu$ is the $\Fi$-trivialisation of $C^{\cm}_{E,K/k}$ induced by the inclusions $\cm\subseteq\co_K\subset K_\infty$.
\end{itemize}
\end{mainthm}

See also Remark \ref{convergencerk} and Corollary \ref{tfcor} for further details on claim (1), and Remark \ref{formofreprk} and Theorem \ref{PBT} for further details on claim (2).

Although both sides of the equality in claim (3) depend on the given choice of $\cm$, the equality itself is \textit{a priori} independent of the choice of taming module.

We also refer the reader to Remark 6.2.4 in \cite{fghp} for a discussion of the analogy between the above instances of $\cm$-modification, and the $T$-modifications that have been commonplace in the study of equivariant special Artin $L$-values for global fields since the work of Gross in \cite{gross} and of Rubin in \cite{rubin}. We simply recall that such $T$-modifications are necessary in order to circumvent certain technical difficulties associated with the presence of roots of unity.

\subsection{A construction of non-abelian Stickelberger elements}\label{13}In this section we discuss some explicit consequences of our main result (Theorem \ref{thm:mainA}, Theorem \ref{MT}).

By Maschke's Theorem, the group rings $\F[G]$ and $\F_\infty[G]$ are semisimple if and only if the characteristic $\ell$ does not divide the degree of the extension $K/k$. The theory of reduced norms is often restricted to semisimple rings in the literature.

More generally, given a finite group $G$, it is also a classical result (see Proposition \ref{NEWfiniteproducts}) that group rings for $G$ in characteristic $\ell$ decompose as finite direct sums of matrix rings over commutative rings if and only if $G$ has an abelian $\ell$-Sylow subgroup and a normal $\ell$-complement, if and only if $\ell$ does not divide the order of the commutator subgroup of $G$. 
Over such group rings, it is still possible to define a well-behaved notion of reduced determinant that recovers semisimple reduced norms, as well as determinants in the abelian case (see Definition \ref{rdfidefs}).

In the rest of \S \ref{13}, we assume that $\ell$ does not divide the degree of $K$ over the maximal abelian subextension of $k$ in $K$. 
For any Drinfeld module $E/\co_k$ and any taming module $\cm$ for $K/k$, we then define the `$\cm$-modified Stickelberger element' of $(E,K/k)$ as the reduced determinant of the $\cm$-modified $L$-value,
$$\theta^{\cm}_{E,K/k}\,:=\,{\rm Nrd}_{\F_\infty[G]}\left(\Theta^{\cm}_{E,K/k}\right)\,\in\,Z(\F_\infty[G])^\times.$$
Here $Z(\F_\infty[G])^\times$ denotes the unit group of the centre $Z(\FiG)$ of $\FiG$.

Reduced determinants also allow us to define non-commutative Fitting ideals over $A[G]$, in a manner that is compatible with the constructions of Nickel \cite{nickel}, Johnston-Nickel \cite{jn} and Burns-Sano \cite{nagm} for orders in semisimple algebras (see Definition \ref{rdfidefs}).

As a consequence of Theorem \ref{MT}, we then derive in Theorem \ref{mtII} that suitable normalisations of $\theta^{\cm}_{E,K/k}$ belong to the non-commutative Fitting ideal of the `$\cm$-modified Taelman class group' 
$$H(E/\cm)\,:=\,\cok\left(K_\infty\,\xrightarrow{{\rm exp}_E}\, E(K_\infty)/E(\cm)\right)$$
of $(E,K/k)$. More precisely, we define a subset $\mathcal{R}$ of $Z(\FiG)$ for which
\begin{align*}Z(A[G])\cdot\theta^{\cm}_{E,K/k}\cdot\mathcal{R}\,\,\subseteq&\,\Fit_{A[G]}\left(H(E/\cm)\right)\\ \subseteq&\,\Fit_{A[G]}\left(H(E/\co_K)\right)\\ \subseteq&\,{\rm Ann}_{Z(A[G])}\left(H(E/\co_K)\right)\end{align*}
as ideals in the centre $Z(A[G])$ of $A[G]$. Let us note that $\theta^{\cm}_{E,K/k}$ is typically transcendental, and that $\mathcal{R}$ is amenable to computation and always non-trivial (see Remark \ref{n-tremark} and \S \ref{322}). 

For any finite abelian extension $K/k$ and any Drinfeld module $E/\co_k$, Theorem \ref{mtII} recovers the `refined Brumer-Stark Theorem for Drinfeld modules' \cite[Thm. 1.5.5]{fghp} of Ferrara, Green, Higgins and Popescu (see Remark \ref{BScomp}).

For any finite Galois extension $K/k$ of degree not divisible by $\ell$ and any Drinfeld module $E/\co_k$, one necessarily has $\cm=\co_K$ (so we may drop the superscript $\cm$) and, as Corollary \ref{mtIII} to Theorem \ref{mtII}, we prove the equality of $Z(A[G])$-ideals
 $$Z(A[G])\cdot\theta_{E,K/k}\cdot\mathcal{R}=\Fit_{A[G]}\left(H(E/\co_K)\right).$$
This equality is a vast non-abelian generalisation of Theorem A of Angl\`es and Taelman in \cite{at} (see Remark \ref{ATcomp}).

The question of how to construct non-abelian Stickelberger elements in completely general finite Galois extensions $K/k$ remains open. We will return to it in future work.

\subsection{Structure of the article} In \S\ref{2} we introduce characteristic classes and refined Euler characteristics and we state Theorem \ref{MT}. In \S\ref{3} we introduce reduced determinants and non-commutative Fitting ideals and we state Theorem \ref{mtII} and Corollary \ref{mtIII}. In \S\ref{4} we prove a topological extension of a result of Fukaya and Kato for Whitehead groups and we use it to associate a suitable class to nuclear automorphisms. In \S\ref{5} we prove the refined trace formula and we apply it to derive the convergence of the special $L$-values. In \S\ref{proofMT} we prove Theorem \ref{MT} and in \S\ref{7} we prove Theorem \ref{mtII}.

\subsection{Acknowledgements} The authors are very grateful to Oliver Braunling for his generous and helpful advice regarding the proof of Proposition \ref{Braunling} that we present in \S \ref{Nsection}. They are also very grateful to Carlos de Vera Piquero for his interest in this project and for actively and frequently participating in pertinent discussions. 

They also thank Bruno Angl\`es, Francesc Bars, Dominik Bullach, Jos\'e Ignacio Burgos Gil, David Burns, Alexandre Daoud, Andrei Jaikin-Zapirain, Henri Johnston, Donghyeok Lim, Maxim Mornev, Cristian Popescu and Lenny Taelman for many helpful discussions and correspondance, and an anonymous referee for their careful reading, and their generous comments and suggestions, that helped improve the article.

The first author acknowledges support from the Grant CA1/RSUE/2021-00623 funded by the Spanish Ministry of Universities, the Recovery, Transformation and Resilience Plan, and Universidad Autónoma de Madrid; and from the Deutsche Forschungsgemeinschaft (DFG, German Research Foundation) under Germany’s Excellence Strategy – EXC-2047/1 – 390685813.
The second author acknowledges support for this article as part of Grants CEX2023-001347-S, PID2022-142024NB-I00 and CNS2023-145167 funded by MICIU/AEI/10.13039/501100011033.

\section{The refined class number formula}\label{2}

In this section we state our main result.
We fix a prime number $\ell$ and a power $q$ of $\ell$. 
We set $A:=\Fq[t]$, $\F:=\Fq(t)$ and $\F_\infty:=\Fq(\!(t^{-1})\!)$. 
For any field extension $L/\F$ and $A$-module $M$, we set $M_L:=L\otimes_A M$ and use identical notation for homomorphisms.

\subsection{Characteristic classes}
We first define the notion of characteristic classes in Whitehead groups.

\subsubsection{}\label{211}
Unless otherwise specified, we regard all rings $R$ as unital and all $R$-modules as left $R$-modules. We write $Z(R)$ for their centre, $R^\times$ for their unit group and $K_1(R)$ for their Whitehead group.

We recall that, in addition to its interpretation as the abelianisation of $\GL(R)$, 
the latter group is the $K_1$-group of the category of finitely generated, projective $R$-modules (cf. \cite{swan,curtisr}). We write $[P,\phi]$, or simply $[\phi]$ when $P$ is clear from context, for the class of an automorphism $\phi$ of a finitely generated projective $R$-module $P$.
We always use multiplicative notation for $K_1$-groups.

A ring is said to be semilocal if the quotient by its Jacobson radical is a semisimple Artinian ring. 
If $R$ is semilocal, then Bass' Theorem implies that the canonical group homomorphism $R^\times\to K_1(R)$ is surjective (see \cite[Vol. II, Thm. (40.31)]{curtisr}).

If $R$ is a semilocal topological ring, we endow $R^\times$ with the group topology induced by the injection
\[R^\times\,\longrightarrow\,R\times R\]
that maps $u\in R^\times$ to $(u,u^{-1})$. We then endow the quotient $K_1(R)$ of $R^\times$ with the corresponding group topology.
In general, this topology need not be Hausdorff, even if $R$ is. In Proposition \ref{151} below we prove that if $R$ is an adic ring in the sense of \cite{fukaya-kato}, then $K_1(R)$ is in fact a profinite abelian group.

\begin{lemma}\label{ctsopenlem} Let $\beta:\Lambda\to S$ be an injective, continuous and open homomorphism between semilocal topological rings. Then the restriction $\gamma:\Lambda^\times\to S^\times$ of $\beta$, and the map $\delta:K_1(\Lambda)\to K_1(S)$ induced by $\beta$, are both continuous and open. In particular, $\im(\delta)$ is a clopen subgroup of $K_1(S)$.
\end{lemma}
\begin{proof} 
We write $\iota_1$, $\iota_2$ for the respective continuous injections $\Lambda^\times\to\Lambda\times\Lambda$ and $S^\times\to S\times S$ that map $u$ to $(u,u^{-1})$, note that $\iota_2\circ\gamma=(\beta\times\beta)\circ\iota_1$, and fix open subsets $U_1,U_2$ of $\Lambda$ and $V_1,V_2$ of $S$.
Then $\gamma$ is continuous because
$$\gamma^{-1}(\iota_2^{-1}(V_1\times V_2))=\iota_1^{-1}((\beta\times\beta)^{-1}(V_1\times V_2))=\iota_1^{-1}(\beta^{-1}(V_1)\times\beta^{-1}(V_2))$$
is open. Also, $\gamma$ is open because the following image is open:
\begin{align*}\gamma(\iota_1^{-1}(U_1\times U_2))=&\{\beta(\lambda):\lambda\in\Lambda^\times, \lambda\in U_1, \lambda^{-1}\in U_2\}\\
=&\{s\in(S^\times\cap\beta(U_1)): s^{-1}\in(S^\times\cap\beta(U_2))\}&\\
=&\{s\in S^\times:s\in\beta(U_1),s^{-1}\in\beta(U_2)\}&=\iota_2^{-1}(\beta(U_1)\times\beta(U_2)).
\end{align*}
Here the second equality is valid because if $s=\beta(u_1)\in S^\times$ with $s^{-1}=\beta(u_2)$ and with $u_i\in U_i$, then $1=\beta(u_1u_2)$ so, since $\beta$ is injective, $u_1$ belongs to $\Lambda^\times$ with $u_1^{-1}=u_2$.

We next write $\pi_1$, $\pi_2$ for the respective surjections $\Lambda^\times\to K_1(\Lambda)$ and $S^\times\to K_1(S)$, recall that they are continuous and open, note that $\delta\circ\pi_1=\pi_2\circ\gamma$, and fix open subsets $U$ of $K_1(\Lambda)$ and $V$ of $K_1(S)$. Then $\delta$ is continuous because
$$\delta^{-1}(V)=\pi_1(\gamma^{-1}(\pi_2^{-1}(V)))$$
is open, and $\delta$ is open because the following image is also open:
$$\delta(U)=\delta(\pi_1(\pi_1^{-1}(U)))=\pi_2(\gamma(\pi_1^{-1}(U))).$$
\end{proof}

\subsubsection{}

We next recall some useful facts concerning group rings. Let $G$ be a finite group. A $G$-module $M$ is said to be \emph{$G$-cohomologically-trivial}, or \emph{$G$-c.-t.}, if the Tate cohomology groups of $M$ vanish in all degrees and with respect to all subgroups of $G$.

\begin{lemma}\label{TMremark1}
(i) If $S$ is a Dedekind domain (or a field), then an $S[G]$-module is projective if and only if it is $S$-projective and $G$-c.-t.

\noindent{}(ii) If $T/S$ is an extension of fields and $M,N$ are finitely generated $S[G]$-modules, then
$M\cong N$ if and only if $T\otimes_S M\cong T\otimes_S N$ over $T[G]$. 

\noindent{}(iii) If $S$ is a discrete valuation ring with quotient field $T$ and $M,N$ are finitely generated, projective $S[G]$-modules, then
$M\cong N$ if and only if $T\otimes_S M\cong T\otimes_S N$ over $T[G]$.
\end{lemma}
\begin{proof} To prove claim (i) we note that a free $S[G]$-module $\bigoplus_{I}S[G]\cong\bigoplus_I(S\otimes_{\ZZ}\ZZ[G])\cong(\bigoplus_I S)\otimes_{\ZZ}\ZZ[G]$ is $G$-induced. Now, a projective $S[G]$-module is clearly $S$-projective and, since it is a direct summand in a free $S[G]$-module, it is relatively projective as a $G$-module (cf. \cite[Ch. VIII, \S 1]{SerreLocFie}) and thus $G$-c.-t. by \cite[Ch. IX, \S 3, Examples]{SerreLocFie}.

For the converse implication, we follow the proof of \cite[Prop. 4.1.b]{chinburg} to reduce it to the case \cite[Prop. 4.1.a]{chinburg} of fields. Alternatively, one could also deduce claim (i) from \cite[Prop. 4.1.b]{chinburg}.

We let $N$ be an $S[G]$-module that is $S$-projective and $G$-c.-t. By \cite[Ex. 14.2, p. 120]{SerreLinRep}, it suffices to fix an arbitrary maximal ideal $\p$ of $S$ and prove that the $(S/\p)[G]$-module $N/\p N$ is projective.
We fix $\p$ and a uniformiser $\pi$ such that $\p S_{(\p)}=\pi S_{(\p)}$ and set $P:=S_{(\p)}\otimes_S N$. Then $P$ is $S_{(\p)}$-projective (in fact, free) and, since Tate cohomology commutes with localisation, also $G$-c.-t. In addition,
$$N/\p N\cong (S/\p) \otimes_S N\cong(S_{(\p)}/\p S_{(\p)}) \otimes_S N\cong(S_{(\p)}/\pi S_{(\p)}) \otimes_{S_{(\p)}} P\cong P/\pi P.$$
Thus, the exactness of $0\to P\stackrel{\pi}{\to}P\to P/\pi P\to 0$ implies that $N/\p N$ is $G$-c.-t. Since $S/\p$ is a field, $N/\p N$ is $(S/\p)[G]$-projective by \cite[Prop. 4.1.a]{chinburg}, as required.

Claim (ii) is the Noether--Deuring Theorem (cf. \cite[Vol. I., \S 6, p. 139, Ex. 6]{curtisr} or \cite[(19.25)]{lam}) and claim (iii) is Swan's Theorem (cf. \cite[Vol. I, Thm. (32.1)]{curtisr}).
\end{proof}

\subsubsection{}
In this section we fix an arbitrary finite group $G$ and associate a characteristic class in the (topological) group $K_1(\F[G])$ to a natural family of $A[G]$-modules. 

For a finite $A[G]$-module $M$, we consider the finitely generated $A[G]$-module $$M':=A[G]\otimes_{\Fq[G]}M$$ with $t(h\otimes m)=th\otimes m$ for $h\in A[G]$ and $m\in M$. We then define $\tau_M\in{\rm End}_{A[G]}(M')$ by setting
$$\tau_M(h\otimes m):=(th\otimes m)-(h\otimes t\cdot m).$$
By \cite[Vol. II, (43.4)]{curtisr}, $\tau_M$ has the following useful property.
\begin{lemma}\label{43.4} Let $M$ be a finite $A[G]$-module. Then there is an exact sequence
$$0\to M'\stackrel{\tau_M}{\longrightarrow}M'\longrightarrow M\to 0.$$
In particular, the scalar extension $\tau_{M,\F}$ of $\tau_M$ belongs to ${\rm Aut}_{\F[G]}(M'_\F)$
\end{lemma}

\begin{definition}\label{ccdef}{\em Let $G$ be a finite group. Let $M$ be an $A[G]$-module that is both finite and $G$-cohomologically-trivial.
The \emph{characteristic class} of $M$ is the element $$c_G(M):=[M'_\F,\tau_{M,\F}]$$ of $K_1(\F[G])$. 
By abuse of notation, we identify $c_G(M)$ with its image $[M'_{\Fi},\tau_{M,\Fi}]$ in $K_1(\FiG)$.
}\end{definition}

\begin{remarks}\label{ccremarks}{\em \

\noindent{}(i) An $\Fq[G]$-module that is $G$-c.-t. is $\Fq[G]$-projective (by Lemma \ref{TMremark1}(i)). Thus, $M'_\F$ is $\F[G]$-projective.

\noindent{}(ii) If $G$ is abelian, the determinant of $\tau_M\in{\rm End}_{A[G]}(M')$ is a generator of the Fitting ideal of $M$ in $A[G]$. In fact, it is the `monic' generator in the sense of \cite[Rem. A.3.3]{fghp}.

\noindent{}(iii) See also Lemma \ref{Fitchar} below for a similar explicit relationship between characteristic classes and non-commutative Fitting ideals. 
}\end{remarks}

Characteristic classes are multiplicative, in the following sense.

\begin{lemma}\label{additivity} If $0\to L\to M\to N\to 0$ is a short exact sequence of $A[G]$-modules, two of which are finite and $G$-c.-t., then so is the third, and one has $c_G(M)=c_G(L)\cdot c_G(N)$.
In particular, characteristic classes are well-defined on isomorphism classes of $A[G]$-modules (that are finite and $G$-c.-t.), and commute with finite direct products.\end{lemma}
\begin{proof} It is clear that the third module is finite and $G$-c.-t. Since $A[G]=\Fq[G][t]$ is a flat $\Fq[G]$-module, we obtain a short exact sequence $0\to L'\xrightarrow{1\otimes\alpha}M'\xrightarrow{1\otimes\beta}N'\to 0$, where we have denoted by $\alpha,\beta$ the given $A[G]$-homomorphisms.

Now, for any $h\in A[G]$ and $l\in L$, one has $$(1\otimes\alpha)(\tau_L(h\otimes l))=(th\otimes\alpha(l))-(h\otimes\alpha(t\cdot l))=(th\otimes\alpha(l))-(h\otimes t\cdot\alpha(l))=\tau_M((1\otimes\alpha)(h\otimes l))$$
and, similarly, for any $m\in M$, one has $(1\otimes\beta)(\tau_M(h\otimes m))=\tau_N((1\otimes\beta)(h\otimes m))$. The equality $c_G(M)=c_G(L)\cdot c_G(N)$ thus follows directly from the definition of $K_1(\F[G])$ (cf. \cite[Vol. II, Def. (38.28)]{curtisr} or \cite[Ch. 13]{swan}).
\end{proof}

Characteristic classes also satisfy the following base change property.
\begin{lemma}\label{subgroupcc} Let $H$ be a subgroup of $G$ and let $N$ be an $A[H]$-module that is both finite and $H$-c.-t. Let $M$ be the $A[G]$-module $\Fq[G]\otimes_{\Fq[H]}N$, with $tg(x\otimes n)=g x\otimes t\cdot n$ for each $g\in G$, $x\in \Fq[G]$ and $n\in N$.

Then $M$ is both finite and $G$-c.-t. and in $K_1(\F[G])$ one has
$\iota_H(c_H(N))\,=\,c_G(M)$, where $\iota_H$ is the canonical homomorphism $K_1(\F[H])\to K_1(\F[G])$.
\end{lemma}
\begin{proof} It is clear that $M$ is finite and, applying Lemma \ref{TMremark1}(i) twice, we find that $N$ is $\Fq[H]$-projective, so $M$ is $\Fq[G]$-projective, so $M$ is $G$-c.-t.

We now consider the commutative diagram of $A[G]$-modules
\begin{equation*}
\begin{CD}
A[G]\otimes_{A[H]}(A[H]\otimes_{\Fq[H]}N) @> 1\otimes\tau_{N} >>  A[G]\otimes_{A[H]}(A[H]\otimes_{\Fq[H]}N) \\
@V \phi VV @V \phi VV \\
A[G]\otimes_{\Fq[H]}N @> \sigma >>  A[G]\otimes_{\Fq[H]}N \\
@A \psi AA @AA \psi A \\
A[G]\otimes_{\Fq[G]}M @> \tau_{M} >>  A[G]\otimes_{\Fq[G]}M.
\end{CD}\end{equation*}
Here $\phi$ and $\psi$ are the canonical isomorphisms and for $x\in A[G]$ and $n\in N$ we set
$$\sigma(x\otimes n)\,:=\,(tx\otimes n)-(x\otimes t\cdot n).$$
The commutativity of the diagram is straightforward to verify, and implies that
\begin{align*}\iota_H(c_H(N))=\iota_H([N'_\F,\tau_{N,\F}])
=&[(A[G]\otimes_{A[H]}N')_\F,(1\otimes\tau_{N})_\F]\\
=&[(A[G]\otimes_{\Fq[H]}N)_\F,\sigma_\F]
=[M'_\F,\tau_{M,\F}]
=c_G(M).
\end{align*}
\end{proof}

\subsection{Refined Euler characteristics}
Let $G$ be a finite group. In this section we define refined Euler characteristics in the relative algebraic $K_0$-group $K_0(A[G],\F_\infty[G])$ of the ring inclusion $A[G]\subset\FiG$ (cf. \cite{swan,curtisr}). Elements of this group are equivalence classes $[P,\phi,Q]$ of triples comprising finitely generated projective $A[G]$-modules $P,Q$ and an isomorphism of $\FiG$-modules $\phi:P_{\Fi}\cong Q_{\Fi}$.
We always use additive notation for $K_0$-groups.

We write $D^{\rm p}(A[G])$ for the full triangulated subcategory 
of the derived category $D(A[G])$
comprising complexes that are perfect. We recall that an object of $D(A[G])$ is said to be \emph{perfect} if it is isomorphic (in $D(A[G])$) to a bounded complex of finitely generated, projective modules.

\begin{definition}\label{Dpc}{\em We write $D^{\rm p,c}(A[G],\Fi)$ for the full subcategory of $D^{\rm p}(A[G])$ comprising complexes $C$ with the property that $H^n(C)_{\Fi}$ is a (finitely generated) projective $\FiG$-module for every $n\in\ZZ$. (Here `c' stands for cohomology.)
}\end{definition}

\begin{definition}\label{ECex}{\em Let $C$ be an object of $D^{\rm p,c}(A[G],\Fi)$ that is isomorphic (in $D(A[G])$) to a complex of the form
\begin{equation}\label{representative}P^1\stackrel{d}{\longrightarrow} P^2\end{equation} with the first term in degree one, where both $P^1$ and $P^2$ are finitely generated, projective $A[G]$-modules. Assume also given an isomorphism of $\FiG$-modules $$\lambda\,:\,H^{2}(C)_{\Fi}\,\longrightarrow\,H^{1}(C)_{\Fi}.$$ 

The \emph{refined Euler characteristic} of the pair $(C,\lambda)$ is the element
$$\chi_{G}(C,\lambda)\,:=\,-\left[P^1,\gamma,P^2\right]$$ of $K_0\bigl(A[G],\FiG\bigr)$, where $\gamma$ is the composite isomorphism of $\FiG$-modules
$$P^1_{\Fi}\stackrel{\iota_1}{\longrightarrow}\im(d_{\Fi})\oplus H^1(C)_{\Fi}\xrightarrow{{\rm id}\oplus\lambda^{-1}}\im(d_{\Fi})\oplus H^2(C)_{\Fi}\xrightarrow{\iota_2^{-1}}P^2_{\Fi},$$
for an arbitrary choice of splittings $\iota_1$ and $\iota_2$ of the respective short exact sequences
\begin{align*}0\to H^1(C)_{\Fi}\longrightarrow &P^1_{\Fi}\xrightarrow{d_{\Fi}}\im(d_{\Fi})\to 0,\\
0\to \im(d_{\Fi})\longrightarrow &P^2_{\Fi}\longrightarrow H^2(C)_{\Fi}\to 0.\end{align*}
We note that Definition \ref{Dpc} ensures that the second sequence splits and combines with Lemma \ref{TMremark1}(i) to imply that $\im(d_{\Fi})$ is projective, so the first sequence splits. Also, strictly speaking, we have denoted $\iota_1:P^1_{\Fi}\to\im(d_{\Fi})\oplus H^1(C)_{\Fi}$ and $\iota_2:P^2_{\Fi}\to\im(d_{\Fi})\oplus H^2(C)_{\Fi}$ the isomorphisms induced by arbitrary choices of such splittings.
}\end{definition}

\begin{remark}{\em It is straightforward to prove directly (or to deduce from general arguments due to Burns \cite{ewt}) that $\chi_{G}(C,\lambda)$ does not depend on the choice of representative of $C$ of the form (\ref{representative}) or on the choice of splittings $\iota_1,\iota_2$.
}\end{remark}

\begin{remark}\label{ECRemark}{\em  
After restricting consideration to the subcategory $D^{\rm p,c}(A[G],\Fi)$ of $D^{\rm p}(A[G])$, one may mimic the use of Deligne's universal determinant functors \cite{delignedet} that is made by Breuning and Burns in \cite[Def. 5.5]{Breuning Burns} (even though $\FiG$ is not, in general, a regular ring).

In this way, to any pair $(C,\lambda)$, with $C$ in $D^{\rm p,c}(A[G],\Fi)$ and with 
$$\lambda\,:\,H^{\rm ev}(C)_{\Fi}\,\xrightarrow{\sim}\,H^{\rm od}(C)_{\Fi}$$
an isomorphism of $\FiG$-modules,
one associates a refined Euler characteristic in $K_0\bigl(A[G],\Fi[G]\bigr)$.
Here we have written
$H^{\rm ev}(C)$ and $H^{\rm od}(C)$ for the direct sum of the even and of the odd degree cohomology of $C$, respectively. This more general construction specialises to recover Definition \ref{ECex}.
}\end{remark}

\subsection{Statement of the refined class number formula}\label{Statement}

In this section we state our main result. Let $k$ be a finite separable extension of $\F=\Fq(t)$ and let $K$ be a finite Galois extension of $k$ with group $G:={\rm Gal}(K/k)$.

We endow $\Fi$ with the $t^{-1}$-adic absolute value, and $\FiG=\prod_{g\in G}(\Fi g)$ with the supremum norm. Using the ring isomorphism $\FiG\cong\Fq[G](\!(t^{-1})\!)$ to define a function $V:(\FiG\setminus\{0\})\to\ZZ$ by $V(\sum_{i\geq N}r_it^{-i}):=N$ for $r_i\in\Fq[G]$ with $r_N\neq 0$, one finds that the norm of $0\neq f=\sum_{g\in G}f_gg\in\FiG$ with $f_g\in\Fi$ is computed as
$${\max}_{g\in G}\{|f_g|\}=
{\max}_{g\in G}\{q^{-v_\infty(f_g)}\}=q^{-\min_{g\in G}\{v_\infty(f_g)\}}=q^{-V(f)}.$$
One easily verifies that this is a (sub-multiplicative) ring norm. The ring topology it induces is, in fact, the unique topology making $\FiG$ a Hausdorff $\Fi$-vector space.

We set $K_\infty:=\F_\infty\otimes_\F K$, a free topological $\Fi[G]$-module of rank $[k:\F]$. 
We say that a finitely generated $A[G]$-submodule $M$ of $K_\infty$ is an `$A[G]$-lattice in $K_\infty$' if the canonical map $M_{\Fi}\to K_\infty$ is bijective. An $A[G]$-submodule of $K_\infty$ that is $A$-free of finite rank is an $A[G]$-lattice if and only if it is discrete and co-compact in $K_\infty$ (cf. \cite[Prop. 1.6]{debry}).

\begin{remark}\label{TMremark2}{\em Claims (ii) and (iii) of Lemma \ref{TMremark1} combine to imply that any projective $A[G]$-lattice $M$ in $K_\infty$ is locally-free of constant local rank $n:=[k:\F]$, i.e., that $A_{(P)}\otimes_A M\cong A_{(P)}[G]^{n}$ for every prime $P$ of $A$.}
\end{remark}

\subsubsection{} For any commutative $\Fq$-algebra $\co$, we denote by $\tau$ the $q$-th power Frobenius of $\co$, that is, the endomorphism of $\co$ given by $\tau(x)=x^q$. The twisted polynomial ring $\co\{\tau\}$, with commutation relation $\tau\cdot x=x^q\cdot\tau$, acts on any commutative $\co$-algebra.

Let $E$ be a Drinfeld module defined over the integral closure $\co_k$ of $A$ in $k$. We interpret $E$ as a functor, from the category of $\co_k\{\tau\}[G]$-modules to the category of $A[G]$-modules, induced by a given $\Fq$-algebra morphism $\phi_E:A\to\co_{k}\{\tau\}$. (In this article, we only ever consider Drinfeld $A$-modules for $A:=\Fq[t]$ and defined over $\co_k$.)

The exponential power series of $E$ defines a homomorphism of $A[G]$-modules
$${\rm exp}_E:K_\infty\to E(K_\infty).$$
The main result of Taelman in \cite{TMA} implies both that $\expo$ is an $A[G]$-lattice in $K_\infty$ and that the $A[G]$-module
$$H(E/\co_K):=\frac{E(K_\infty)}{E(\co_K)+{\rm exp}_E(K_\infty)}$$
is finite.
The former assertion is an analogue of Dirichlet's theorem for the preimage under the archimedean exponential of the units of a number field. The analogy between $H(E/\co_K)$ and class groups is more elaborate, but see Remark 6 in loc. cit.

\subsubsection{}
We now use the notion of `taming module' that has been introduced by Ferrara, Green, Higgins and Popescu in \cite[Def. 1.3.2]{fghp}.
\begin{definition}\label{defTM}{\em A \emph{taming module} for $K/k$ is an $\co_k\{\tau\}[G]$-submodule $\cm$ of $\co_K$ that is $\co_k[G]$-projective and such that the quotient $\co_K/\cm$ is finite and only supported at prime ideals of $\co_k$ which are wildly ramified in $K/k$.
}\end{definition}

\begin{remark}\label{TMremark3}{\em Taming modules $\cm$ for $K/k$ exist and, in fact, are $\co_k[G]$-locally-free, of constant local rank 1. These claims follow from Prop. A.2.4 in loc. cit., upon replacing the use of Cor. A.1.7(3) by Swan's Theorem, as stated in Lemma \ref{TMremark1}(iii). Indeed, set $\mathcal{R}:=\co_k$ and $\mathcal{S}:=\co_K$. Then, if $v\in\Spec(\mathcal{R})$ is tamely ramified in $K/k$, the fact that $\mathcal{S}_{(v)}\cong\mathcal{R}_{(v)}[G]$, that is used in loc. cit., follows by Swan's Theorem from the isomorphism of $k[G]$-modules $k\otimes_{\mathcal{R}_{(v)}}\mathcal{S}_{(v)}\cong K\cong k[G]\cong k\otimes_{\mathcal{R}_{(v)}}\mathcal{R}_{(v)}[G]$.

In particular, for any taming module $\cm$ for $K/k$ and for every $\p\in\Spec(\co_k)$, the finite $\Fq[G]$-module $\cm/\p\cm$ is free, thus also $G$-c.-t. (by Lemma \ref{TMremark1}(i)), and the characteristic classes $c_G(\cm/\p\cm)$ and $c_G(E(\cm/\p\cm))$ are well-defined.
}\end{remark}

\begin{definition}{\em We fix a taming module $\cm$ for $K/k$.\

\noindent{}(i) The \emph{$\cm$-modified complex of units} of $(E,K/k)$ is the complex $C^{\cm}_{E,K/k}$ of $A[G]$-modules
$$K_\infty\,\,\xrightarrow{({\rm exp}_E,0)}\,\,\left(E(K_\infty)/E(\cm)\right)\,\oplus\,\cm$$ with the first term placed in degree one. 

\noindent{}(ii) The \emph{$\cm$-modified Taelman class group}  of $(E,K/k)$ is the $A[G]$-module
$$H(E/\cm):=\frac{E(K_\infty)}{E(\cm)+{\rm exp}_E(K_\infty)}.$$}
\end{definition}

\begin{remark}\label{regRK}{\em The term `complex of units' is taken from \cite{mornev}. It is clear that
$$H^j\left(C^{\cm}_{E,K/k}\right)=\begin{cases}\expm,\,\,\,\,\,\,\,\,\,\,\,\,\,\,\,\,\,\,\,\,\,\,\,\,\,\,\,\,\,\,\,\,\,\,\,\,\,\,\,\,j=1,\\ H(E/\cm)\oplus\cm,\,\,\,\,\,\,\,\,\,\,\,\,\,\,\,\,\,\,\,\,\,\,\,\,\,\,\,\,\,\,\,\,\,\,\,\,\,j=2,\\ 0,\,\,\,\,\,\,\,\,\,\,\,\,\,\,\,\,\,\,\,\,\,\,\,\,\,\,\,\,\,\,\,\,\,\,\,\,\,\,\,\,\,\,\,\,\,\,\,\,\,\,\,\,\,\,\,\,\,\,\,\,\,\,\,\,\,\,\,\,\,\,\,\,\,j\neq 1,2,\end{cases}$$
that $\expm$ is an $A[G]$-lattice in $K_\infty$ and that $H(E/\cm)$ 
is finite. 
We denote by $\lambda^\cm_{E,K/k}$ the isomorphism
\begin{equation*}\label{reg}H^1(C^{\cm}_{E,K/k})_{\Fi}=\expm_{\Fi}= K_\infty\cong\,(\co_K)_{\Fi}=\cm_{\Fi}= H^2(C^{\cm}_{E,K/k})_{\Fi},\end{equation*}
and by $\lambda^{\cm,-1}_{E,K/k}$ its inverse (which is thus induced by the inclusions $\cm\subseteq\co_K\subset K\subset K_\infty$, as claimed in Theorem \ref{thm:mainA}(3))}.
\end{remark}

\subsubsection{}
We may now state the main result of this article.
We recall that, since $\FiG$ is naturally a semilocal topological ring, 
$K_1(\FiG)$ is a topological abelian group, and we also use the canonical connecting homomorphism
\begin{equation*}\label{partial}\partial_G\,:\,K_1(\FiG)\to K_0(A[G],\FiG).\end{equation*}

\begin{theorem}\label{MT} Let $K$ be a finite Galois extension of $k$ and let $E$ be a Drinfeld module over $\co_k$. Let $\cm$ be a taming module for $K/k$.

Then the complex $C^{\cm}_{E,K/k}$ belongs to the category $D^{\rm p,c}(A[G],\Fi)$, the infinite product
$$\Theta^\cm_{E,K/k}:=\prod_{\p\in\Spec(\co_k)}\left(c_G\bigl(\cm/\p\cm\bigr)\cdot c_G\bigl(E(\cm/\p\cm)\bigr)^{-1}\right)$$
converges to a unique element of $K_1(\Fi[G])$, 
and in $K_0(A[G],\Fi[G])$ one has
\begin{equation}\label{Mequality}\partial_G\bigl(\Theta^\cm_{E,K/k}\bigr)\,=\,-\chi_{G}\bigl(C^\cm_{E,K/k},\lambda^{\cm,-1}_{E,K/k}\bigr).\end{equation}
\end{theorem}

The proof of Theorem \ref{MT} is given in \S\ref{proofMT}.

\begin{remark}\label{convergencerk}{\em
By convergence of a countable, ordered infinite product in the topological abelian group $K_1(\FiG)$ we mean that the sequence of partial products converges. In Corollary \ref{tfcor} (see also Corollary \ref{convergencecor}(ii)) we prove that, for any given order on the countable set $\Spec(\co_k)$, the infinite product in Theorem \ref{MT} converges to a unique element $\Theta^\cm_{E,K/k}$ of $K_1(\Fi[G])$ and, moreover, that this limit element $\Theta^\cm_{E,K/k}$ is independent of the given order on $\Spec(\co_k)$.
}
\end{remark}

\begin{remark}\label{impartial}{\em
The map $\partial_G$ is defined by the equality $\partial_G([\Fi[G]^n, \phi]):=[A[G]^n,\phi,A[G]^n]$ for any $n\geq 1$ and any automorphism $\phi$ of $\FiG^n$, and fits in the exact localisation sequence
$$K_1(A[G])\longrightarrow K_1(\Fi[G])\stackrel{\partial_G}{\longrightarrow}K_0(A[G],\Fi[G]).$$
Although $\partial_G$ need not be surjective, the equality (\ref{Mequality}) takes place in the subgroup
$$\im(\partial_G)\,\cong\,K_1(\Fi[G])/\im(K_1(A[G])).$$

Now, if $G$ is abelian, then taking determinants over $\FiG$ induces an isomorphism
$$\im(\partial_G)\,\cong\,\Fi[G]^\times/A[G]^\times$$ and \cite[Prop. A.3.4]{fghp} shows that there is a decomposition
$$\Fi[G]^\times\cong\Fi[G]^+\times A[G]^\times.$$
Here $\Fi[G]^+$ is the subgroup of `monic elements' of $\FiG^\times$ in which their `equivariant Tamagawa number formula for Drinfeld modules' (Thm. 1.5.1 in loc. cit.) takes place.

Our proof of Theorem \ref{MT} makes clear that, for abelian extensions $K/k$ and under this identification of $\im(\partial_G)$ with $\Fi[G]^+$, the equality (\ref{Mequality}) specialises to recover the equivariant Tamagawa number formula of loc. cit.; see Remark \ref{forearlierref} below for details. In particular, (\ref{Mequality}) also specialises to recover Taelman's class number formula \cite[Thm. 1]{TAM} in the case $K=k$.
}\end{remark}

\begin{remark}\label{624}{\em It is straightforward to show that the validity of the equality (\ref{Mequality}) is \textit{a priori} independent of the choice of taming module $\cm$. As further discussed in \cite[Rk. 6.2.4]{fghp}, one should think of the above $\cm$-modifications as analogous to the usual $T$-modifications in the study of equivariant special Artin $L$-values for global fields.}\end{remark}

\begin{remark}\label{formofreprk}{\em We actually prove that the complex $C^{\cm}_{E,K/k}$ belongs to the category $D^{\rm p,c}(A[G],\Fi)$ and has a representative of the form (\ref{representative}). For $K=k$, this fact specialises to recover the main result of \cite{TMA}.}\end{remark}

\begin{remark}{\em The statement of Theorem \ref{MT} generalises in the obvious manner to abelian $t$-modules $E$ defined over $\co_k$, and our proof may be extended to this general setting. To avoid obscuring the main ideas of the present article, the details of this generalisation of Theorem \ref{MT} will be presented in a separate note. For now we recall that, under the assumption that the Galois group $G$ is abelian, Green and Popescu have recently proved an equivariant Tamagawa number formula for abelian $t$-modules \cite[Thm. 1.32]{greenpopescu} and Beaumont has proved an equivariant class formula for $z$-deformations of $t$-modules \cite[Thm. 1.2]{beaumont}.
}\end{remark}

\begin{example}\label{Carlitz}{\em We consider the Carlitz module $E=C$, defined by the morphism $\phi_C:A\to \co_k\{\tau\}$ given by $\phi_C(t):=t+\tau$. Given a prime $\p\in\Spec(\co_k)$ that is tamely ramified in $K/k$, we show that the Euler factor
$$c_G\bigl(\cm/\p\cm\bigr)\cdot c_G\bigl(C(\cm/\p\cm)\bigr)^{-1}=c_G\bigl(\co_K/\p\co_K\bigr)\cdot c_G\bigl(C(\co_K/\p\co_K)\bigr)^{-1}$$
of $\Theta^\cm_{C,K/k}$ at $\p$ is equal in $K_1(\F[G])$ to $[1-(N\p)^{-1}\Fr_\p\cdot e_{I_\p}]^{-1}$.
Here $I_\p$ and $\Fr_\p$ denote the restriction to $G$ of an inertia group and Frobenius automorphism for $\p$, respectively; we use the idempotent $e_{I_\p}:=|I_\p|^{-1}\sum_{g\in I_\p}g$ of $A[G]$; and the monic element $N\p$ of $A$ is defined by the equality $\bar\p^{f_\p}=N\p\cdot A$, where $\bar\p$ denotes the restriction of $\p$ and $f_\p:=[\co_k/\p:A/\bar\p]$. A corresponding interpretation of Euler factors of $\Theta^\cm_{E,K/k}$ for general Drinfeld modules $E$ will be proved in forthcoming work.

We regard the tamely ramified prime $\p$ as fixed and first claim that
\begin{equation}\label{2requiredEF1}c_G\left(\co_K/\p\co_K\right)=[N\p].
\end{equation}
We set $\kappa_{\bar\p}:=A/\bar\p$ and write $P$ for the monic generator of $\bar\p$. We fix a prime $\q$ of $K$ above $\p$, with completion $K_\q\supset\co_{K_\q}$ and decomposition subgroup $G_\p$ in $G$. Then 
\begin{align*}&\co_K/\p\co_K\cong\Fq[G]\otimes_{\Fq[G_\p]}\co_{K_\q}/\p\co_{K_\q}\cong\Fq[G]\otimes_{\Fq[G_\p]}(\co_{k}/\p)[G_\p]\\
\cong\: &\Fq[G]\otimes_{\Fq[G_\p]}(\kappa_{\bar\p}^{f_\p})[G_\p]\cong\Fq[G]\otimes_{\Fq[G_\p]}(\kappa_{\bar\p}[G_\p])^{f_\p}\cong(\Fq[G]\otimes_{\Fq[G_\p]}\kappa_{\bar\p}[G_\p])^{f_\p}.\end{align*} 
Thus Lemmas \ref{additivity} and \ref{subgroupcc} imply that (\ref{2requiredEF1}) is valid if
$c_{G_\p}{\left(\kappa_{\bar\p}[G_\p]\right)}=[P]$ in $K_1(\F[G_\p])$.

By definition, the left-hand side of the latter equality is the class of the action of $(t\otimes 1)-(1\otimes t)$ on $\F[G_\p]\otimes_{\Fq[G_\p]}\kappa_{\bar\p}[G_\p]$. To compute this action we write $d$ for the degree of $P={\sum}_{i=0}^{i=d}a_it^{i}$ and use the $A[G_\p]$-basis $\{1\otimes t^{i}:0\leq i\leq d-1\}$ of $A[G_\p]\otimes_{\Fq[G_\p]}\kappa_{\bar\p}[G_\p]$. Then the matrix of $(t\otimes 1)-(1\otimes t)$ is given by
\begin{equation*}\label{2horriblematrix}M_\tau:=
\begin{bmatrix}
t & -1 & 0 & 0 & \cdots & 0 & 0\\
0 & t & -1 & 0 & \cdots & 0 & 0\\
0 & 0 & t & -1 & \cdots & 0 & 0\\
0 & 0 & 0 & t & \ddots & 0 & 0\\
\vdots                & \vdots & \vdots & \ddots & \ddots & \vdots & \vdots\\
0 & 0 & 0 & 0 & \cdots & t & -1\\
a_0 & a_1 & a_2 & a_3 & \cdots & a_{d-2} & (a_{d-1}+t)
\end{bmatrix}.\end{equation*}
But it is straightforward to verify that $M_\tau$ is equal, up to multiplication by elementary matrices in $M_d(\F)$, to the diagonal matrix $D_P:={\rm diag}(1,1,\ldots,1,P)$. Thus $c_{G_\p}\left(\kappa_{\bar\p}[G_\p]\right)=[M_\tau]=[D_P]=[P]$, as required to prove (\ref{2requiredEF1}).

We next claim that 
\begin{equation}\label{2requiredEF2}c_G\bigl(C(\co_K/\p\co_K)\bigr)=[N\p-\Fr_\p\cdot e_{I_\p}].\end{equation} We assume, as we may, that $I_\p\subseteq G_\p$ and $\Fr_\p\in G_\p$, and we set $e:=|I_\p|$, $M:=K^{I_\p}$,
$L:=K^{G_\p}$, $J_\p:=G_\p/I_\p$, $\q_M:=\q\cap\co_M$, $\q_L:=\q\cap\co_L$ and $e'_{I_\p}=1-e_{I_\p}$. We note that $\co_{K_\q}\cong\co_{L_{\q_L}}[G_\p]$ is $G_\p$-c.t. by Lemma \ref{TMremark1}(i), that $\q^{a}\co_{K_\q}\cong\co_{K_\q}$ for any $a\geq 0$, and thus that $\q^{a}/\q^{b}\cong\q^{a}\co_{K_\q}/\q^{b}\co_{K_\q}$ is $G_\p$-c.-t. for any $b\geq a\geq 0$.

Since $C(\co_K/\p\co_K)\cong\Fq[G]\otimes_{\Fq[G_\p]}C(\co_{K}/\q^{e})$ and $$N\p-\Fr_\p\cdot e_{I_\p}=(e_{I_\p}(N\p-\Fr_\p)+e'_{I_\p})\cdot(e_{I_\p}+e'_{I_\p}N\p),$$ to deduce (\ref{2requiredEF2}) from Lemmas \ref{additivity} and \ref{subgroupcc}, it is enough to show that
\begin{equation}\label{2requiredEFGP}c_{G_\p}\bigl(C(\co_K/\q)\bigr)=[e_{I_\p}(N\p-\Fr_\p)+e'_{I_\p}]\,\,\,\,\,\,\text{ and }\,\,\,\,\,\,c_{G_\p}\bigl(C(\q/\q^{e})\bigr)=[e_{I_\p}+e'_{I_\p}N\p].
\end{equation}

To consider the first equality we note that $J_\p$ is abelian (in fact, cyclic) and apply \cite[Prop. A.5.1(3)]{fghp} to the extension $M/L$ and the unramified prime $\q_L$. Together with Prop. A.4.1 in loc. cit. (or with Lemma \ref{Fitchar} below), this gives $\det_{\F[J_\p]}(c_{J_\p}(C(\co_M/\q_L\co_M)))=N\p-\Fr_{\p}I_\p$ in $\F[J_\p]^\times$. We abbreviate $\tau_M$ for $M=C(\co_K/\q)$ to $\tau$. Then, since $I_\p$ acts trivially on
$$\co_K/\q\cong\co_M/\q_M=\co_M/\q_L\co_M,$$
it follows that $c_{G_\p}(C(\co_K/\q))=[C(\co_K/\q)'_\F,\tau_\F]$ is mapped to
\begin{align*}&(\det_{\F[J_\p]}([\F[J_\p]\otimes_{\F[G_\p]}C(\co_K/\q)'_\F,1\otimes\tau_\F]),1)\\ 
=\: &(\det_{\F[J_\p]}([\F[J_\p]\otimes_{\Fq[J_\p]}C(\co_M/\q_L\co_M),\tau_\F]),1)=(N\p-\Fr_{\p}I_\p,1)
\end{align*}
under the composite isomorphism
\begin{align*}K_1(\F[G_\p])\cong K_1(e_{I_\p}\F[G_\p])\times K_1(e'_{I_\p}\F[G_\p])\cong K_1(\F[J_\p])\times & K_1(e'_{I_\p}\F[G_\p])\\ \stackrel{\det}{\cong}\F[J_\p]^\times\times & K_1(e'_{I_\p}\F[G_\p]).\end{align*}
Since the same is true of $[e_{I_\p}(N\p-\Fr_\p)+e'_{I_\p}]$, the first equality in (\ref{2requiredEFGP}) is valid.

To consider the second equality in (\ref{2requiredEFGP}), we note that for all $n\geq 1$, $x\in\q^n$ and $a\in A$, $(\phi_C(a))(x)$ is congruent to $ax$ modulo $\q^{qn}$. In particular, $C(\q^n/\q^{n+1})\cong\q^n/\q^{n+1}$ so, by induction on $e$, one can deduce from Lemma \ref{additivity} that $c_{G_\p}(C(\q/\q^{e}))=c_{G_\p}(\q/\q^{e})$.

To compute $c_{G_\p}(\q/\q^{e})$ we note that $\co_K/\q^{e}=\co_K/\q_L\co_K\cong(\co_L/\q_L)[G_\p]$ and that $e_{I_\p}(\co_K/\q^{e})\cong(\co_L/\q_L)[J_\p]$ surjects onto $e_{I_\p}(\co_K/\q)=\co_K/\q$. Since the latter two $(\co_L/\q_L)$-spaces have dimension $|J_\p|$, in fact we have
$$\co_K/\q\cong e_{I_\p}(\co_L/\q_L)[G_\p]=e_{I_\p}(\co_k/\p)[G_\p]\cong(e_{I_\p}\kappa_{\bar\p}[G_\p])^{f_\p}.$$
Now, (\ref{2requiredEF1}) applied to the extension $K/L$ gives $c_{G_\p}(\co_K/\q^{e})=c_{G_\p}(\co_K/\q_L\co_K)=[N\p]$. Set $\kappa:=\kappa_{\bar\p}[G_\p]$ and $f:=f_\p$. Writing $I_d$ for the identity matrix, the argument used to prove (\ref{2requiredEF1}) shows that
\begin{align*}c_{G_\p}(e_{I_\p}\kappa)=\:&[(e_{I_\p}\kappa)'_\F,\tau_{e_{I_\p}\kappa,\F}]=[(e_{I_\p}\kappa)'_\F,\tau_{e_{I_\p}\kappa,\F}]\cdot[(e'_{I_\p}\kappa)'_\F,1]\\ 
=\:&[\kappa'_\F,e_{I_\p}\tau_{\kappa,\F}\oplus e'_{I_\p}{\rm id}_{\kappa'_\F}]=[e_{I_\p}M_\tau+e'_{I_\p}I_d],\end{align*}
and $e_{I_\p}M_\tau+e'_{I_\p}I_d$ is equal, up to multiplication by elementary matrices in $M_d(\F[G_\p])$, to the diagonal matrix $D'_P:={\rm diag}(1,1,\ldots,1,e_{I_\p}P+e'_{I_\p})$. Using Lemma \ref{additivity} we have $$c_{G_\p}(\co_K/\q)=c_{G_\p}(e_{I_\p}\kappa)^{f}=[e_{I_\p}M_\tau+e'_{I_\p}I_d]^{f}=[D'_P]^{f}=[e_{I_\p}P+e'_{I_\p}]^{f}=[e_{I_\p}N\p+e'_{I_\p}].$$ The second equality in (\ref{2requiredEFGP}) now follows from Lemma \ref{additivity} as
\begin{align*}c_{G_\p}(C(\q/\q^{e})) = c_{G_\p}(\q/\q^{e}) = c_{G_\p}(\co_K/\q^{e})\cdot c_{G_\p}(\co_K/\q)^{-1} =\: &[N\p]\cdot[e_{I_\p}N\p+e'_{I_\p}]^{-1}\\ =\:&[e_{I_\p}+e'_{I_\p}N\p].\end{align*}

Finally, (\ref{2requiredEF1}) and (\ref{2requiredEF2}) combine to give our claim
$$c_G\bigl(\co_K/\p\co_K\bigr)\cdot c_G\bigl(C(\co_K/\p\co_K)\bigr)^{-1}=[N\p]\cdot [N\p-\Fr_\p\cdot e_{I_\p}]^{-1}=[1-(N\p)^{-1}\Fr_\p\cdot e_{I_\p}]^{-1}.$$
}\end{example}

\section{A construction of non-abelian Stickelberger elements}\label{3}

In this section, we restrict attention to finite Galois extensions $K/k$ with the property that the corresponding group rings 
decompose as finite direct sums of matrix rings over commutative rings. In this setting, we will construct non-abelian Stickelberger elements for Drinfeld modules and state their relations to the Galois module structure of Taelman class groups, that are encoded within Theorem \ref{MT}. 

In particular, we obtain generalisations of the `refined Brumer-Stark Theorem for Drinfeld modules' \cite[Thm. 1.5.5]{fghp} of Ferrara, Green, Higgins and Popescu, and of Theorem A of Angl\`es and Taelman in \cite{at}.

In the general case, even defining suitable notions of reduced determinants and of non-commutative Fitting ideals will require new and rather involved algebraic techniques. We will return to these problems in future work.

\subsection{\texorpdfstring{Reduced determinants and non-commutative Fitting ideals}{Reduced determinants and non-commutative Fitting ideals}}

We fix a finite group $G$ with the property that $\ell$ does not divide the order of the commutator subgroup $G'$ of $G$. Equivalently, $G$ has an abelian $\ell$-Sylow subgroup $L$ and a normal $\ell$-complement $N$, and thus admits a semidirect product decomposition
\begin{equation}\label{semidirect}G\cong N\rtimes L.\end{equation}
(Indeed, if $\ell\nmid|G'|$ then any $\ell$-Sylow subgroup $L$ is isomorphic to the direct factor $LG'/G'$ of the abelian group $G/G'$, and the $\ell$-complement $\ker(G\to G/G'\to LG'/G')$ is normal. Conversely, $G'$ must be contained in $N$, so $\ell\nmid|G'|$.)

\subsubsection{}\label{311} If $F$ is an algebraically closed field of characteristic $\ell$, then the proof of \cite[Part 2, \S 6, Lem. 1.10, pp. 231-233]{Passman} gives a canonical decomposition of $F[G]$ as a finite direct sum of matrix rings over commutative $F$-algebras. DeMeyer and Janusz \cite[Thm. 3.1]{dMJ} then use this fact to prove that for any field $F$ of characteristic $\ell$, $F[G]$ is an Azumaya algebra. However, their proof is non-constructive, in that it does not provide information on how to decompose $F[G]$ as a sum of matrix rings.

In this section, we explain how to extend the argument of \cite[Part 2, \S 6, Lem. 1.10, pp. 231-233]{Passman} to give a corresponding canonical decomposition of $\Fl[G]$.

We fix a decomposition (\ref{semidirect}) of $G$. Then $G$ acts on the set of central primitive idempotents of $\Fl[N]$ by conjugation. Let $\{\co_i:1\leq i\leq m\}$ be the set of orbits of this action. For each $i$, fix an element $e_i\in\co_i$. Then the Wedderburn decomposition of $e_i\Fl[N]$ is of the form 
\begin{equation}\label{eiWD}e_i\Fl[N]\,\cong\, M_{m_i}(\Fl(i))\end{equation}
for some $m_i\in\NN$ and some finite field extension $\Fl(i)/\Fl$ (note that $\Fl(i)$ is indeed a field by Wedderburn's little theorem). We also define an abelian $\ell$-group
$$L_i:=\{\lambda\in L:\lambda e_i \lambda^{-1}=e_i\},$$
the centraliser of $e_i$ in $L$.

\begin{prop}\label{NEWfiniteproducts} There is a canonical isomorphism
$$\Fl[G]\,\cong\,{\bigoplus}_{i=1}^{i=m}M_{m_i\cdot|\co_i|}(\Fl(i)[L_i]).$$
\end{prop}
\begin{remark}{\em Given an index $1\leq i\leq m$, different choices of idempotent $e_i\in\co_i$ give rise to canonically isomorphic summands $M_{m_i\cdot|\co_i|}(\Fl(i)[L_i])$. The summand corresponding to the orbit $\{e_N\}$ for $e_N:=|N|^{-1}\sum_{\nu\in N}\nu$ is $M_{1}(\Fl[L])=\Fl[L]$.
If $\ell\nmid|G|$, Proposition \ref{NEWfiniteproducts} recovers the Wedderburn decomposition $\Fl[G]=\Fl[N]\cong{\bigoplus}_{i=1}^{i=m}M_{m_i}(\Fl(i))$. If $G$ is abelian, it simply gives $\Fl[G]=\Fl[N][L]\cong{\bigoplus}_{i=1}^{i=m}\Fl(i)[L]$.}
\end{remark}

In the rest of section \S \ref{311} we prove Proposition \ref{NEWfiniteproducts}.
For each $1\leq i\leq m$ we set $f_i:=\sum_{e\in\co_i}e\in\Fl[N]$ and observe that this idempotent is central in $\Fl[G]$. Since $\sum_{i=1}^{i=m}f_i=1$, we have
\begin{equation}\label{fdecomp}\Fl[G]\,=\,{\bigoplus}_{i=1}^{i=m}f_i\Fl[G].\end{equation}

We next fix $i$ and let
$$G_i:=\{g\in G:ge_ig^{-1}=e_i\}\,\supseteq\, N$$
be the centraliser of $e_i$ in $G$.
We then apply \cite[Part 2, \S 6, Lem. 1.7, pp. 228,229]{Passman} to $G$, with $H$ taken to be $N$, $K$ taken to be $\Fl$, $e_1$ taken to be $e_i$ and $e$ taken to be $f_i$, to obtain the following result.

\begin{lemma}\label{17} There is a canonical isomorphism
$$M_{|\co_i|}(e_i\Fl[G_i])\,\cong\, f_i\Fl[G]$$
that maps a matrix $(s_{kj})_{1\leq k,j\leq |\co_i|}$ to $\sum_{1\leq k,j\leq |\co_i|}s_{kj}(g_k^{-1}e_ig_j)$, where the elements $g_l\in G$ are chosen so that $g_1=1$ and $$\co_i=\{g_l^{-1}e_ig_l:1\leq l\leq|\co_i|\}.$$
\end{lemma}

The proof of Proposition \ref{NEWfiniteproducts} is now completed upon combining (\ref{fdecomp}) with Lemma \ref{17} and with the following result.

\begin{lemma}\label{18} The isomorphism (\ref{eiWD}) induces a canonical isomorphism (see (\ref{15}))
$$e_i\Fl[G_i]\,\cong\, M_{m_i}(\Fl(i)[L_i]).$$
\end{lemma}
\begin{proof} Because $e_i$ is a central idempotent of $\Fl[G_i]$, $R:=e_i\Fl[G_i]$ is a ring with identity element $e_i$ and with
$$R\supseteq e_i\Fl[N]\stackrel{(\ref{eiWD})}{\cong} M_{m_i}(\Fl(i)).$$
Let $E_{kj}$ be the matrix with 1 in position $(k,j)$ and zeros elsewhere, and let $\{e_{kj}\}_{1\leq k,j\leq m_i}\subseteq e_i\Fl[N]$ be their preimages through (\ref{eiWD}). We note that $e_{11}e_i=e_{11}$, and hence that
$e_{11}Re_{11}=e_{11}(e_i\Fl[G_i])e_{11}=e_{11}\Fl[G_i]e_{11}$.
Because $$e_{kj}e_{ab}=\begin{cases}0,\,\,\,\,\,\,\,\,\,\,\,j\neq a,\\ e_{kb},\,\,\,\,\,\,j=a,\end{cases}$$ and $e_i=e_{11}+\ldots+e_{m_i m_i}$, \cite[Part 2, \S 6, Lem. 1.5, p. 227]{Passman} shows that the map
\begin{align}\label{15}M_{m_i}(e_{11}\Fl[G_i]e_{11})=M_{m_i}(e_{11}Re_{11})&\cong R=e_i\Fl[G_i],\\
\notag(s_{kj})_{1\leq k,j\leq m_i}&\mapsto{\sum}_{1\leq k,j\leq m_i}s_{kj}e_{kj},\end{align} is a ring isomorphism.
Therefore, it suffices to show that $e_{11}\Fl[G_i]e_{11}\cong\Fl(i)[L_i]$.

Next we observe that the map $L_i=G_i\cap L\to G_i\to G_i/N$ is bijective. Then, the argument used to prove \cite[Part 2, \S 6, Lem. 1.8, pp. 229-230]{Passman}, with $G$ taken to be $G_i$, $H$ taken to be $N$ and $e$ taken to be $e_i$, shows that $e_{11}\Fl[G_i]e_{11}$ is a twisted group ring, in the sense of Part 1, \S 2 of loc. cit., of the group $L_i$ over the field
$$e_{11}\Fl[N]e_{11}=e_{11}(e_i\Fl[N])e_{11}\stackrel{(\ref{eiWD})}{\cong}E_{11}(M_{m_i}(\Fl(i)))E_{11}\cong\Fl(i).$$
Indeed, although for a given group $G$, normal subgroup $H$, field $K$ and idempotent $e\in K[H]$ central in $K[G]$, the statement of this result assumes that $eK[H]\cong M_m(K)$ for some $m$, and the relevant argument then shows that $e_{11}K[G]e_{11}$ is a twisted group ring of $G/H$ over $e_{11}K[H]e_{11}=e_{11}(eK[H])e_{11}\cong E_{11}(M_m(K))E_{11}\cong K$; every step works in an identical manner upon using the isomorphism (\ref{eiWD}) instead. 

To finally complete the proof of the lemma, we now recall that $L_i$ is an $\ell$-group and $\Fl(i)$ is a perfect field. Thus \cite[Part 1, \S 2, Lem. 2.10, p. 20]{Passman} implies that any twisted group ring of $L_i$ over $\Fl(i)$ is necessarily equal to the group ring $\Fl(i)[L_i]$.
\end{proof}

\subsubsection{}

We now fix a commutative (and associative) $\mathbb{F}_\ell$-algebra $B$ and, in the notation of Proposition \ref{NEWfiniteproducts}, for each index $1\leq i\leq m$, we define a commutative $B$-algebra $$R_i:=B \otimes_{\mathbb{F}_\ell}\Fl(i)[L_i].$$
We note that $R_i$ is finitely generated as a $B$-module. Setting also $n_i:=m_i\cdot|\co_i|$, from Proposition \ref{NEWfiniteproducts} we then derive a canonical decomposition
\begin{equation}\label{matrixdecomp}B[G]\,\cong\,\bigoplus_{i=1}^m M_{n_i}(R_i).\end{equation}
From this we derive a decomposition
$Z(B[G])\cong\bigoplus_{i=1}^m R_i$
and, for any $a,b\geq 1$, also
$$M_{a\times b}(B[G])\,\cong \bigoplus_{i=1}^m M_{(n_i\cdot a)\times (n_i\cdot b)}(R_i).$$
For any element $x$ of either of the above rings, let us write $x_i$ for the $i$-component of its image in the right-hand side. We denote by $e_{11}^{i}\in M_{n_i}(R_i)$ the matrix with 1 in position $(1,1)$ and zeros elsewhere, and regard it as an element of $B[G]$ through (\ref{matrixdecomp}).

We may now provide definitions of reduced determinants and of non-commutative Fitting ideals, the latter of which we follow \cite[Def. 1]{jn} of Johnston and Nickel for.

\begin{definition}\label{rdfidefs}{\em \

\noindent{}(i) Let $H$ be a square matrix over $B[G]$. The \emph{reduced determinant} of $H$ is
$$\Nrd_{B[G]}(H)\,:=\,(\det_{R_i}(H_i))_{1\leq i\leq m}\,\in\, Z(B[G]).$$

\noindent{}(ii) Let $M$ be a finitely presented $B[G]$-module. The \emph{non-commutative Fitting ideal} of $M$ is the ideal of $Z(B[G])$ given by
$$\Fit_{B[G]}(M)\,:=\,\bigoplus_{i=1}^m\Fit_{R_i}(e_{11}^{i}M).$$
}\end{definition}

\begin{remark}{\em If each $n_i$ is equal to 1, these recover the usual definitions of determinants and of Fitting ideals over $\bigoplus_{i=1}^mR_i=B[G]$. We have chosen the notation $\Nrd_{B[G]}$ to emphasize the fact that, if $B$ is a field and $\ell\nmid|G|$, so that $G=N$, then reduced determinants also coincide with the usual reduced norms over the semisimple ring $B[N]$. We also refer the reader to \cite[Thm. 2.2, Prop. 2.4, Prop. 3.4]{jn} for the basic properties of non-commutative Fitting ideals.}\end{remark}

\begin{remark}\label{Nrdforendomorphisms}{\em Let $P$ be a finitely generated, projective $B[G]$-module, and let $\Phi$ be a $B[G]$-endomorphism of $P$. Then the reduced determinant
$$\Nrd_{B[G]}(\Phi)\,:=\,\Nrd_{B[G]}\left(H(\Phi\oplus{\rm id}_Q)\right)\,\in\, Z(B[G])$$
of $\Phi$ is easily seen to be independent of the choice of $B[G]$-module $Q$ with the property that $P\oplus Q$ is free of finite rank, as well as of the choice of $B[G]$-basis with respect to which the matrix $H(\Phi\oplus{\rm id}_Q)$ of the endomorphism $\Phi\oplus{\rm id}_Q$ of $P\oplus Q$ is computed.

Given $B'\supseteq B$, the compatibility of the respective decompositions of the form (\ref{matrixdecomp}) implies that $\Nrd_{B'[G]}(B'\otimes_B \Phi)=\Nrd_{B[G]}(\Phi)$. Non-commutative Fitting ideals satisfy an analogous base-change property (cf. \cite[Thm. 2.2(vii)]{jn}).
}\end{remark}

In the next result, we set $SK_1(R_i):=\{x\in K_1(R_i)\mid\,\det_{R_i}(x)=1\}$.

\begin{lemma}\label{passestoK} Taking reduced determinants induces a 
split short exact sequence of abelian groups
$$1\to\bigoplus_{i=1}^m SK_1(R_i)\longrightarrow K_1(B[G])\xrightarrow{\Nrd_{B[G]}} Z(B[G])^\times\to 1.$$
If $B$ is a (commutative) local ring,
then the first term vanishes and $\Nrd_{B[G]}$ is an isomorphism. 
\end{lemma}
\begin{proof}
It suffices to prove the case $m=1$. In this case, the first claim follows from the Morita equivalence between $B[G]\cong M_{n_1}(R_1)$ and $R_1$ and the second claim from Prop. (5.28) in Volume I and Prop. (45.12) in Volume II of \cite{curtisr}.    
\end{proof}

\begin{lemma}\label{ctsNrd}Let $Z(\FiG)^\times\subseteq\FiG^\times$ have the subspace topology. Then the isomorphism $\Nrd_{\FiG}:K_1(\FiG)\to Z(\FiG)^\times$ is continuous.
\end{lemma}
\begin{proof} Each $R_i:=\Fi\otimes_{\mathbb{F}_\ell}\Fl(i)[L_i]$ is a finite dimensional $\Fi$-space, and so is the right-hand side of (\ref{matrixdecomp}). We endow it with the supremum ring norm, which defines the unique Hausdorff $\Fi$-space topology. Since $\FiG$ is also a Hausdorff $\Fi$-space, the ring isomorphism (\ref{matrixdecomp}) is an isomorphism of topological rings. The same claim is true of the restriction $Z(\Fi[G])\cong\bigoplus_{i=1}^m R_i$ of (\ref{matrixdecomp}).
The respective induced maps $\Fi[G]^\times\cong\bigoplus_{i=1}^m \GL_{n_i}(R_i)$ and $Z(\Fi[G])^\times\cong\bigoplus_{i=1}^m R_i^\times$ are isomorphisms of topological groups. 

Now, for each $i$, $\det_{R_i}:\GL_{n_i}(R_i)\to R_i^\times$ is continuous, since it may be computed through sums and products on the topological ring $R_i$. Thus $\bigoplus_{i=1}^m \det_{R_i}:\Fi[G]^\times\to Z(\Fi[G])^\times$ is continuous. Since the latter map factors as the composition of $\Nrd_{\FiG}$ with the open quotient map $\FiG^\times\to K_1(\FiG)$, $\Nrd_{\FiG}$ is continuous, as claimed.
\end{proof}

We make explicit the relationship between non-commutative Fitting ideals and characteristic classes.
\begin{lemma}\label{Fitchar} Let $M$ be an $A[G]$-module that is both finite and $G$-c.t. Then $$\Fit_{A[G]}(M)=Z(A[G])\cdot{\rm Nrd}_{\F[G]}(c_G(M)).$$ In particular, ${\rm Nrd}_{\F[G]}(c_G(M))$ belongs to $Z(A[G])$ and $\Fit_{A[G]}(M)$ is principal.
\end{lemma}
\begin{proof}
The $A[G]$-module $M'$ is finitely generated and projective (by Lemma \ref{TMremark1}(i)), so we may and do fix an $A[G]$-module $Q$ with the property that $M'\oplus Q$ is free of finite rank, say $n$.
By Lemma \ref{43.4}, there is then a short exact sequence
$$0\to M'\oplus Q\xrightarrow{\tau_M\oplus{\rm id}}M'\oplus Q\longrightarrow M\to 0.$$

Writing $(\tau_M\oplus{\rm id})_i$ for the endomorphism of $M_{n_i}(R_i)^{n}$ induced by the restriction of $\tau_M\oplus{\rm id}$ and an arbitrary choice of $A[G]$-basis of $M'\oplus Q$, we then find that
\begin{align*}\Fit_{A[G]}(M)=&{\bigoplus}_{i=1}^m\left(R_i\cdot\det_{R_i}(e^{i}_{11}(\tau_M\oplus{\rm id}))\right)\\
=&{\bigoplus}_{i=1}^m\left(R_i\cdot\det_{R_i}((\tau_M\oplus{\rm id})_i)\right)\\
=&Z(A[G])\cdot{\rm Nrd}_{\F[G]}(\left[(M'\oplus Q)_\F,(\tau_M\oplus{\rm id})_\F\right])\\ =&Z(A[G])\cdot{\rm Nrd}_{\F[G]}(c_G(M)).\end{align*}
Here the second equality holds by \cite[Lem. 2.3]{jn}.
\end{proof}

One may finally deduce the following explicit relationship between non-commutative Fitting ideals and refined Euler characteristics. Since we do not use it in the sequel, we omit the proof. 

\begin{corollary}\label{FIREC} Let $M$ be an $A[G]$-module that is both finite and $G$-c.t.
Then an element $x$ of $Z(\FiG)^\times$ satisfies \begin{equation*}\label{CElemma}\partial_{G}\left(\Nrd_{\FiG}^{-1}(x)\right)\,=\,-\chi_{G}(M[-2],0)\end{equation*} in $K_0(A[G],\FiG)$ if and only if it belongs to $Z(A[G])$ and generates $\Fit_{A[G]}(M)$.
\end{corollary}

\subsection{Galois structure of Taelman class groups}

In this section we will consider finite Galois extensions $K/k$ that satisfy the following:

\begin{As}\label{assumptionSD}{\em $\ell$ does not divide the degree of $K$ over the maximal abelian subextension of $k$ in $K$.
}\end{As}

The Galois group $G:=\Gal(K/k)$ then admits a decomposition of the form (\ref{semidirect}), which induces (compatible) decompositions (\ref{matrixdecomp}) for the group rings $A[G]$ and $\Fi[G]$.

\begin{example}\label{TheExample}{\em Let $\ell=q=2$ and $k=\mathbb{F}_2(t)$, and let $D/\mathbb{F}_2[t]$ be the Drinfeld module defined by the morphism $\phi_D:\mathbb{F}_2[t]\to\mathbb{F}_2[t]\{\tau\}$ given by $\phi_D(t):=t+\tau+\tau^2$. Then the extension $K_{D,t}:=k\left(D[t]\right)$ generated by the $t$-torsion points on $D$ is Galois with ${\rm Gal}(K_{D,t}/k)\cong S_3$, so $K_{D,t}$ has degree 3 over the maximal abelian subextension of $k$ in $K_{D,t}$ (cf. \cite[Exam. 3.5.6]{Papikian}).
}\end{example}

\subsubsection{}
We now fix a Drinfeld module $E/\co_k$ and a taming module $\cm$ for $K/k$. 

\begin{definition}{\em The \emph{$\cm$-modified Stickelberger element} of $(E,K/k)$ is
$$\theta^{\cm}_{E,K/k}\,:=\,{\rm Nrd}_{\F_\infty[G]}\left(\Theta^{\cm}_{E,K/k}\right)\,\in\,Z(\F_\infty[G])^\times.$$}
\end{definition}

We set $\UU_{E,K}^\cm:=\expm$.

\begin{theorem}\label{mtII} Let $M^1$ be any projective $A[G]$-lattice in $K_\infty$ that contains $\UU_{E,K}^\cm$ and fits into an exact commutative diagram as in claim (i) of Theorem \ref{POdiag} below.

Then the set
\begin{equation}\label{M1set}\left\{\theta^{\cm}_{E,K/k}\cdot {\rm Nrd}_{\FiG}\left(\psi_{\Fi}\circ \lambda^{\cm,-1}_{E,K/k}\right)\Bigm\vert\,\psi\in\Hom_{A[G]}\left(M^1,\cm\right)\right\}\end{equation}
is contained in $Z(A[G])$ and moreover is contained in the $Z(A[G])$-ideal
$\Fit_{A[G]}\left(H(E/\cm)\right)$.

In particular, the set (\ref{M1set}) is also contained in the ideals ${\rm Ann}_{Z(A[G])}\left(H(E/\cm)\right)$, $\Fit_{A[G]}\left(H(E/\co_K)\right)$ and ${\rm Ann}_{Z(A[G])}\left(H(E/\co_K)\right)$ of $Z(A[G])$.
\end{theorem}

\begin{remark}\label{BScomp}{\em Fix any finite abelian extension $K$ of $k$, any Drinfeld module $E/\co_k$ and any taming module $\cm$ for $K/k$. Then, by construction, any `$E(K_\infty)/E(\cm)$-admissible' $A[G]$-lattice as in the `refined Brumer-Stark Theorem for Drinfeld modules' of \cite[Thm. 1.5.5]{fghp} fits into an exact commutative diagram as in claim (i) of Theorem \ref{POdiag} below. See Lemma \ref{admissible} for the details. 

For any such abelian extension,
Theorem \ref{mtII} therefore specialises to recover the refined Brumer-Stark Theorem of loc. cit.
Indeed, $\Nrd_{\FiG}:K_1(\FiG)\cong Z(\FiG)^\times$ is given by taking determinants if $G$ is abelian and, as explained in Remark \ref{forearlierref} below, $\theta^{\cm}_{E,K/k}$ coincides with the element $\Theta_{K/k}^{E,\cm}(0)$ defined on \cite[p. 2221]{fghp}. 

Meanwhile, $M^1$ may be taken to be a free $A[G]$-lattice (of rank $n:=[k:\F]$) by Theorem \ref{POdiag}. In this case, the factor $[\cm:M^1]_G^{-1}$ (as in Def. 4.1.4 of loc. cit.) is given by $\det(Y)\cdot{\rm Nrd}_{\F[G]}(c_G(\cN/\cm))$, where $\cN$ is an arbitrary free $A[G]$-lattice in $K_\infty$ that contains $\cm$ and $Y\in{\rm GL}_n(\FiG)$ is the transition matrix between arbitrary $A[G]$-bases $(e_i)_{1\leq i\leq n}$ of $M^1$ and $(e'_i)_{1\leq i\leq n}$ of $\cN$. To express this factor in the form ${\rm Nrd}_{\FiG}(\psi_{\Fi}\circ \lambda^{\cm,-1}_{E,K/k})$, one can define $\psi:M^1\to\cm$ as mapping a basis element $e_i$ to ${\rm Nrd}_{\F[G]}(c_G(\cN/\cm))\cdot e'_i$; indeed, by Lemma \ref{Fitchar}, ${\rm Nrd}_{\F[G]}(c_G(\cN/\cm))$ is a generator of the (classical) Fitting ideal $\Fit_{A[G]}(\cN/\cm)\subseteq{\rm Ann}_{A[G]}(\cN/\cm)$, so the latter term belongs to $\cm$.
}\end{remark}

\begin{remark}\label{finerexpectation}{\em Fix $M^1$ as in Theorem \ref{mtII} and abbreviate the $Z(A[G])$-ideal generated by (\ref{M1set}) to $I(M^1)$. Then we in fact prove that
\begin{equation}\label{doubleinclusions}\Fit_{A[G]}\left(M^1/\UU_{E,K}^\cm\right)\cdot\Fit_{A[G]}\left(H(E/\cm)\right)\,\subseteq\, I(M^1)\,\subseteq\, \Fit_{A[G]}\left(H(E/\cm)\right).
\end{equation}
}\end{remark}

\begin{remark}\label{n-tremark}{\em
By Remark \ref{TMremark2}, any $A[G]$-modules $M^1$ and $\cm$ in Theorem \ref{mtII} are locally-free of rank equal to $[k:\F]$. For every injective $\psi$, the reduced determinant in (\ref{M1set}) is invertible in $Z(\FiG)$. Moreover, the set (\ref{M1set}) is amenable to explicit computation in examples.
}\end{remark}

The proof of Theorem \ref{mtII} (and Remark \ref{finerexpectation}) is given in \S\ref{7}.

\subsubsection{}\label{322} In this section we make the result of Theorem \ref{mtII} explicit in some special cases.

\begin{remark}\label{rank1reg}{\em To construct an element of the set (\ref{M1set}) we assume that $K/k$ satisfies Assumption \ref{assumptionSD} and also, for simplicity, is tamely ramified at all maximal ideals of $\co_k$.
We set $n:=[k:\F]$, omit $\cm=\co_K$ from all notation and fix the following data.
\begin{itemize}
\item A free $A[G]$-lattice $L$ in $K_\infty$ containing $\UU_{E,K}$ and a basis $\underline{b}=\{b_j\}_{1\leq j\leq n}$ of $L$.
\item An element $a\in A\setminus\{0\}$ that annihilates $H(E/\co_K)$.
\item An $n$-tuple $\underline{\alpha}=(\alpha_i)_{1\leq i\leq n}$ in $\co_K$.
\end{itemize}
We set $c_j:=a^{-1}\cdot b_j\in L_{\F}$ and $\underline{c}:=\{c_j\}_{1\leq j\leq n}$. We then set $R_{\underline{c},\underline{\alpha}}^{E}:=(r_{i,j})_{1\leq i,j\leq n}$ in $M_n(\FiG)$, where
$r_{i,j}\in\FiG$ is the unique element defined by the equality
$$1\otimes\alpha_i\,=\,\sum_{j=1}^{j=n}r_{i,j}\cdot c_j$$
in $(\co_K)_{\Fi}\cong K_\infty= L_{\Fi}=\bigoplus_{j=1}^{j=n}\Fi[G]\cdot b_j=\bigoplus_{j=1}^{j=n}\Fi[G]\cdot c_j$.

By the argument of \cite[Prop. 4.2.3]{fghp}, the free $A[G]$-lattice $M^1:=\bigoplus_{j=1}^{j=n}A[G]\cdot c_j$ contains $\UU_{E,K}$ and fits into an exact commutative diagram as in claim (i) of Theorem \ref{POdiag} below.
We define $\psi\in\Hom_{A[G]}(M^1,\co_K)$ by setting $\psi(c_i):=\alpha_i$ for each $i$. Then
$$(\lambda_{E,K/k}^{-1}\circ\psi_{\Fi})(c_i)=\lambda_{E,K/k}^{-1}(\alpha_i)=\sum_{j=1}^{j=n}r_{i,j}\cdot c_j$$
so
$${\rm Nrd}_{\FiG}(\psi_{\Fi}\circ \lambda^{-1}_{E,K/k})={\rm Nrd}_{\FiG}(\lambda_{E,K/k}^{-1}\circ\psi_{\Fi})={\rm Nrd}_{\FiG}(R_{\underline{c},\underline{\alpha}}^{E}).$$
By Theorem \ref{mtII}, the product $\theta_{E,K/k}\cdot{\rm Nrd}_{\FiG}(R_{\underline{c},\underline{\alpha}}^{E})$ belongs to $\Fit_{A[G]}\left(H(E/\co_K)\right)$.
}\end{remark}

\begin{example}\label{CarlitzStick}{\em We assume the notation and hypotheses of Remark \ref{rank1reg} and also the notation of Example \ref{Carlitz}. In particular, we consider the Carlitz module $E=C$, defined by the morphism $\phi_C:A\to \co_k\{\tau\}$ with $\phi_C(t):=t+\tau$. Then the computation of Euler factors given in Example \ref{Carlitz} combines with Lemma \ref{ctsNrd} to give the equality
\begin{align*}\label{CSelement}\theta_{C,K/k}=&\Nrd_{\FiG}\bigl({\prod}_{\p\in\Spec(\co_k)}c_G\bigl(\co_K/\p\co_K\bigr)\cdot c_G\bigl(C(\co_K/\p\co_K)\bigr)^{-1}\bigr)\\ =&{\prod}_{\p\in\Spec(\co_k)}{\rm Nrd}_{\F[G]}(1-(N\p)^{-1}\Fr_\p\cdot e_{I_\p})^{-1}\end{align*}
in $Z(\Fi[G])^\times$. Therefore, by Theorem \ref{mtII},
\begin{equation*}\label{CSFitting}(\prod_{\p\in\Spec(\co_k)}{\rm Nrd}_{\F[G]}(1-(N\p)^{-1}\Fr_\p\cdot e_{I_\p})^{-1})\cdot{\rm Nrd}_{\FiG}(R_{\underline{c},\underline{\alpha}}^C)\in\Fit_{A[G]}\left(H(C/\co_K)\right).\end{equation*}
}\end{example}

\subsubsection{}

In the special case of finite Galois extensions $K$ of $k$ of degree not divisible by $\ell$, one has $\cm=\co_K$ and may also take $M^1=\UU_{E,K}^{\co_K}=\expo$. For such extensions, we thus immediately derive from (\ref{doubleinclusions}) the following explicit equality.

\begin{corollary}\label{mtIII} Let $K/k$ be a finite Galois extension of degree not divisible by $\ell$ and let $E/\co_k$ be a Drinfeld module.
Then
$$Z(A[G])\cdot \left\{\theta^{\co_K}_{E,K/k}\cdot {\rm Nrd}_{\FiG}\left(\varphi_{\Fi}\circ \lambda^{\co_K,-1}_{E,K/k}\right)\Bigm\vert\,\varphi\in\Hom_{A[G]}\left(\expo,\co_K\right)\right\}$$ is equal to $\Fit_{A[G]}\left(H(E/\co_K)\right)$.
\end{corollary}

\begin{remark}\label{ATcomp}{\em Corollary \ref{mtIII} constitutes a non-abelian generalisation of Theorem A of Angl\`es and Taelman in \cite{at}, who considered the case where $k=\Fq(t)$, $E=C$ is the Carlitz module, and $K=\Fq(t)\left(C[P]\right)$ is the extension generated by the $P$-torsion points on the Carlitz module, for some prime $P$ of $A$. Note that this extension is abelian and of degree congruent to $-1$ modulo $q$.

Indeed, in the setting of loc. cit. it is straightforward to deduce from Corollary \ref{mtIII} the equality of $A[G]$-lattices
$$\theta^{\co_K}_{C,K/\Fq(t)}\cdot\co_K\,=\,\Fit_{A[G]}\left(H(C/\co_K)\right)\cdot {\rm exp}_C^{-1}(C(\co_K)).$$
}\end{remark}

\section{Nuclear automorphisms}\label{4}

In \cite[\S 2]{TAM}, Taelman adapted Anderson's `trace calculus' from \cite[\S 2]{Anderson} to develop a theory of determinants for `nuclear endomorphisms' of infinite-dimensional $\Fq$-vector-spaces.
Later, given a finite abelian group $\mathcal{A}$, Ferrara, Green, Higgins and Popescu \cite[\S 2]{fghp} extended Taelman's theory to consider $\mathcal{A}$-equivariant determinants for nuclear endomorphisms of topological, projective $\Fq[\mathcal{A}]$-modules.

Since we do not have an appropriate notion of reduced determinant over the general group rings we wish to work with, we instead associate appropriate classes in Whitehead groups to nuclear automorphisms.

\subsection{Whitehead groups of adic rings}\label{Wadic}

\subsubsection{}
We follow Fukaya and Kato \cite[\S 1.4]{fukaya-kato} in defining a notion of adic ring.

\begin{definition}\label{adic}{\em 
Fix a prime $\ell$. A ring $\Lambda$ is \textit{adic} if there is a two-sided ideal $I$ of $\Lambda$ such that $\Lambda/I^n$ is finite of $\ell$-power order for each $n\geq 1$, and the map $$\Lambda\longrightarrow{{\varprojlim}}_{n\geq 1}\Lambda/I^n$$ is bijective.
}\end{definition}

\begin{remark}{\em Every adic ring $\Lambda$ is semilocal, and we regard it as a topological ring, endowed with the profinite topology.
Then $K_1(\Lambda)$ is a topological group, as in \S \ref{211}. This topology on $\Lambda$ (and thus on $K_1(\Lambda)$) is independent of the choice of $I$; for example, fundamental systems of neighbourhoods of zero are given by the powers of the Jacobson radical (by the proof of \cite[Lem. 1.4.4(1)]{fukaya-kato}), by the left ideals of finite index (by \cite[Lem. 5.1.1(b)]{RZ} for $M=\Lambda$) or by the two-sided ideals of finite index (by \cite[\S 1.4.5]{fukaya-kato}).
}\end{remark}

\begin{example}{\em For any finite ring $R$ of $\ell$-power order, the power series ring $R[\![Z]\!]$ in a formal variable $Z$ is adic (one can take $I$ to be generated by $Z$).
Another example of adic ring is the completed group ring $\co[\![\mathcal{G}]\!]$ for the valuation ring $\co$ of a finite extension of $\QQ_\ell$ and a profinite group $\mathcal{G}$ that has a topologically finitely generated pro-$\ell$ open normal subgroup (see \cite[\S1.4.2]{fukaya-kato}).
}\end{example}

In this section we prove the following extension of \cite[Prop. 1.5.1]{fukaya-kato}.

\begin{proposition}\label{151}
Let $\Lambda$ be an adic ring and let $I$ be any ideal as in Definition \ref{adic}. Then the canonical homomorphism
\begin{equation}\label{map151}K_1(\Lambda)\longrightarrow{\varprojlim}_{n\geq 1}K_1(\Lambda/I^n)\end{equation} is an isomorphism of topological groups.
\end{proposition}

\begin{remark}\label{K1Rfinite}{\em We recall again that for a semilocal ring $R$, the map $R^\times\to K_1(R)$ is surjective. In particular, for any finite ring $R$, the abelian group $K_1(R)$ is finite, so we regard the right-hand side of (\ref{map151}) as a profinite group.
Proposition \ref{151} thus implies that $K_1(\Lambda)$ is a profinite abelian group, and in particular is a Hausdorff space.
}\end{remark}

\begin{remark}{\em The proof of Proposition \ref{151} relies crucially on the special case, proved by Fukaya and Kato, in which $I$ is taken to be the Jacobson radical of $\Lambda$ and the corresponding map of the form (\ref{map151}) is shown to be an isomorphism of abelian groups.}\end{remark}

\subsubsection{} This section is devoted to the proof of Proposition \ref{151}.
We fix an ideal $I$ as in Definition \ref{adic} and write $J$ for the Jacobson radical of $\Lambda$. Note that we have inclusions $I\subseteq J$ and $1+I\subseteq 1+J\subseteq\Lambda^\times$ and that, in addition, the finiteness of $\Lambda/I$ implies that $J^m\subseteq I$ for some $m\geq1$. In the sequel we fix such an $m$.

\begin{lemma}\label{K1surj} For each $k\geq 1$, the maps
$$K_1(\Lambda/J^{mk})\to K_1(\Lambda/I^k)\,\,\,\text{ and }\,\,\,K_1(\Lambda/I^k) \to K_1(\Lambda/J^k)$$ 
are both surjective.\end{lemma}
\begin{proof}
The result follows readily from the definition of Whitehead groups after we ensure that, for every $n\geq 1$, the 
maps
$${\rm \GL}_n(\Lambda/J^{mk})\to {\rm \GL}_n(\Lambda/I^k)\,\,\,\text{ and }\,\,\,{\rm \GL}_n(\Lambda/I^k) \to {\rm \GL}_n(\Lambda/J^k)$$
are both surjective.

A two-sided ideal $H$ in a ring $R$ is radical if $1+H\subseteq R^\times$. Since both $I^k$ and $J^k$ are radical in $\Lambda$, the ideal $I^k/J^{mk}$ is radical in $\Lambda/J^{mk}$ and the ideal $J^k/I^k$ is radical in $\Lambda/I^k$. The required surjectivities follow from these facts as in \cite[Exer. I.1.12]{K-book}.
\end{proof}

Noting that the quotient rings occurring in Lemma \ref{K1surj} are finite, thus so are their $K_1$-groups (by Remark \ref{K1Rfinite}), we deduce that the induced limit maps
$$\alpha:{\varprojlim}_{k\geq 1}K_1(\Lambda/J^{mk})\to {\varprojlim}_{k\geq 1}K_1(\Lambda/I^k),\,\,\,\,\,\,\beta:{\varprojlim}_{k\geq 1}K_1(\Lambda/I^k) \to {\varprojlim}_{k\geq 1}K_1(\Lambda/J^k)$$
are both surjective.
Since the composition $\beta\circ\alpha$ is itself bijective, we then deduce that so are both $\alpha$ and $\beta$.

We may finally apply the result \cite[Prop. 1.5.1]{fukaya-kato} of Fukaya and Kato, which states that the map
$$K_1(\Lambda)\longrightarrow{\varprojlim}_{n\geq 1}K_1(\Lambda/J^n)$$ is bijective. Noting that this map is the composition of the map (\ref{map151}) with $\beta$, we have now shown that the map (\ref{map151}) is an isomorphism of abelian groups.

It remains to show that (\ref{map151}) is in fact, an isomorphism of topological groups. In order to do so we require the following fact, which is straightforward to prove.

\begin{lemma}\label{unitshomeo} The map $\Lambda^\times\longrightarrow\varprojlim_{n\geq 1}(\Lambda/I^n)^\times$ is an isomorphism of topological groups.\end{lemma}

We now denote by $\rho_n$ each map $(\Lambda/I^n)^\times\to K_1(\Lambda/I^n)$, write $\rho$ for their limit and
consider the commutative square
\begin{equation}\label{topqts}
\begin{tikzcd}
    \Lambda^\times \ar[r] \ar[d] &{\varprojlim}_{n\geq 1}(\Lambda/I^n)^\times \ar[d,"{\rho}"] \\
    K_1(\Lambda) \ar[r] & {\varprojlim}_{n\geq 1}K_1(\Lambda/I^n)
    \end{tikzcd}
\end{equation}
in which the horizontal arrows are bijective and the vertical arrows are surjective.

The vertical arrow on the left is both continuous and open, by definition of the topological group $K_1(\Lambda)$. As for $\rho$, it coincides with the composite map
$${\varprojlim}_{n\geq 1}(\Lambda/I^n)^\times\longrightarrow\Bigl({\varprojlim}_{n\geq 1}(\Lambda/I^n)^\times\Bigr)\Big/\Bigl({\varprojlim}_{n\geq 1}\ker(\rho_n)\Bigr)\cong{\varprojlim}_{n\geq 1}K_1(\Lambda/I^n).$$ Here, the isomorphism is of topological groups by \cite[III.59, Cor. 3]{TG}. Therefore, $\rho$ is also continuous and open.

These facts, together with Lemma \ref{unitshomeo} and the commutativity of (\ref{topqts}), 
imply that the map (\ref{map151}) is an isomorphism of topological groups, as claimed in Proposition \ref{151}.

\subsubsection{}The following consequence of Proposition \ref{151} will be useful in our study of infinite products.
In claim (i), given an adic ring $\Lambda$, an ideal $I$ as in Definition \ref{adic} and $n\in\NN$, we write $\pi_{I^n}$ for the map $K_1(\Lambda)\to K_1(\Lambda/I^n)$.

\begin{corollary}\label{convergencecor} Let $\Lambda$ be an adic ring. Let $(x_i)_{i\in\NN}$ be a sequence in $K_1(\Lambda)$.
\begin{itemize}
\item[(i)] The sequence $(\prod_{i=1}^{i=j}x_i)_{j\in\NN}$ converges in $K_1(\Lambda)$ if and only
if for any $I$ as in Def. \ref{adic}, and
every $n\in\NN$, $\pi_{I^n}(x_i)=1$ for all but finitely many indices $i\in\NN$.
\item[(ii)] Let $\beta:\Lambda\to S$ be an injective, continuous and open, homomorphism into a semilocal topological ring $S$. Let $\delta:K_1(\Lambda)\to K_1(S)$ be the map induced by $\beta$. Assume that $\ker(\delta)$ is closed in $K_1(\Lambda)$.

If the sequence $(\prod_{i=1}^{i=j}x_i)_{j\in\NN}$ converges in $K_1(\Lambda)$, then the sequence
$(\prod_{i=1}^{i=j}\delta(x_i))_{j\in\NN}$ converges in $K_1(S)$, and:
\begin{itemize}
\item[(a)] it converges to a unique element $\prod_{i\in\NN}\delta(x_i)$ of $K_1(S)$;
\item[(b)] for any bijection $\sigma:\NN\to\NN$, one has $\prod_{i\in\NN}\delta(x_{\sigma(i)})=\prod_{i\in\NN}\delta(x_i)$; and
\item[(c)] $\prod_{i\in\NN}\delta(x_i)=\delta(\prod_{i\in\NN}x_i)$.
\end{itemize}
\end{itemize}
\end{corollary}
\begin{proof} 
We first prove (i), and also (ii) in the special case that $\beta$ is the identity $\Lambda\to\Lambda$.

We fix $I$. Given $n\in\NN$, since $K_1(\Lambda/I^n)$ is a (finite) discrete space, the sequence $(\prod_{i=1}^{i=j}\pi_{I^n}(x_i))_{j\in\NN}$ converges if and only if $\pi_{I^n}(x_i)=1$ for all but finitely many indices $i\in\NN$. Thus, the latter condition is valid for every $n\in\NN$ if and only if the sequence $((\prod_{i=1}^{i=j}\pi_{I^n}(x_i))_{j\in\NN})_{n\in\NN}$ converges in $\prod_{n\in\NN}K_1(\Lambda/I^n)$, in which case it converges to a family of finite products $(\prod_{i=1}^{i=B_n}\pi_{I^n}(x_i))_{n\in\NN}$, for some $B_n\in\NN$. Such a family of finite products is obviously unique, belongs to $\varprojlim_{n\in\NN}K_1(\Lambda/I^n)$, and is invariant under re-ordering of the indices $i$ (as in claim (ii)(b)). By Proposition \ref{151}, both claim (i), and also (ii) in the special case $\beta={\rm id}_\Lambda$, are valid (with claim (ii)(c) being trivial).

In the general case, $\delta$ is continuous and open by Lemma \ref{ctsopenlem} so, by \cite[III.16, Prop. 24]{TG}, it induces an isomorphism of topological groups $K_1(\Lambda)/\ker(\delta)\cong\im(\delta)$. If $\ker(\delta)$ is assumed to be closed in the profinite abelian group $K_1(\Lambda)$, then it would follow that $\im(\delta)$ is a Hausdorff subspace of $K_1(S)$.

Even without this assumption, the continuity of $\delta$ implies that if the sequence $(\prod_{i=1}^{i=j}x_i)_{j\in\NN}$ converges to $\prod_{i\in\NN}x_i$ in $K_1(\Lambda)$, then the sequence \begin{equation*}\label{sequenceim}({\prod}_{i=1}^{i=j}\delta(x_i))_{j\in\NN}=(\delta({\prod}_{i=1}^{i=j}x_i))_{j\in\NN}\end{equation*} converges to $\delta(\prod_{i\in\NN}x_i)$ in $\im(\delta)$. It therefore also converges to $\delta(\prod_{i\in\NN}x_i)$ in $K_1(S)$. 

Assuming now that $\ker(\delta)$ is closed, $\im(\delta)$ is closed in $K_1(S)$ by Lemma \ref{ctsopenlem} and $\im(\delta)$ is Hausdorff, so the limit $\prod_{i\in\NN}\delta(x_i):=\delta(\prod_{i\in\NN}x_i)$ is indeed unique in $K_1(S)$.

It only remains to prove claim (ii)(b), which now follows from the case $\delta={\rm id}_\Lambda$ since $\prod_{i\in\NN}\delta(x_{\sigma(i)}):=\delta(\prod_{i\in\NN}x_{\sigma(i)})=\delta(\prod_{i\in\NN}x_i)=:\prod_{i\in\NN}\delta(x_i)$.
\end{proof}

\subsection{Nuclear automorphisms in Whitehead groups}

Let $G$ be a finite group.
The power series ring $$\Lambda:=\Fq[G][\![Z]\!]$$ over $R:=\Fq[G]$ is adic, with respect to $\ell$ and to the ideal $I$ generated by $Z$. 

In this section we use Proposition \ref{151} to associate a class in $K_1(\Lambda)$ to `nuclear $\Lambda$-automorphisms' of $R$-modules.

\subsubsection{}\label{421} In order to define a suitable notion of nuclear $\Lambda$-automorphism, we first follow \cite[Def. 2.1.1]{fghp} in fixing the following data $(V,\cU)$:

\smallskip

\noindent{}$\bullet$ A compact topological $R$-projective module $V$.
\smallskip

\noindent{}$\bullet$ A decreasing sequence $\cU=\{U_i\}_{i\geq 1}$ of $R$-projective submodules of $V$ that forms a basis of open neighbourhoods of 0 in $V$.

\begin{remark}\label{VGct}{\em By Lemma \ref{TMremark1}(i), an $R$-module is projective if and only if it is $G$-c.t. In particular, any quotient of projective $R$-modules is projective.
}\end{remark}

\begin{definition}\label{igeqM}{\em A continuous $R$-endomorphism $\varphi$ of $V$ is \emph{locally contracting} if there is $M\geq 1$ such that $\varphi(U_i)\subseteq U_{i+1}$ for all $i\geq M$. Then $U_M$ is a \emph{nucleus} for $\varphi$.}\end{definition}

\begin{remarks}\label{common-sum-comp}{\em \

\noindent{}(i) Any finite family of locally contracting endomorphisms of $V$ has a common nucleus.

\noindent{}(ii) If $\varphi$ and $\psi$ are locally contracting then so are the sum $\varphi+\psi$ and composition $\varphi\psi$.
}\end{remarks}

\begin{remark}\label{fgV}{\em If $V$ is a finitely generated, projective $R$-module then we always take $U_i=0$ for $i\geq 1$, so that every continuous endomorphism of $V$ is locally contracting.}\end{remark}

In applications of the general theory to Drinfeld module actions in \S \ref{5}, we are interested in the following type of example.

\begin{example}\label{RZ}{\em If $V$ is also a topological $R[Z]$-module, with respect to the $Z$-adic topology on $R[Z]$, and $W$ is an open submodule of $V$ that is $R$-projective and has no $Z$-torsion, one may take $U_i=Z^{i}\cdot W$. In this case, the action on $V$ of any polynomial without constant term in $R[Z]$ is locally contracting.

If, in this setting, $V$ is also an $R[Z]\{\tau\}$-module and $W$ is a submodule, then the action of $\tau$ is locally contracting, since $\tau\cdot U_i=\tau Z^{i}\cdot W=Z^{iq}\tau\cdot W\subseteq Z^{iq}\cdot W=U_{iq}$.
}\end{example}

\subsubsection{}
We may now define nuclear $\Lambda$-automorphism, with respect to our fixed $(V,\cU)$.

\begin{definition}\label{nuclear}{\em A \emph{nuclear $\Lambda$-automorphism of $(V,\cU)$} is a sequence $\Phi=(\varphi_j)_{j\geq 1}$ of continuous, locally contracting $R$-endomorphisms of $V$. 
We write
$$\tilde\varphi_j:=\begin{cases}{\rm id}_V,\,\,\,\,\,\,\,\,\,\,\,\,\,\,\,\,\,\,\,\,\,\,\,\,\,\,\text{if }j=0,\\
\varphi_j,\,\,\,\,\,\,\,\,\,\,\,\,\,\,\,\,\,\,\,\,\,\,\,\,\,\,\,\,\,\text{if }j\neq 0\end{cases}$$ and then denote by $1+\Phi$ the sequence $(\tilde\varphi_j)_{j\geq 0}$.
}\end{definition}

For any $R$-module $M$ and any $N\geq 1$, we set $M_N:=(\Lambda/Z^N)\otimes_R M$.

\begin{lemma}\label{autom} Fix a nuclear $\Lambda$-automorphism $\Phi=(\varphi_j)_{j\geq 1}$ of $(V,\cU)$, an integer $N\geq 1$ and a common nucleus $U\in\cU$ for $\varphi_1,\ldots,\varphi_{N-1}$. Then the $\Lambda/Z^N$-endomorphism
\begin{equation*}\label{poly-autom}(1+\Phi)_N:=\sum_{j=0}^{j=N-1}(Z^j\otimes\tilde\varphi_j)\,:\,\,\,(V/U)_N\,\longrightarrow\,(V/U)_N\end{equation*}
is bijective and its class $\left[1+\Phi\right]_N$ in $K_1(\Lambda/Z^N)$ is independent of the choice of $U\in\cU$.
\end{lemma}
\begin{proof}
We can write the inverse of $(1+\Phi)_N$ explicitly using the standard recursive formulae for inverses of power series. 
For any $R$-endomorphisms $\psi_0,\ldots,\psi_{N-1}$ of $V/U$, one has the following equality of $\Lambda/Z^N$-endomorphisms of $(V/U)_N$,
$$(1+\Phi)_N\circ \sum_{j=0}^{j=N-1}(Z^j\otimes \psi_j)\,=\,\sum_{k=0}^{k=N-1}\Bigr(Z^k \otimes \Bigl(\sum_{j=0}^{j=k}\tilde\varphi_j \psi_{k-j}\Bigr)\Bigr).$$
For $\psi_0 = {\rm id}_{V/U}$ and
$\psi_k = - {\sum}_{j=1}^{j=k}\tilde\varphi_j \psi_{k-j}$ for $k>0$,
one indeed has
${\sum}_{j=0}^{j=k}\tilde\varphi_j \psi_{k-j} = 0$.

Now assume given common nuclei $U$, $W$ for $\varphi_1,\ldots,\varphi_{N-1}$ in $\cU$ and assume without loss of generality that $U \subset W$, so there is a descending chain
\begin{equation*}
    W = U_J \supset U_{J+1} \supset \cdots \supset U_M = U.
\end{equation*}
Consider the short exact sequences of finitely generated, projective $\Lambda/Z^N$-modules
\begin{center}
    \begin{tikzcd}[column sep= small]
    0 \ar[r] & (W / U)_N\ar[r] & (V/ U)_N \ar[r] & (V/W)_N \ar[r] & 0 \\
    0 \ar[r] & (U_{J+1} / U)_N \ar[r] & (W/ U)_N \ar[r] & (W/U_{J+1})_N \ar[r] & 0\\
    &&\cdots\\
    0 \ar[r] & (U_{M-1} / U)_N \ar[r] & (U_{M-2}/ U)_N \ar[r] & (U_{M-2}/U_{M-1})_N \ar[r] & 0.
    \end{tikzcd}
\end{center}
Since $(1+\Phi)_N$ defines automorphisms of each of these modules, in $K_1(\Lambda/Z^N)$ one has
$$\left[(V/U)_N,(1+\Phi)_N\right]\,=\,\left[(V/W)_N,(1+\Phi)_N\right]\cdot{\prod}_{i=J}^{i=M-1}\left[(U_i/U_{i+1})_N,(1+\Phi)_N\right].$$
However, $(1+\Phi)_N$ is the identity map on each of the modules $(U_i/U_{i+1})_N$, so the class of $(1+\Phi)_N$ does not depend on whether it is computed with respect to $U$ or $W$.
\end{proof}

\begin{remark}\label{compatclasses}{\em In the notation of Lemma \ref{autom}, if $N\geq 2$, then $[1+\Phi]_{N-1}=[(V/U)_{N-1},\sum_{j=0}^{j=N-2}(Z^j\otimes\tilde\varphi_j)]\in K_1(\Lambda/Z^{N-1})$ is the image of $[1+\Phi]_N\in K_1(\Lambda/Z^N)$.
}\end{remark}

\begin{definition}\label{limitclass}{\em
For a nuclear $\Lambda$-automorphism $\Phi$ of $(V,\cU)$, we use Proposition \ref{151}, Lemma \ref{autom} and Remark \ref{compatclasses} to define the \emph{power series class} of $\Phi$ in $K_1(\Lambda)$ to be $$\left[1+\Phi\right]=\left[1+\Phi\mid V\right]:=\Bigl(\left[1+\Phi\right]_N\Bigr)_{N\geq 1}.$$
}\end{definition}

\begin{remark}{\em Definitions \ref{igeqM} and \ref{nuclear} only make sense with respect to a fixed $V$ and $\cU$, and the same is true for Definition \ref{limitclass}. Even though the notation $[1+\Phi]$ omits the choice of both $V$ and $\cU$ and the notation $[1+\Phi\mid V]$ omits the choice of $\cU$, we always specify the $R$-module $V$ on which the components of $\Phi$ act, and the sequence $\cU$ with respect to which (they are locally contracting and) the power series class of $\Phi$ is taken to be computed (via choices of common nuclei $U\in\cU$ in Lemma \ref{autom}).
}\end{remark}

\begin{remark}\label{fgclass}{\em If $V$ is a finitely generated, projective $R$-module, then since $U_i=0$ for each $i$ by Remark \ref{fgV}, $\left[1+\Phi\mid V\right]$ is simply the class in $K_1(\Lambda)$ of the $\Lambda$-automorphism $\sum_{j=0}^{j=\infty}(Z^j\otimes\tilde\varphi_j)$ of the finitely generated, projective $\Lambda$-module $\Lambda\otimes_R V=\varprojlim_{N\geq 1}V_N$.}\end{remark}

\begin{remark}\label{ses}{\em
Let $V^1$ be a closed $R$-projective submodule of $V$. Set $V^2=V/V^1$, $\cU^1=\{U_i^1\}$ with $U_i^1:=U_i \cap V^1$ and $\cU^2=\{U_i^2\}$, where $U_i^2$ denotes the image of $U_i$ in $V^2$. Let $\{a,b\}=\{1,2\}$.

If each $R$-module $U_i^{a}$ is projective and $\Phi$ is a nuclear $\Lambda$-automorphism of $(V,\cU)$ that induces a well-defined  nuclear $\Lambda$-automorphism of $(V^{a},\cU^{a})$, then $\Phi$ also induces a nuclear $\Lambda$-automorphism of $(V^b,\cU^b)$ and, computing the respective power series classes with respect to $\cU$, $\cU^1$ and $\cU^2$, one has
\begin{equation*}
\left[1+ \Phi \mid V \right]\,=\,\left[1+\Phi \mid V^{1}\right]\cdot\left[1+\Phi \mid V^2\right].
\end{equation*} 
}\end{remark}

\begin{remark}\label{multiplicativity}{\em Given nuclear $\Lambda$-automorphisms $\Phi=(\varphi_j)_{j\geq 1}$ and $\Psi=(\psi_j)_{j\geq 1}$ of $(V,\cU)$, the sequence
$$(1+\Phi)(1+\Psi)-1\,:=\,\bigl(\varphi_j+\psi_j+{\sum}_{i=1}^{i=j-1}\varphi_i\psi_{j-i}\bigr)_{j\geq 1}$$ is a nuclear $\Lambda$-automorphism of $(V,\cU)$, and one has
$$\left[1+(1+\Phi)(1+\Psi)-1\right]\,=\,\left[1+\Phi\right]\cdot\left[1+\Psi\right].$$
}\end{remark}

\subsubsection{} Nuclear automorphisms comprising a single non-trivial term play an important role in the sequel.

\begin{definition}\label{charPSC}{\em For a continuous, locally contracting $R$-endomorphism $\varphi$ of $V$ and $m\geq 1$,  
we set $\left[1-\varphi_m\right]_N=\left[1+\Phi\right]_N$ for $N\geq 1$ and $\left[1-\varphi_m\right]=\left[1+\Phi\right]$, for $\Phi=(\varphi_j)$ with
$$\varphi_j:=\begin{cases}-\varphi,\,\,\,\,\,\,\,\,\,\,\,\,\,\,\,\,\,\,\,\,\,\,\,\text{if }j=m,\\
0,\,\,\,\,\,\,\,\,\,\,\,\,\,\,\,\,\,\,\,\,\,\,\,\,\,\,\,\,\,\text{if }j\neq m\end{cases}.$$
}\end{definition}

The following result extends \cite[Thm. 2]{TAM} to our setting.

\begin{proposition}\label{varphialpha} Assume given continuous $R$-endomorphisms $\alpha,\varphi$ of $V$ with the following properties.
\begin{itemize}\item[(a)] $\alpha$ is surjective. \item[(b)] Both of the compositions $\varphi\alpha$ and $\alpha\varphi$ are locally contracting, and there is a common nucleus $U_a\in\cU$ for them with
$$\varphi(U_a)\,\subseteq\,U_{a+b}\,\,\,\text{ and }\,\,\,\alpha(U_{a+b})\,\subseteq\,U_a$$
for some $b\geq 1$.
\end{itemize}
Then for every $m\geq 1$, one has
$$\left[1-(\varphi\alpha)_m\right]\,=\,\left[1-(\alpha\varphi)_m\right].$$
\end{proposition}
\begin{proof}
We first note that, since $V$ and $U_a$ are $R$-projective, the exact sequences
$$0\to\ker(\alpha)\to V\xrightarrow{\alpha}V\to 0\,\,\,\,\,\,\text{ and }\,\,\,\,\,\,0\to\ker(\alpha)\to\alpha^{-1}(U_a)\xrightarrow{\alpha}U_a\to 0$$
imply that the $R$-module $\alpha^{-1}(U_a)\cong\ker(\alpha)\oplus U_a$ is projective. Remark \ref{VGct} then implies that the quotients $\alpha^{-1}(U_a)/U_{a+b}$ and $V/\alpha^{-1}(U_a)$ are $R$-projective.

For each $m,N\geq 1$, one then has
\begin{align*}
\left[1-(\varphi\alpha)_m\right]_N&=\left[(V/U_{a+b})_N,1-(Z^m\otimes\varphi\alpha)\right]\\
&=\left[(V/U_{a+b})_N,1-(Z^m\otimes\varphi\alpha)\right]\cdot\left[(\alpha^{-1}(U_a)/U_{a+b})_N,1-(Z^m\otimes\varphi\alpha)\right]^{-1}\\
&=\left[\bigl(V/\alpha^{-1}(U_a)\bigr)_N,1-(Z^m\otimes\varphi\alpha)\right].
\end{align*}
The claimed equality now follows from the commutativity of the following 
square of $\Lambda/Z^N$-modules with bijective horizontal arrows:
\begin{equation*}\label{vasquareN}
\begin{CD}
\bigl(V/\alpha^{-1}(U_a)\bigr)_N @> 1\otimes\alpha >>  (V/U_a)_N \\
@V 1-(Z^m\otimes\varphi\alpha) VV @VV 1-(Z^m\otimes\alpha\varphi) V \\
\bigl(V/\alpha^{-1}(U_a)\bigr)_N @> 1\otimes\alpha >>  (V/U_a)_N. 
\end{CD}\end{equation*}
\end{proof}

\section{The refined trace formula}\label{5}

Taelman's trace formula \cite[Thm. 3]{TAM} provides an interpretation of Euler products associated to nuclear automorphisms, as a single determinant for their action on a compact (but not finitely generated) module over the formal power series ring $\Fq[\![Z]\!]$.

In this section we fix notation and setting as in Theorem \ref{MT} and prove a refined trace formula in the Whitehead group $K_1(\Fq[G][\![Z]\!])$. This formula provides an interpretation of $K$-theoretic Euler products as a single power series class of nuclear automorphisms, as introduced in Definition \ref{limitclass}.

In particular, our trace formula will imply that the Euler product $\Theta_{E,K/k}^{\cm}$ occurring in Theorem \ref{MT} converges to the element of $K_1(\Fi[G])$ obtained as the evaluation
\begin{equation*}\label{evaluation}\evt:K_1(\Lambda)\longrightarrow K_1(R(\!(Z)\!))\xrightarrow{Z\mapsto t^{-1}}K_1(R(\!(t^{-1})\!))=K_1(\Fi[G])\end{equation*} of a given power series class in $K_1(\Lambda)$.
Here, and in the sequel, we continue to abbreviate $\Fq[G][\![Z]\!]$ to $\Lambda$ and $\Fq[G]$ to $R$.

Before proceeding we note that $\Lambda$ is the closed unit ball in the ultrametric space $R(\!(Z)\!)$ and is thus an open subspace. By Lemma \ref{ctsopenlem}, $\evt$ is continuous and open. To apply Corollary \ref{convergencecor} (ii) to the inclusion $\Lambda\subset R(\!(Z)\!)\cong\FiG$, we use the following result, which is proved in \S \ref{Nsection} below. 

\begin{proposition}\label{otherKthm}The map $\evt$ is injective.\end{proposition}

In particular, $\ker(\evt)=\{1\}$ is closed in the profinite group $K_1(\Lambda)$.

\subsection{The compact module}

In this section we review the construction of an arithmetic family of pairs $(V_S^\cm,\cU)$ as in \S \ref{421} given in \cite{fghp}.
The refined trace formula of \S \ref{52} applies to a class of nuclear automorphisms of such pairs.
    
We fix a finite set $S$ of places of $k$ containing the set $S^\infty_k$ of places above the place at infinity of $\F$,
as well as a taming module $\cm$ for $K/k$. 

For each $v\in S$ we define a $G$-module $K_v:={\prod}_{w\mid v}K_w$, with $w$ running over the places of $K$ that divide $v$. We endow $K_v$ with the supremum of the norms on each $K_w$. If $v\in S^\infty_k$, resp. $v\in S\setminus S^\infty_k$, we assume that $t^{-1}$, resp. an uniformiser of $k_v$, has norm $q^{-1}$ in $K_v$.
We then define $\co_{k,S}[G]\{\tau\}$-modules $$K_S:={\prod}_{v\in S}K_v\,\,\,\,\,\,\,\,\text{ and }\,\,\,\,\,\,\,\,\cm_S:=\co_{k,S}\otimes_{\co_k}\cm,$$ where $\co_{k,S}$ is the ring of $S$-integers in $k$. We endow $K_S$ with the supremum norm.
Embedding $\cm_S$ into $K_S$ diagonally, we obtain an $\co_{k,S}[G]\{\tau\}$-module
$$V_S^\cm:=K_S/\cm_S.$$ 
We note that $\cm_S$ is discrete in $K_S$, because $\co_{k,S}\otimes_{\co_k}\co_K$ is. We write $\varpi$ for the projection $K_S\to V_S^\cm$.

\begin{remark}\label{RZonV}{\em For each $v\in S\setminus S^\infty_k$, we fix a uniformiser $\pi_v$ of $k_v$. For each $v\in S$, we endow $K_v$ with the  following action of $R[Z]$:
$$Z\cdot x:=\begin{cases}t^{-1}x,\,\,\,\,\,\,\,\,\,\,\,\,\,\,\,\,v\in  S^\infty_k,\\ \pi_vx,\,\,\,\,\,\,\,\,\,\,\,\,\,\,\,\,\,\,v\notin S^\infty_k.\end{cases}$$
Then $\cm_S$ is an $R[Z]$-submodule of $K_S$ and $V_S^\cm$ becomes a topological $R[Z]$-module. 

Let $\{\cW_v\}_{v\in S}$ be a family of open, $R$-projective $\co_{k_v}[G]\{\tau\}$-submodules of $K_v$ with the property that for each $v$, $\{Z^{i}\cdot\cW_v\}_{i\geq 0}$ is a basis of open neighbourhoods of zero in $K_v$. Then, since $\cm_S$ is discrete, there is $B\in\NN$ with the property that $(\prod_{v\in S}Z^B\cdot\cW_v)\cap\cm_S=\{0\}$ in $K_S$. Rescaling to $W_v:=Z^B\cdot\cW_v$, we may define an 
open $R[Z]$-submodule that is $R$-projective and has no $Z$-torsion by setting
\begin{equation}\label{varpi}W=\varpi\bigl({\prod}_{v\in S}W_v\bigr)\,\subseteq\,V_S^\cm.\end{equation}
Any such choice of family $\{\cW_v\}_{v\in S}$ thus defines data $U_i:=Z^{i}\cdot W$ as in Example \ref{RZ}.
}\end{remark}

\begin{definition}\label{tauzero}{\em We write $\co_{k,S}\{\tau\}\tau$ for the subset of $\co_{k,S}\{\tau\}$ comprising polynomials in $\tau$ with constant term $0$. 
We then denote the set of sequences in $\co_{k,S}\{\tau\}\tau$ by $\langle\tau\rangle:={\prod}_{j\geq 1}\co_{k,S}\{\tau\}\tau$.
In the sequel, we interpret any element of $\co_{k,S}\{\tau\}\tau$ as a continuous $R$-endomorphism of $V_S^\cm$ via its own left-action.
}\end{definition}

The next result is essentially proved by Ferrara, Green, Higgins and Popescu \cite{fghp}. 

\begin{proposition}\label{Uprop} The $R$-module $V_S^\cm$ is compact and projective, and there is a (non-empty) family $\cT_S^\cm$ of sequences of $R$-submodules of $V_S^\cm$ with the following properties.
\begin{itemize}
\item[(i)] Each sequence in $\cT_S^\cm$ is of the form $\{Z^{i}\cdot W\}_{i\geq 1}$ for some
$R$-projective submodule $W$ of $V_S^\cm$ as in (\ref{varpi}).
In addition, there exists $a_W\in\NN$ with the property that
$$Z^{a_W+i}\cdot\varpi(B_{K_S})\subseteq Z^{i}\cdot W\subseteq Z^{i}\cdot\varpi(B_{K_S})$$ for large enough $i$, where $B_{K_S}$ denotes the closed unit ball in $K_S$.

\item[(ii)] Any element of $\co_{k,S}\{\tau\}\tau$ is locally contracting with respect to any pair $(V_S^\cm,\cU)$ with $\cU$ in $\cT^\cm_S$.
\item[(iii)] If $\Phi$ is a sequence in $\langle\tau\rangle$, the power series class $\left[1+\Phi\mid V_S^\cm\right]$ of $\Phi$ in $K_1(\Lambda)$ is independent of the sequence in $\cT_S^\cm$ with respect to which it is computed.
\item[(iv)] Let $\p_0$ be a place of $k$ that is not in $S$ and set $S':=S\cup\{\p_0\}$. Then for any $\Phi$ in $\langle\tau\rangle$, computing the power series classes on the left-hand side with respect to sequences in $\cT^\cm_{S'}$ and in $\cT^\cm_S$ respectively (and the right-hand side as in Remark \ref{fgV}), one has
$$\left[1+\Phi\mid V_{S'}^\cm\right]\cdot\left[1+\Phi\mid V_S^\cm\right]^{-1}=\left[1+\Phi\mid \cm/\p_0\cm\right].$$
\end{itemize}
\end{proposition}
\begin{proof}
Since all the above claims follow in a straightforward manner from arguments in \cite{fghp}, we limit ourselves to providing a sketch of the proof.

The fact that $K_S$ is a projective $R$-module follows upon combining the normal basis theorem with Shapiro's Lemma. The fact that the $R$-module $V_S^\cm$ is compact and projective then follows directly from the properties of taming modules.

We then let $\cT_S^\cm$ comprise any sequences $\cU$ constructed as in \S 2.3.2 of loc. cit. Any such sequences satisfy the first statement of claim (i) by construction, the second statement of claim (i) by the argument in Cor. A.2.5 (see also (5.1.1)) of loc. cit., and then also claim (ii) by the argument in Lem. 2.3.3 of loc. cit. (or of Example \ref{RZ}).

To prove claim (iii) we fix $\Phi=(\varphi_j)_{j\geq 1}$ in $\langle\tau\rangle$, as well as sequences $\cU=\{U_i\}_{i\geq 1}$ and $\cU'=\{U_i'\}_{i\geq 1}$ in $\cT_S^\cm$, and an integer $N\geq 1$. It is enough to show that in $K_1(\Lambda/Z^N)$ 
one has
$$\left[(V_S^\cm/U_i)_N,(1+\Phi)_N\right]\,=\,\left[(V_S^\cm/U'_k)_N,(1+\Phi)_N\right]$$ for some common nuclei $U_i$ and $U'_k$ for $\varphi_1,\ldots,\varphi_{N-1}$.
But the argument of Lemma 2.3.10 of loc. cit. shows that there exist common nuclei $U_i$ and $U'_k$ with
$$U_i\,\supset\,U'_k\,\,\,\text{ and }\,\,\,\varphi_j(U_i)\,\subset\,U'_k\,\,\,\text{ for every }j<N.$$
Using this choice of nuclei we find the required equality
\begin{align*}\left[(V_S^\cm/U'_k)_N,(1+\Phi)_N\right]=&\left[(V_S^\cm/U_i)_N,(1+\Phi)_N\right]\cdot\left[(U_i/U'_k)_N,(1+\Phi)_N\right]\\ =&\left[(V_S^\cm/U_i)_N,(1+\Phi)_N\right].\end{align*} 

To prove claim (iv) we write $\co_{\p_0}$ for the valuation ring of the completion of $k$ at $\p_0$ and $\cm_{\p_0}:=\co_{\p_0}\otimes_{\co_k}\cm$ for the $\p_0$-adic completion of $\cm$. Then the proof of Lemma 3.0.1 of loc. cit. gives a short exact sequence of compact $\co_{k,S}[G]\{\tau\}$-modules
\begin{center}
    \begin{tikzcd}
    0 \ar[r] & \cm_{\p_0} \ar[r,"\psi"]& V_{S'}^\cm \ar[r,"\eta"] & V_S^\cm \ar[r] & 0,
    \end{tikzcd}
\end{center}
and a sequence $\cU'=\{U'_i\}_{i\geq 1}$ in $\cT^\cm_{S'}$ with the property that
$\psi^{-1}(U'_i)=\p_0^{i}\cdot\cm_{\p_0}$ 
and that $\{\eta(U_i')\}_{i\geq 1}$ belongs to $\cT^\cm_S$.

By Remark \ref{ses} one then finds that
$$\left[1+\Phi\mid V_{S'}^\cm\right]\cdot\left[1+\Phi\mid V_S^\cm\right]^{-1}=\left[1+\Phi\mid \cm_{\p_0}\right],$$
computing the power series class on the right-hand side with respect to the sequence $\{\p_0^{i}\cdot\cm_{\p_0}\}_{i\geq 1}$. The claimed equality is thus valid because $\p_0\cdot\cm_{\p_0}$ is a common nucleus for the action of all $\varphi_j\in \co_{k,S}\{\tau\}\tau$ on $\cm_{\p_0}$, by the argument of Example \ref{RZ}.
\end{proof}

\subsection{The trace formula and convergence of the \texorpdfstring{$L$}{L}-value}\label{52}

We may now state our refined trace formula for nuclear $\Lambda$-automorphisms in $\langle\tau\rangle$ (see Definition \ref{tauzero}). This result constitutes a non-abelian generalisation of \cite[Thm. 3.0.2]{fghp}. 

\begin{theorem}\label{tf} Let $S$ be a finite set of places of $k$ containing $S^\infty_k$ and let $\cm$ be a taming module for $K/k$.
Then for each nuclear $\Lambda$-automorphism $\Phi$ of $V_S^\cm$ in the set $\langle\tau\rangle$, the $S$-truncated Euler product
$${\prod}_{\p\notin S}\left[1+\Phi\mid \cm/\p\cm\right]$$
converges in $K_1(\Lambda)$ to the inverse of the power series class $\left[1+\Phi\mid V_S^\cm\right]$, computed with respect to any sequence in $\cT_S^\cm$. This limit is unique in $K_1(\Lambda)$ and, moreover, is independent of the order of the Euler factors.
\end{theorem}

Before proving Theorem \ref{tf} we state its consequence for $\Theta_{E,K/k}^{\cm}$. Set $V_\infty^\cm:=V_{S^\infty_k}^\cm$. 

\begin{corollary}\label{tfcor} Let $\cm$ be a taming module for $K/k$. Then the Euler product $\Theta_{E,K/k}^{\cm}$ converges in $K_1(\Fi[G])$ to $$\evt\left(\left[1+\Phi_E\mid V_\infty^\cm\right]\right),$$ where $\Phi_E=(\varphi_{j,E})_{j\geq 1}$ with $\varphi_{j,E}=(t-\phi_E(t))t^{j-1}$. This limit is unique in $K_1(\Fi[G])$ and, moreover, is independent of the order of the Euler factors.
\end{corollary}
\begin{proof}
We recall that $\Spec(\co_k)$ is a countable set. We fix an order and, for each $i\in\NN$, write $x_i\in K_1(\Lambda)$ for $\left[1+\Phi\mid \cm/\p_i\cm\right]$. Theorem \ref{tf} implies that the sequence of partial products $(\prod_{i=1}^{i=j}x_i)_{j\in\NN}$ converges in $K_1(\Lambda)$ to (a unique limit) $\prod_{i\in\NN}x_i=\left[1+\Phi\mid V_\infty^\cm\right]$ (independent of the chosen order).

By Proposition \ref{otherKthm} and Corollary \ref{convergencecor}(ii), it follows that the sequence of partial products $(\prod_{i=1}^{i=j}\evt(x_i))_{j\in\NN}$ converges in $K_1(\FiG)$ to a unique limit, that is independent of the chosen order, and is moreover equal to $\evt([1+\Phi\mid V_\infty^\cm])$.
Given these facts,
it is enough to show that, for any $\p\notin S^\infty_k$, one has
\begin{equation}\label{enoughlocal}\evt\left(\left[1+\Phi_E\mid \cm/\p\cm\right]\right)=c_G\left(\cm/\p\cm\right)^{-1}\cdot c_G\left(E\bigl(\cm/\p\cm\bigr)\right)\end{equation}
in $K_1(\Fi[G])$. We fix such a place $\p$. 

Now, a direct computation (as in the proof of Lemma \ref{autom}) gives for each $N\geq 2$ the following equality of continuous $\Lambda/Z^N$-automorphisms of $(\cm/\p\cm)_N$:
\begin{equation}\label{pwfact}\bigl(\id-(Z\otimes t)\bigr)^{-1} \circ \bigl(\id-(Z\otimes\phi_E(t))\bigr)=\id+{\sum}_{j=1}^{j=N-1}\bigl(Z^j\otimes\varphi_{j,E}\bigr).\end{equation}

We set $(\cm/\p\cm)_\Lambda:=\Lambda\otimes_{R}\cm/\p\cm$.
Writing $$\alpha_t:=\id-(Z\otimes t)\,\,\,\,\,\,\text{  and  }\,\,\,\,\,\,\alpha_{\phi_E(t)}:=\id-(Z\otimes\phi_E(t))$$ for the respective $\Lambda$-automorphisms of the finitely generated, projective module $(\cm/\p\cm)_\Lambda$, we thus find that
$$\left[1+\Phi_E\mid \cm/\p\cm\right]=\left[(\cm/\p\cm)_\Lambda,\,\alpha_t\right]^{-1}\cdot\left[(\cm/\p\cm)_\Lambda,\,\alpha_{\phi_E(t)}\right]$$ in $K_1(\Lambda)$. Using the notation of Lemma \ref{43.4}, this equality implies that
\begin{align*}&\evt\left(\left[1+\Phi_E\mid \cm/\p\cm\right]\right)\\
=&\left[(\cm/\p\cm)'_{\Fi},\,\id-(t^{-1}\otimes t)\right]^{-1}\cdot\left[(\cm/\p\cm)'_{\Fi},\,\id-(t^{-1}\otimes\phi_E(t))\right]\\
=&\left[(\cm/\p\cm)'_{\Fi},\,(t\otimes\id)-(\id\otimes t)\right]^{-1}\cdot\left[(\cm/\p\cm)'_{\Fi},\,(t\otimes\id)-(\id\otimes\phi_E(t))\right]\\
=&c_G\left(\cm/\p\cm\right)^{-1}\cdot c_G\left(E\bigl(\cm/\p\cm\bigr)\right)
\end{align*}
in $K_1(\Fi[G])$. We have now verified the required equality (\ref{enoughlocal}). 
\end{proof}

In the rest of this section we prove Theorem \ref{tf}, by adapting the proofs of \cite[Thm. 3]{TAM} and \cite[Thm. 3.0.2]{fghp}. We recall again that $\Spec(\co_k)$ is a countable set.

We fix $N\geq 1$. By (the proof of) Corollary \ref{convergencecor}(i), it is enough to prove that 
$${\prod}_{\p\notin S}\left[(1+\Phi)_N\mid \left(\cm/\p\cm\right)_N\right]$$
converges in $K_1(\Lambda/Z^N)$ to the inverse of $\left[1+\Phi\mid V_S^\cm\right]_N$, computed with respect to any sequence in $\cT_S^\cm$.

We write $\Phi=(\varphi_j)_{j\geq 1}$ and fix $D\,>\,N\cdot{\rm max}_{j<N}\{\deg_\tau(\varphi_j)\}$. We then also set
$$T\,:=\,S\cup\{\p\in\Spec(\co_{k,S}) \mid \,[\co_{k,S}/\p:\Fq]<D\}.$$
We note that the set $T$ is finite.
By Proposition \ref{Uprop} (iv), it hence suffices to prove that
\begin{equation}\label{enoughT}{\prod}_{\p\notin T}\left[(1+\Phi)_N\mid \left(\cm/\p\cm\right)_N\right]\end{equation}
converges in $K_1(\Lambda/Z^N)$ to the inverse of $\left[1+\Phi\mid V_T^\cm\right]_N$, computed with respect to any sequence in $\cT_T^\cm$. In fact, we claim that every factor in the product (\ref{enoughT}) is trivial, and that so is $\left[1+\Phi\mid V_T^\cm\right]_N$.

Now the argument of \cite[Thm. 3]{TAM} (which relies on a trick due to Anderson \cite[Prop. 9]{Anderson}), together with the multiplicativity property in Remark \ref{multiplicativity}, reduce these claims to showing that
\begin{equation}\label{19}\left[1-((s\tau^d)\alpha)_m\mid\cm/\p\cm\right]\,=\,\left[1-(\alpha(s\tau^d))_m\mid \cm/\p\cm\right]\end{equation}
and that
\begin{equation}\label{20}\left[1-((s\tau^d)\alpha)_m\mid V_T^\cm\right]\,=\,\left[1-(\alpha(s\tau^d))_m\mid V_T^\cm\right]\end{equation} for $\p\notin T$, any $\alpha,s\in\co_{k,T}$ and any $d,m\geq 1$. 

To prove (\ref{19}) we recall from Remark \ref{TMremark3} that $\cm$ is a locally-free $\co_k[G]$-module of constant local rank 1. Therefore,
$$\frac{\cm}{\p\cm}\cong\frac{\cm_{(\p)}}{\p\cm_{(\p)}}\cong\frac{\co_{k,(\p)}[G]}{\p\co_{k,(\p)}[G]}\cong\left(\frac{\co_{k,(\p)}}{\p\co_{k,(\p)}}\right)[G]\cong\left(\frac{\co_{k,T}}{\p\co_{k,T}}\right)[G].$$
Thus, if $\alpha$ belongs to $\p\co_{k,T}$, then both of the endomorphisms $(s\tau^d)\alpha$ and $\alpha(s\tau^d)$ of $\cm/\p\cm$ are trivial, so that both sides of (\ref{19}) are trivial. On the other hand, if $\alpha\notin\p\co_{k,T}$, then its action defines an automorphism of $\cm/\p\cm$, so (\ref{19}) follows from Proposition \ref{varphialpha} applied to $V=\cm/\p\cm$, $\alpha=\alpha$ and $\varphi=s\tau^d$.

To deduce (\ref{20}) from Proposition \ref{varphialpha}, we must fix $\cU=\{Z^{i}W\}_{i\geq 1}$ in $\cT_T^\cm$ and, after noting that $(s\tau^d)\alpha$ and $\alpha(s\tau^d)$ are both locally contracting with respect to any such choice (by Proposition \ref{Uprop} (ii)), show that they have a common nucleus $Z^{a}W$ which satisfies
\begin{equation}\label{enougha+b}(s\tau^d)(Z^{a}W)\subseteq Z^{a+b}W\,\,\,\text{ and }\,\,\,\alpha(Z^{a+b}W)\subseteq Z^{a}W\end{equation}
for some $b$.
We abuse notation to write $v$ for the normalised valuation at a place $v\in T$, and set
$$e_v:=\begin{cases}v(t^{-1}),\,\,\,\,\,\,\,\,v\in S_k^\infty,\\ 1,\,\,\,\,\,\,\,\,\,\,\,\,\,\,\,\,\,\,\,\,\,v\in T\setminus S_k^\infty. 
\end{cases}$$
We then fix a common nucleus $Z^{a}W$ for $(s\tau^d)\alpha$ and $\alpha(s\tau^d)$, with $a$ large enough that
$$b\,:=\,{\rm min}_{v\in T}\left\{a+\frac{v(s)}{e_v}\right\}\,\geq\,{\rm max}\left\{1,\,{\rm max}_{v\in T}\left\{-\frac{v(\alpha)}{e_v}\right\}\right\}.$$ 

We recall from (\ref{varpi}) that we may fix a decomposition $W=\prod_{v\in T}W_v$, where each $W_v$ is an $\co_{k_v}[G]\{\tau\}$-module.
For any $x=(x_v)\in W$ one then has
$$s\tau^d(Z^{a}x)=\bigl((st^{-2a}t^{-(aq^d-2a)}\tau^d(x_v))_{v\in S_k^\infty},(s\pi_u^{2a}\pi_u^{aq^d-2a}\tau^d(x_u))_{u\in T\setminus S_k^\infty}\bigr).$$
Since we have inequalities $$v(st^{-2a})=v(s)+2av(t^{-1})\geq (a+b)v(t^{-1}),\,\,\,\,\,\,\,\,\,\,u(s\pi_u^{2a})=u(s)+2a\geq a+b$$ and $aq^d\geq 2a$, the first inclusion of (\ref{enougha+b}) is valid.

For $x=(x_v)\in W$ one also has
$$\alpha(Z^{a+b}x)=Z^{a}\bigl((\alpha t^{-b}x_v)_{v\in S_k^\infty},(\alpha\pi_u^{b}x_u)_{u\in T\setminus S_k^\infty}\bigr).$$
Since $v(\alpha t^{-b})=v(\alpha)+bv(t^{-1})\geq 0$ and $u(\alpha\pi_u^{b})=u(\alpha)+b\geq 0$, the second inclusion of (\ref{enougha+b}) is also valid.

We finally note that by Corollary \ref{convergencecor}(ii), applied to the identity map $\Lambda\to\Lambda$, the limit $\left[1+\Phi\mid V_S^\cm\right]$ is unique in $K_1(\Lambda)$ and independent of the order of the Euler factors.
This completes the proof of Theorem \ref{tf}.

\section{The proof of Theorem \ref{MT}}\label{proofMT}

Before proceeding to prove Theorem \ref{MT}, we find it convenient to introduce the following general notion of $K$-theoretic Fitting classes.

\begin{definition}\label{KFdef}{\em Let $G$ be a finite group. Let $M$ be an $A[G]$-module that is both finite and $G$-c.-t. 
The \emph{$K$-theoretic Fitting class} of $M$ in $K_0(A[G],\FiG)$ is $$\KFit_G(M):=\partial_G(c_G(M)).$$ 
}\end{definition}

\begin{remark}\label{rationalFit}{\em In fact, $\KFit_G(M)$ can also be regarded as the image of $c_G(M)$ in $K_0(A[G],\F[G])$. We note that $\partial_G$ factors through the injective map $K_0(A[G],\F[G])\to K_0(A[G],\FiG)$}.
\end{remark}

The following result justifies our terminology.

\begin{lemma}\label{Fitindependence} Let $M$ be an $A[G]$-module that is both finite and $G$-c.-t. Let
\begin{equation}\label{categoryT}0\to P\xrightarrow{\alpha}Q\to M\to0\end{equation} be a short exact sequence of $A[G]$-modules, in which both $P$ and $Q$ are finitely generated and projective.
Then in $K_0(A[G],\F[G])$ one has
$$[P,\alpha_\F,Q]\,=\,\KFit_G(M).$$
\end{lemma}
\begin{proof}
As observed in Remark \ref{rationalFit}, throughout this proof we may, and do, regard $\KFit_G(M)$ as an element of $K_0(A[G],\F[G])$. We also abuse notation and write $\partial_G$ for the canonical map $K_1(\F[G])\to K_0(A[G],\F[G])$.

Let $\mathcal{C}$ be the category of finite $A[G]$-modules that admit a projective presentation of the form (\ref{categoryT}). There is an isomorphism from the Grothendieck group $K_0(\mathcal{C})$ of $\mathcal{C}$ to $K_0(A[G],\F[G])$ that maps $[M]$ to $[P,\alpha_\F,Q]$ (cf. \cite[Vol. II, Rk. (40.19), Lem. (40.10)]{curtisr}). In particular, $[P,\alpha_\F,Q]$ depends only on the isomorphism class of $M$.

Since Lemma \ref{43.4}
gives a projective presentation of the form (\ref{categoryT}),
we find that
$$[P,\alpha_\F,Q]=[M',\tau_{M,\F},M']=\partial_G([M'_\F,\tau_{M,\F}])=\partial_G(c_G(M))=\KFit_G(M).$$
\end{proof}

\subsection{Computation through an auxiliary lattice}\label{71}
To deduce Theorem \ref{MT} from Corollary \ref{tfcor}, for a fixed taming module $\cm$ for $K/k$, we must compute $\left[1+\Phi_E\mid V_\infty^\cm\right]$ in $K_1(\Fq[G][\![Z]\!])$. Here $\Phi_E=(\varphi_{j,E})_{j\geq 1}$ with $\varphi_{j,E}=(t-\phi_E(t))t^{j-1}$ and the power series class is computed with respect to an arbitrary choice of sequence in $\cT_{S_k^\infty}^\cm$. 

We set $R:=\Fq[G]$, $\Lambda:=R[\![Z]\!]$, $V_E:=E(K_\infty)/E(\cm)$, $\mathbb{U}:=\expm$ and $H:=H(E/\cm)$. By abuse of notation, we denote by ${\rm exp}_E$ the map $$K_\infty/\UU\longrightarrow V_E$$ induced by the exponential of $E$.
In this section we prove the next intermediate result.

\begin{theorem}\label{POdiag} The following claims are valid.
\begin{itemize}\item[(i)]
There exists a free $A[G]$-lattice $M^1$ in $K_\infty$ that contains $\UU$, 
and an exact commutative diagram of $A[G]$-modules of the form
\begin{equation*}
\begin{CD}
@. 0 @. 0 @. 0 @.\\
@. @V  VV @V  VV @V  VV @.\\
0 @>  >> M^1/\UU @>  >>  K_\infty/\UU @>  >> K_\infty/M^1 @>  >> 0\\
@. @V  VV @VV {\rm exp}_E V   @\vert @.\\
0 @>  >> M^2 @>  >>  V_E @>  >> K_\infty/M^1 @>  >> 0\\
@. @V  VV @V  VV @V  VV @.\\
@. H @. H @. 0 @.\\
@. @V  VV @V  VV @. @.\\
@. 0 @. 0 @. @. @.
\end{CD}\end{equation*}
in which the first row and second column are canonical. 
\item[(ii)] For any $A[G]$-modules $M^1$, $M^2$ as in claim (i), in $K_0(A[G],\Fi[G])$ one has
$$\partial_G\bigl(\Theta^\cm_{E,K/k}\bigr)\,=\,\bigl[M^1,\lambda,\cm\bigr]\,+\,\KFit_G(M^2).$$
Here we identify $\lambda=\lambda^\cm_{E,K/k}$ with the isomorphism
$M^1_{\Fi}=\UU_{\Fi}\xrightarrow{\lambda}\cm_{\Fi}$, and $M^2$ is $G$-c.-t. by the exactness of the second row of the diagram.
\end{itemize}
\end{theorem}

\begin{remark}{\em In fact, claim (ii) of Theorem \ref{POdiag} remains valid for any projective $A[G]$-lattice $M^1$ as in claim (i) (even if it is not free) and associated $A[G]$-module $M^2$.
}\end{remark}

To prove Theorem \ref{POdiag}, we first note that the $A$-module $K_\infty/\mathbb{U}$ is divisible, and we fix an $A$-module section $s:H\hookrightarrow V_E$ to the tautological surjection.

The construction of $A[G]$-modules $M^1$, $M^2$ as in claim (i) of Theorem \ref{POdiag} then follows directly from the following adaptation of the result \cite[Prop. 4.2.3]{fghp}, applied to the exact sequence $$0\to K_\infty/\UU\xrightarrow{{\rm exp}_E}V_E\longrightarrow H\to 0.$$

\begin{lemma}\label{admissible} Fix an exact sequence of topological $A[G]$-modules of the form
$$0\to K_\infty/L\stackrel{\iota}{\to}V\to\cH\to 0$$
where $L$ is an $A[G]$-lattice in $K_\infty$, $V$ is compact and $R$-projective and $\cH$ is finite. Fix an $A$-module section $s:\cH\hookrightarrow V$ to this sequence. Then the following claims are valid:
\begin{itemize}\item[(i)]
There exists a free $A[G]$-lattice $M^1$ in $K_\infty$ that contains $L$ and with the property that
$M^2:=\iota(M^1/L)\oplus s(\cH)$ is an $A[G]$-submodule of $V$.
\item[(ii)] For $M^1$ as in claim (i), there is an exact commutative diagram of $A[G]$-modules
\begin{equation*}
\begin{CD}
@. 0 @. 0 @. 0 @.\\
@. @V  VV @V  VV @V  VV @.\\
0 @>  >> M^1/L @>  >>  K_\infty/L @>  >> K_\infty/M^1 @>  >> 0\\
@. @V  VV @VV \iota V   @\vert @.\\
0 @>  >> M^2 @>  >>  V @>  >> K_\infty/M^1 @>  >> 0\\
@. @V  VV @VV   V @V  VV @.\\
@. \cH @. \cH @. 0 @.\\
@. @V  VV @V  VV @. @.\\
@. 0 @. 0. @. @. @. 
\end{CD}\end{equation*}
\item[(iii)] Set $L_0:=L$ and let $\{L_i\}_{1\leq i\leq m}$ be a finite set of $A[G]$-lattices in $K_\infty$, with the property that the union $\bigcup_{i=0}^{i=m}L_i$ is contained in a common $A[G]$-lattice in $K_\infty$. Then, in claim (i), $M^1$ may be chosen so that $\bigcup_{i=0}^{i=m}L_i\subseteq M^1$.
\end{itemize}
\end{lemma}
\begin{proof}
Claims (i) and (iii) follow directly from the approach of \cite[Prop. 4.2.3, Rk. 4.2.4]{fghp}. For later use, we note that the $G$-action on $V=\iota(K_\infty/L)\oplus s(\cH)$ is
\begin{equation}\label{Gaction}g\cdot(\iota(\lambda),s(h))=(\iota(g\lambda)+a_{g,h},b_{g,h})=(\iota(g\lambda)+a_{g,h},s(gh))\end{equation} for any $\lambda\in K_\infty/L$, $h\in\cH$ and $g\in G$, where
$(a_{g,h},b_{g,h}):=g\cdot(0,s(h))$ and we have also used the fact that $b_{g,h}$ is necessarily equal to $s(gh)$.

Now, from the explicit definition of $M^2$, we derive an exact commutative diagram of $A[G]$-modules
\begin{equation*}
\begin{CD}
@. 0 @. 0 @. @. @.\\
@. @V  VV @V  VV @. @.\\
0 @>  >> M^1/L @>  >>  K_\infty/L @>  >> K_\infty/M^1 @>  >> 0\\
@. @V  VV @VV \iota V   @V  VV @.\\
0 @>  >> M^2 @>  >>  V @>  >> V/M^2 @>  >> 0\\
@. @V  VV @VV V @. @.\\
@. \cH @. \cH @. @. @.\\
@. @V  VV @V  VV @. @.\\
@. 0 @. 0. @. @. @.
\end{CD}\end{equation*}
Since the first two columns of this diagram induce the identity map on $\cH$, the arrow in the third column must be bijective by the snake lemma, and it is therefore straightforward to modify the second row to obtain the diagram in claim (ii).
\end{proof}

\begin{remark}\label{forearlierref}{\em If $K/k$ is abelian, then the main result \cite[Thm. 1.5.1]{fghp} of Ferrara, Green, Higgins and Popescu follows directly from Theorem \ref{POdiag}.

Indeed, as in Remark \ref{impartial} and since $G$ is abelian, taking determinants over $\FiG$ induces an isomorphism
$$\im(\partial_G)\cong K_1(\FiG)/\im(K_1(A[G]))\cong\FiG^\times/A[G]^\times\cong\FiG^+,$$
with the last map as in \cite[Def. A.3.2, Prop. A.3.4, Cor. A.3.6]{fghp}. Also, by Prop. A.4.1 in loc. cit. (or by Lemma \ref{Fitchar} in this special abelian case), the determinant of the characteristic class $c_G(M)$ of any suitable module $M$ is equal to its `$A[G]$-size' $|M|_G$, as in Def. 1.2.4 or Def. A.4.2 of loc. cit.
Therefore, using also that taking determinants is continuous (as a special case of Lemma \ref{ctsNrd}), $\partial_G\bigl(\Theta^\cm_{E,K/k}\bigr)$ maps to
\begin{align*}&{\prod}_{\p\in\Spec(\co_k)}\det_{\FiG}\bigl(c_G\bigl(\cm/\p\cm\bigr)\cdot c_G\bigl(E(\cm/\p\cm)\bigr)^{-1}\bigr)\\ =&{\prod}_{\p\in\Spec(\co_k)}|\cm/\p\cm|_G\cdot |E(\cm/\p\cm)|_G^{-1},\end{align*}
which recovers the definition of $\Theta_{K/k}^{E,\cm}(0)\in\FiG^+$ on p. 2221 of loc. cit.

It thus suffices to observe that taking determinants maps the sum $[M^1,\lambda,\cm]+\KFit_G(M^2)$ to the `ratio of volumes' ${\rm Vol}(V_E)/{\rm Vol}(K_\infty/\cm)$ of \cite[Thm. 1.5.1]{fghp}.
On the one hand, for any free $A[G]$-lattice $\cN$ in $K_\infty$ that contains $\cm$, the former sum is the image under $\partial_G$ of the product
$$c_G(M^2)\cdot[X]\cdot c_G(\cN/\cm)^{-1},$$ where $X\in{\rm GL}_n(\FiG)$ is the transition matrix between arbitrary $A[G]$-bases of $\cN$ and of $M^1$ respectively. 
The claimed observation is now clear because, on the other hand, taking $\Lambda_0$ to be $\cm$ in Def. 4.2.6 of loc. cit., ${\rm Vol}(K_\infty/\cm)=1$ while ${\rm Vol}(V_E)$ is computed (taking $\Lambda'$ to be $M^1$, so that $\Lambda'/\Lambda\times s(H)=M^2$) as
$$|M^2|_G\cdot[M^1:\cm]_G^{-1}=|M^2|_G\cdot\det(X^{-1})^{-1}\cdot|\cN/\cm|_G^{-1}=|M^2|_G\cdot\det(X)\cdot|\cN/\cm|_G^{-1},$$
with the factor $[M^1:\cm]_G$ and the first equality given by Def. 4.1.4 in loc. cit.
}\end{remark}

In the rest of \S \ref{71} we prove Theorem \ref{POdiag} (ii).
We fix $A[G]$-modules $M^1$ and
$$M^2:={\rm exp}_E(M^1/\mathbb{U})\oplus s(H)$$
as in claim (i), and proceed to compute the class $\left[1+\Phi_E\mid V_\infty^\cm\right]$.
To do so we first construct, for each natural number $N$, approximations $M^1_N$ and $M^2_N$ of $M^1$ and $M^2$ that are related to the taming module $\cm$.

We set $n:=[k:\F]$, recall that $K_\infty\cong\FiG^n$ and fix an $\FiG$-basis $b_\bullet$ of $K_\infty$. We use this basis to interpret any $\sigma\in{\rm GL}_n(\FiG)$ as an automorphism of $K_\infty$. 

We fix a free $A[G]$-lattice $\cN$ in $K_\infty$ containing $\cm$ and $A[G]$-bases ${\bf e}$ and ${\bf e}'$ of $\cN$ and of $M^1$ and write $X\in{\rm GL}_n(\FiG)$ for the transition matrix between ${\bf e}'$ and ${\bf e}$. 

Endowing $M_n(\FiG)=M_n(R(\!(t^{-1})\!))$ with its $t^{-1}$-adic topology, $\GL_n(\FiG)$ is an open subset and $(\{{\rm Id}_n+t^{-N}\Mat_n(R[\![t^{-1}]\!])\})_{N\geq 0}$ is a basis of open neighbourhoods of ${\rm Id}_n$ in $\GL_n(\FiG)$. It is also straightforward to see that $M_n(\F[G])=M_n(R(t))$ is dense in $M_n(\FiG)$ and that $\GL_n(\F[G])=M_n(\F[G])\cap\GL_n(\FiG)$. These facts imply that $\GL_n(\F[G])$ is dense in $\GL_n(\Fi[G])$ and, in particular, that the intersection
$$\GL_n(\F[G])\,\cap\,\{\left({\rm Id}_n+t^{-N}\Mat_n(R[\![t^{-1}]\!])\right)\cdot X^{-1}\}$$
is non-empty for every natural number $N$.
We thus may and do fix a matrix $$\sigma_N \in \{{\rm Id}_n+t^{-2N}\Mat_n(R[\![t^{-1}]\!])\}$$ such that $X\cdot\sigma_N^{-1}$ belongs to $\GL_n(\F[G])$. This property then implies that the free $A[G]$-lattices $\sigma_N(M^1)$ and $\cN$ are contained in a common $A[G]$-lattice in $K_\infty$.

For each $N$ we also fix a push-out $(\pi_N:V_E\to\Pi_N,\iota_N:K_\infty/\sigma_N(\mathbb{U})\to\Pi_N)$ of $(\exp_E:K_\infty/\UU\to V_E,\sigma_N:K_\infty/\UU\to K_\infty/\sigma_N(\mathbb{U}))$. Then $\pi_N$ is bijective and $\iota_N$ is injective (because $\sigma_N$ is bijective and $\exp_E$ is injective), so we derive an exact commutative diagram
\begin{equation}\label{NPOut}
\begin{CD}
0 @>  >> K_\infty/\mathbb{U} @> {\rm exp}_E >> V_E @>  >> H @>  >> 0\\
@.  @V \sigma_N  VV @VV \pi_N V @\vert @.\\
0 @>  >> K_\infty/\sigma_N(\mathbb{U})   @> \iota_N >> \Pi_N @>  >> H @>  >> 0 \end{CD}\end{equation}
with bijective vertical arrows. We write $s_N$ for the composition
$$H\stackrel{s}{\hookrightarrow}V_E\stackrel{\pi_N}{\cong}\Pi_N,$$
which is then an $A$-module section to the lower row of (\ref{NPOut}).

Applying now Lemma \ref{admissible} to $s_N$, to the lower row of the diagram (\ref{NPOut}) and to the finite set of free $A[G]$-lattices $\{\sigma_N(M^1),\cN\}$ we deduce, for each natural number $N$, the existence of a free $A[G]$-lattice $M^1_N$, with
$$\sigma_N(\UU)\cup\cm\subseteq\sigma_N(M^1)\cup\cN\subseteq M^1_N$$
and with the property that
$$M^2_N:=\iota_N(M^1_N/\sigma_N(\UU))\oplus s_N(H)$$
is an $A[G]$-submodule of $\Pi_N$ that fits in an exact commutative diagram 
\begin{equation}\label{ýetanotherdiagram}
\begin{CD}
@. 0 @. 0 @. 0 @.\\
@. @V  VV @V  VV @V  VV @.\\
0 @>  >> M^1_N/\sigma_N(\mathbb{U}) @>  >>  K_\infty/\sigma_N(\mathbb{U}) @>  >> K_\infty/M^1_N @>  >> 0\\
@. @V  VV @VV \iota_N  V   @\vert @.\\
0 @>  >> M^2_N @>  >>  \Pi_N @>  >> K_\infty/M^1_N @>  >> 0\\
@. @V  VV @V  VV @V  VV @.\\
@. H @. H @. 0 @.\\
@. @V  VV @V  VV @. @.\\
@. 0 @. 0. @. @. @. 
\end{CD}\end{equation}

Regarding this data as fixed, we now obtain the following intermediate computation of the class $\left[1+\Phi_E\mid V_\infty^\cm\right]$.

\begin{lemma}\label{commonOL}For every large enough natural number $N$, in $K_1(\Lambda/Z^N)$ one has
$$\left[1+\Phi_E\mid V_\infty^\cm\right]_N\,=\,\left[(M^1_N/\cm)_N,\,1-(Z\otimes t)\right]^{-1}\cdot\left[(M^2_N)_N,\,1-(Z\otimes t)\right].$$
\end{lemma}
\begin{proof}
We first recall that any quotient of projective $R$-modules is also $R$-projective (see Remark \ref{VGct}). 
We fix an open, projective $R$-submodule $Q$ of $K_\infty$ as in Lemma \ref{openfinite} below. In particular, $K_\infty=M_N^1\oplus Q$.
We also fix, as we may, an $R$-splitting $\varpi:\Pi_N\cong M_N^2\oplus (K_\infty/M^1_N)$ of the second row of the diagram (\ref{ýetanotherdiagram}).

Then $\pi_N:V_E\cong\Pi_N$ induces an isomorphism of $R$-modules
$$\tilde\pi:(M^1_N/\cm)\oplus Q= V_\infty^\cm= V_E\cong\Pi_N\cong M_N^2\oplus (K_\infty/M^1_N)= M_N^2\oplus Q.$$

We now write $B_{K_\infty}$ for the closed unit ball in $K_\infty$. Then our choice of $\sigma_N$ combines with \cite[Prop. 5.1.3]{fghp} to imply that for every large enough natural number $j$, one has 
\begin{equation}\label{513}\tilde\pi(t^{-j}B_{K_\infty})=t^{-j}B_{K_\infty}\,\,\,\,\,\,\text{ and }\,\,\,\,\,\,(\tilde\pi^{-1}-1)(t^{-j}B_{K_\infty})\subseteq t^{-(j+2N)}B_{K_\infty}.\end{equation}

In particular, if $j$ is also large enough that $t^{-j}B_{K_\infty}\subseteq Q$, the $R$-isomorphism $\tilde\pi$ induces an $R$-isomorphism
$$(M^1_N/\cm)\oplus (Q/t^{-j}B_{K_\infty})\cong M_N^2\oplus (Q/t^{-j}B_{K_\infty}).$$
Thus, by the Cancellation Theorem \cite[20.12]{lam} (and Lemma \ref{openfinite}), there exists a (non-canonical) $R$-isomorphism
$\beta:M_N^1/\cm\cong M_N^2$.
We then consider the $R$-isomorphism
$$\delta:=\beta\oplus{\rm id}_Q:(M^1_N/\cm)\oplus Q\cong  M_N^2\oplus Q.$$

We next recall that the power series class $\left[1+\Phi_E\mid V_\infty^\cm\right]$ is computed with respect to the sequence $\{t^{-i}W\}_{i\geq 1}$ in $\cT_{S_k^\infty}^\cm$, for some $R$-projective, open, torsion-free-$R[t^{-1}]$-submodule $W$ of $V_\infty^\cm$ of the form (\ref{varpi}), which we now regard as fixed.

We first use the factorisation (\ref{pwfact}) to compute that
\begin{equation*}\label{firstcomput'}\left[1+\Phi_E\mid V_\infty^\cm\right]_N=\left[\frac{1-(Z\otimes\varpi t\varpi^{-1})}{1-(Z\otimes\tilde\pi t\tilde\pi^{-1})}\mid M_N^2\oplus Q\right]_N,\end{equation*}
with the right-hand side computed with respect to a common nucleus in $\{\tilde\pi(t^{-i}W)\}_{i\geq 1}$.

On the other hand, one has
\begin{align*}\left[\frac{1-(Z\otimes\varpi t\varpi^{-1})}{1-(Z\otimes\delta t\delta^{-1})}\mid M_N^2\oplus Q\right]_N&=\left[\frac{1-(Z\otimes t)}{1-(Z\otimes\beta t\beta^{-1})}\mid M_N^2\right]_N\cdot\left[\frac{1-(Z\otimes t)}{1-(Z\otimes t)}\mid Q\right]_N\\
&=\left[\frac{1-(Z\otimes t)}{1-(Z\otimes\beta t\beta^{-1})}\mid M_N^2\right]_N \\
=&\left[(M^1_N/\cm)_N,1-(Z\otimes t)\right]^{-1}\cdot\left[(M^2_N)_N,1-(Z\otimes t)\right].
\end{align*}

It therefore only remains to prove that the element
$$\left[\frac{1-(Z\otimes\delta t\delta^{-1})}{1-(Z\otimes\tilde\pi t\tilde\pi^{-1})}\mid M_N^2\oplus Q\right]_N=\left[\frac{1-(Z\otimes\tilde\pi^{-1}\delta t\delta^{-1}\tilde\pi)}{1-(Z\otimes t)}\mid V_\infty^\cm\right]_N$$
of $K_1(\Lambda/Z^N)$ is trivial or equivalently, in the notation of Definition \ref{charPSC} ($m=1$), that
$$\left[1-\left(\tilde\pi^{-1}\delta t\delta^{-1}\tilde\pi\right)_1\mid V_\infty^\cm\right]_N=\left[1-t_1\mid V_\infty^\cm\right]_N.$$

We set $\psi:=(\tilde\pi^{-1}\delta)-1$ and $\alpha:=t(\delta^{-1}\tilde\pi)$. Then,
by Taelman's argument in \cite[Cor. 1]{TAM}, it suffices to show that
$$\left[1-(\varphi\alpha)_1\right]_N\,=\,\left[1-(\alpha\varphi)_1\right]_N,$$
where $\varphi$ is a composition of, say, $n<N$ endomorphisms in the set $\{\alpha,\psi\}$, with $\psi$ occurring at least once.
In order to do so, we proceed to show that such an $\alpha$ and $\varphi$ satisfy the hypotheses of Proposition \ref{varphialpha}.

The map $\alpha$ is surjective, since both $\delta$ and $\tilde\pi$ are bijective while $V_\infty^\cm$ is $A$-divisible.

To consider hypothesis (b) of Proposition \ref{varphialpha}, we use the integer $a_W$ in Proposition \ref{Uprop} (i), which we recall is independent of $N$. We thus may and do assume that $N \geq 2 a_W$, and we also fix a natural number $i$ large enough that for every $j\geq i-N$, (\ref{513}) is valid and also $t^{-j}B_{K_\infty}\subseteq Q$, so $\delta$ restricts to the identity map on $t^{-j}B_{K_\infty}$.

In this case, one has
\begin{multline*}\varphi\alpha(t^{-i}W)\subseteq\varphi\alpha(t^{-i}B_{K_\infty})=\varphi(t^{-(i-1)}B_{K_\infty})\subseteq t^{-(i-n+2N)}B_{K_\infty}\subseteq t^{-(i+a_W+1)}B_{K_\infty}\\ \subseteq t^{-(i+1)}W\end{multline*}
and, similarly,
$$
\alpha\varphi(t^{-i}W)\subseteq \alpha\varphi(t^{-i}B_{K_\infty})\subseteq t^{-(i-n+2N)}B_{K_\infty}\subseteq t^{-(i+a_W+1)}B_{K_\infty}\\ \subseteq t^{-(i+1)}W.
$$
Thus both $\varphi\alpha$ and $\alpha\varphi$ are locally contracting, with common nucleus $t^{-i}W$. In fact, the same arguments as above also show that
$$\alpha(t^{-(i+a_W+1)}W)\subseteq\alpha(t^{-(i+a_W+1)}B_{K_\infty})\subseteq t^{-(i+a_W)}B_{K_\infty}\subseteq t^{-i}W$$
and that
$$\varphi(t^{-i}W)\subseteq\varphi(t^{-i}B_{K_\infty})\subseteq t^{-(i+2N-n+1)}B_{K_\infty} \subseteq t^{-(i+N-a_W+2)}W  \subseteq t^{-(i+a_W+1)}W.$$
Hypothesis (b) of Proposition \ref{varphialpha} is thus satisfied with $b=a_W+1$. This completes the proof of the lemma.
\end{proof}

\begin{lemma}\label{openfinite} Let $M$ be a free $A[G]$-lattice in $K_\infty$. There is an open, projective $R$-submodule $Q$ of $K_\infty$ with the property that $K_\infty=M\oplus Q$. In addition, if $j\in\NN$ is such that $t^{-j}B_{K_\infty}\subseteq Q$ (with $B_{K_\infty}$ the closed unit ball in $K_\infty$), then $Q/t^{-j}B_{K_\infty}$ is finite.
\end{lemma}
\begin{proof}
We recall that $\Fi$ decomposes as an $\Fq$-space as $A\oplus t^{-1}\Fq[\![t^{-1}]\!]$, with the second summand open in $\Fi$. Let $\mathbf{b}=\{b_1, \ldots, b_n\}$ be a basis of $M$. Then $Q :=\bigoplus_{i=1}^n t^{-1}\Fq[\![t^{-1}]\!][G]b_i=\bigoplus_{i=1}^n t^{-1}R[\![t^{-1}]\!]b_i$ is open in $K_\infty$ because each summand is, gives the $R$-decomposition $K_\infty=M\oplus Q$, and is $R$-projective by Remark \ref{VGct}.

To prove the second claim it suffices to observe that there is $a\geq 0$ with $t^{-a}Q\subseteq B_{K_\infty}$, since then $t^{-(a+j)}Q\subseteq t^{-j}B_{K_\infty}\subseteq Q$, and $Q/t^{-(a+j)}Q$ is finite. Indeed, since the norm of $t^{-1}$ is normalised to be $q^{-1}$, this is true whenever $q^{a}$ is greater than the norms of all elements in the given basis $\mathbf{b}$.
\end{proof}

Returning to the proof of claim (ii) of Theorem \ref{POdiag}, 
for each $N\in\NN$, we write $$\tilde\sigma_N\,\in\, \left\{{\rm Id}_n+Z^{2N}M_n(\Lambda)\right\}$$ for the image of $\sigma_N$. Then, in order to relate the right-hand side of the equality in claim (ii) of Theorem \ref{POdiag} to the right-hand side of the equality in Lemma \ref{commonOL}, 
we require the following intermediate result.

\begin{lemma}\label{itsapowerseries} Fix $N_0\in\NN$. Then, in $K_0(A[G],\FiG)$, one has
$$\bigl[M^1,\lambda,\cm\bigr]+\KFit_G(M^2)=(\partial_G\circ{\rm ev}_{t^{-1}})\left(\left[\tilde\sigma_{N_0}\right]^{-1}\cdot\frac{\left[\Lambda\otimes_R M^2_{N_0},\,1-(Z\otimes t)\right]}{\left[\Lambda\otimes_R M^1_{N_0}/\cm,\,1-(Z\otimes t)\right]}\right).$$
\end{lemma}
\begin{proof}
To compute the right-hand side of the claimed equality we recall from the proof of Lemma \ref{commonOL} that there is an isomorphism of $R$-modules $\beta:M^1_{N_0}/\cm \cong M^2_{N_0}$. Thus
\begin{multline*}\evt\left(\frac{\left[\Lambda\otimes_R M^2_{N_0},1-(Z\otimes t)\right]}{\left[\Lambda\otimes_R M^1_{N_0}/\cm,1-(Z\otimes t)\right]}\right)=\evt\left(\frac{\left[\Lambda\otimes_R M^2_{N_0},1-(Z\otimes t)\right]}{\left[\Lambda\otimes_R M^2_{N_0},1-(Z\otimes \beta t\beta^{-1})\right]}\right)\\
=\frac{\left[\FiG\otimes_R M^2_{N_0},t\otimes 1\right]}{\left[\FiG\otimes_R M^2_{N_0},t\otimes 1\right]}\cdot\frac{\left[\FiG\otimes_R M^2_{N_0},1-(t^{-1}\otimes t)\right]}{\left[\FiG\otimes_R M^2_{N_0},1-(t^{-1}\otimes \beta t\beta^{-1})\right]}
\\=\frac{\left[\FiG\otimes_R M^2_{N_0},(t\otimes 1)-(1\otimes t)\right]}{\left[\FiG\otimes_R M^2_{N_0},(t\otimes 1)-(1\otimes \beta t\beta^{-1})\right]}
\\=\frac{\left[\FiG\otimes_R M^2_{N_0},(t\otimes 1)-(1\otimes t)\right]}{\left[\FiG\otimes_R M^1_{N_0}/\cm,(t\otimes 1)-(1\otimes t)\right]}=\frac{c_G(M^2_{N_0})}{c_G(M^1_{N_0}/\cm)}.
\end{multline*}
Since also
$$(\partial_G\circ{\rm ev}_{t^{-1}})([\tilde\sigma_{N_0}])=\partial_G([K_\infty,\sigma_{N_0}])=\partial_G([\FiG\otimes_{A[G]}M^1,\sigma_{N_0}])=[M^1,\sigma_{N_0},M^1]$$
by definition of $\partial_G$ (or by \cite[Lem. 15.7]{swan}, but note that there is a typo in this result, where $[(P,f,P)]$ should read $[(P,g,P)]$),
we find that the right-hand side of the claimed equality is equal to
$$-[M^1,\sigma_{N_0},M^1]+\KFit_G(M^2_{N_0})-\KFit_G(M^1_{N_0}/\cm).$$

Applying Lemma \ref{Fitindependence} to the exact sequences
$$0\longrightarrow\cm\longrightarrow M^1_{N_0}\longrightarrow M^1_{N_0}/\cm\to 0$$
and
$$0\to M^1\xrightarrow{\sigma_{N_0}}M^1_{N_0}\longrightarrow M^1_{N_0}/\sigma_{N_0}(M^1)\to 0$$
we also find that
\begin{align*}[M^1,\lambda,\cm]+[M^1,\sigma_{N_0},M^1]+\KFit_G(M^1_{N_0}/\cm)=&[M^1,\lambda,M^1_{N_0}]+[M^1,\sigma_{N_0},M^1]\\=&[M^1,\sigma_{N_0},M^1_{N_0}]\\=&\KFit_G(M^1_{N_0}/\sigma_{N_0}(M^1)).\end{align*}
It therefore suffices to prove that
\begin{equation}\label{finalTODO}\KFit_G(M^1_{N_0}/\sigma_{N_0}(M^1))=\KFit_G(M^2_{N_0})-\KFit_G(M^2).
\end{equation}
In order to do so, we define a map
$$\mu:M^2=\exp_E(M^1/\UU)\oplus s(H)\longrightarrow \iota_{N_0}(M^1_{N_0}/\sigma_{N_0}(\UU))\oplus (\pi_{N_0}\circ s)(H)=M^2_{N_0}$$
by setting
$$\mu(\exp_E(\bar m),s(h)):=\left(\iota_{N_0}(\overline{\sigma_{N_0}(m)}),\pi_{N_0}(s(h))\right)$$
for $m\in M^1$ and $h\in H$. Here $\iota_{N_0},\pi_{N_0}$ are as in the push-out diagram (\ref{NPOut}).

Using the commutativity of (\ref{NPOut}) and (\ref{ýetanotherdiagram}) and the explicit $G$-action (\ref{Gaction}) on the modules $M^2$ and $M^2_{N_0}$, it is straightforward to verify that $\mu$ is a well-defined, injective $A[G]$-module homomorphism that makes the following exact diagram commute:
\begin{equation*}
\begin{CD}@. 0 @. 0 @. 0 @. 0 @.\\
@. @V   VV  @V   VV @VV V @VV V @.\\
0 @>  >> \mathbb{U} @>  >> M^1 @>  >> M^2 @>  >> H @>  >> 0\\
@. @V \sigma_{N_0} VV @V \sigma_{N_0}  VV @VV \mu V @\vert @.\\
0 @>  >> \sigma_{N_0}(\mathbb{U}) @>  >> M^1_{N_0} @>  >> M^2_{N_0} @>  >> H @>  >> 0\\
@. @V   VV  @V   VV @VV V @VV V @.\\
@. 0 @. M^1_{N_0}/\sigma_{N_0}(M^1) @. \cok(\mu) @. 0 @.\\
@. @.  @V   VV @VV V @. @.\\
@.  @. 0 @. 0. @.  @.
\end{CD}\end{equation*}

By the snake lemma, we find that $M^1_{N_0}/\sigma_{N_0}(M^1)$ is isomorphic to $\cok(\mu)$ and therefore (by Lemma \ref{additivity}) also that
$$\KFit_G(M^1_{N_0}/\sigma_{N_0}(M^1))=\KFit_G(\cok(\mu))=\KFit_G(M^2_{N_0})-\KFit_G(M^2).$$
We have now proved the equality (\ref{finalTODO}) and thus completed the proof.
\end{proof}

In order to complete the proof of Theorem \ref{POdiag}, we now require certain $K$-theoretical results.
\begin{lemma}\label{limKN} The canonical map $K_1(\Lambda)\to K_0(R,\Lambda)$ is surjective, and the isomorphism in Proposition \ref{151} induces an isomorphism of abelian groups
\begin{equation*}K_0(R,\Lambda)\,\cong\,{\varprojlim}_{N\in\NN}K_0(R,\Lambda/Z^N).\end{equation*}\end{lemma}
\begin{proof} 
The composite map $R\subset\Lambda\to\Lambda/Z$ is an isomorphism of rings. From this fact, it is straightforward to verify that the group $K_0(R,\Lambda/Z)$ is trivial. Since the canonical map $K_0(R,\Lambda)\to K_0(R)$ factors through $K_0(R,\Lambda/Z)$, it must be the zero map. The exactness of the localisation sequence \cite[Thm. 15.5]{swan} then proves the first claim.

For each $N\in\NN$, the canonical map $K_0(R,\Lambda/Z^N)\to K_0(R)$ also factors through $K_0(R,\Lambda/Z)$, so the same argument proves that the canonical map $K_1(\Lambda/Z^N)\to K_0(R,\Lambda/Z^N)$ is surjective.
We therefore have an exact sequence
$$K_1(R)\to K_1(\Lambda)\to K_0(R,\Lambda)\to 0$$
and, for each $N$, setting $$K_1(R)_N:=\im\left(K_1(R)\to K_1(\Lambda/Z^N)\right),$$
also an exact sequence of finite abelian groups 
$$0\to K_1(R)_N\to K_1(\Lambda/Z^N)\to K_0(R,\Lambda/Z^N)\to 0.$$

Now, the map $K_1(R)\to\varprojlim_{N\in\NN}K_1(R)_N$ is surjective. 
Therefore, the result follows from applying the snake lemma to the exact commutative diagram
\begin{equation*}\minCDarrowwidth 1.8em
\begin{CD}
@. K_1(R) @>  >> K_1(\Lambda) @>   >> K_0(R,\Lambda) @>  >>  0\\
@. @V   VV @V   VV @VV  V\\
0 @>  >> {\varprojlim}_{N\in\NN}K_1(R)_N @>  >> {\varprojlim}_{N\in\NN}K_1(\Lambda/Z^N) @>  >> {\varprojlim}_{N\in\NN}K_0(R,\Lambda/Z^N) @>  >>  0.
\end{CD}\end{equation*}
\end{proof}

We also require the next $K$-theoretic result, which will be proved in \S \ref{Nsection} below.

\begin{proposition}\label{Braunling} The canonical map $$\evt^0:K_0(R,\Lambda)\to K_0(R[Z^{-1}],R(\!(Z)\!))\xrightarrow{Z\mapsto t^{-1}}K_0(A[G],\FiG)$$ is injective.\end{proposition}

We next observe that there is a canonical commutative square
\begin{equation}\label{evsquare}
\begin{CD}
K_1(\Lambda) @> \widehat\partial_G  >> K_0(R,\Lambda)\\
@V \evt  VV @VV \evt^0 V\\
K_1(\FiG) @> \partial_G >> K_0(A[G],\FiG).
\end{CD}\end{equation}
For each $N\in\NN$, we set
$$Y'(N):=\left[\tilde\sigma_{N}\right]^{-1}\cdot\frac{\left[\Lambda\otimes_R M^2_{N},\,1-(Z\otimes t)\right]}{\left[\Lambda\otimes_R M^1_{N}/\cm,\,1-(Z\otimes t)\right]}\in K_1(\Lambda).$$
By Lemma \ref{itsapowerseries} and the commutativity of (\ref{evsquare}),
$$\evt^0(\widehat\partial_G(Y'(N))) = (\partial_G\circ\evt)(Y'(N))= \bigl[M^1,\lambda,\cm\bigr]+\KFit_G(M^2).$$
In particular, Proposition \ref{Braunling} implies that the pre-image $Y:=\widehat\partial_G(Y'(N))\in K_0(R,\Lambda)$ of the right-hand side under $\evt^0$ is independent of $N\in\NN$.

Now, Corollary \ref{tfcor} combines with the commutativity of (\ref{evsquare}) to reduce the proof of the equality in claim (ii) of Theorem \ref{POdiag} to proving that
\begin{equation}\label{hatpartial}\widehat\partial_G\left(\left[1+\Phi_E\mid V_\infty^\cm\right]\right)\,=\,Y.\end{equation}

In turn, by Lemma \ref{limKN}, the equality (\ref{hatpartial}) will be valid if and only if its projection to $K_0(R,\Lambda/Z^N)$ is valid for every $N\in\NN$, if and only if its projection to $K_0(R,\Lambda/Z^N)$ is valid for every large enough $N\in\NN$. We may thus fix $N_0\in\NN$, large enough that Lemma \ref{commonOL} applies to $N=N_0$, and proceed to verify that the projection to $K_0(R,\Lambda/Z^{N_0})$ of the equality (\ref{hatpartial}) is valid.

Using now the equality $Y=\widehat\partial_G(Y'(N_0))$ and the commutativity of the square
\begin{equation*}
\begin{CD}
K_1(\Lambda) @>  \widehat\partial_G >> K_0(R,\Lambda)\\
@V   VV @VV  V\\
K_1(\Lambda/Z^{N_0}) @>  >> K_0(R,\Lambda/Z^{N_0}),
\end{CD}\end{equation*}
it suffices to show that the image of $Y'(N_0)$ in $K_1(\Lambda/Z^{N_0})$ is equal to $\left[1+\Phi_E\mid V_\infty^\cm\right]_{N_0}$.

But, by choice of $\sigma_{N_0}$, the image in $K_1(\Lambda/Z^{N_0})$ of $\left[\tilde\sigma_{N_0}\right]$ is trivial, and hence the image of $Y'(N_0)$ in $K_1(\Lambda/Z^{N_0})$ is equal to
$$\frac{\left[(M^2_{N_0})_{N_0},\,1-(Z\otimes t)\right]}{\left[(M^1_{N_0}/\cm)_{N_0},\,1-(Z\otimes t)\right]}.$$
The proof of Theorem \ref{POdiag} is now completed by simply applying Lemma \ref{commonOL} to $N=N_0$.

\subsection{The proofs of Proposition \ref{otherKthm} and Proposition \ref{Braunling}}\label{Nsection}
We are grateful to Oliver Braunling for his advice regarding the argument that we present in this section.

We set $R:=\Fq[G]$, $\Lambda:=R[\![Z]\!]$, $\Pi:=R[Z^{-1}]$ and $\calL:=R(\!(Z)\!)$. By \cite[Thm. 15.5]{swan} and Lemma \ref{limKN}, we have the exact commutative diagram of abelian groups
\begin{equation*}\label{Kkeydiagram}
\begin{CD}
K_1(R) @> \alpha >> K_1(\Lambda) @>  \widehat\partial >> K_0(R,\Lambda) @>  >>  0\\
@V \beta  VV @V \delta  VV @VV \varepsilon V\\
K_1(\Pi) @> \gamma >> K_1(\calL) @> \partial >> K_0(\Pi,\calL).
\end{CD}\end{equation*}
The claims of Proposition \ref{otherKthm} and Proposition \ref{Braunling} are then equivalent to:
\begin{proposition}\label{injectivedeltas} The maps $\delta$ and $\varepsilon$ are both injective.
\end{proposition}

In the rest of this section, we prove Proposition \ref{injectivedeltas}. At the outset we fix notation for the following canonical maps:
\begin{align*}\beta':K_1(R)\to K_1(R[Z]),\,\,&\theta:K_1(R[Z])\to K_1(R[Z,Z^{-1}]),\,\,\iota:K_1(R[Z])\to K_1(\Lambda),\\ &\theta':K_1(\Pi)\to K_1(R[Z,Z^{-1}]),\,\,\,\,\,\,\,\lambda:K_1(R[Z,Z^{-1}])\to K_1(\calL).\end{align*}
We then observe the following relations:
\begin{equation}\label{relations}\alpha=\iota\circ\beta',\,\,\,\,\,\,\,\,\gamma=\lambda\circ\theta',\,\,\,\,\,\,\,\,\delta\circ\iota=\lambda\circ\theta.\end{equation}

\subsubsection{} Before proving Proposition \ref{injectivedeltas}, we adapt several useful results from \cite{K-book}. 

The set $S:=\{Z^j \mid j\geq 0\}$ is a central, multiplicatively closed subset of non-zero divisors in both of the rings $R[Z]$ and $\Lambda$.

We write $H(R[Z])$, resp. $H(\Lambda)$, for the category of $R[Z]$-modules, resp. $\Lambda$-modules, that admit a finite resolution by finitely generated projective $R[Z]$-modules, resp. $\Lambda$-modules. We then write $H_S(R[Z])$ and $H_S(\Lambda)$ for their respective exact subcategories comprising $S$-torsion modules. The following result extends the commutative square $\delta\circ\iota=\lambda\circ\theta$ from (\ref{relations}) to the $K_0$-groups of these categories.

\begin{proposition}\label{Weibel17Prop} There is a commutative diagram of abelian groups with exact rows
\begin{equation*}
\begin{CD}
1 @>  >> K_1(R[Z]) @>  \theta >> K_1(R[Z^{-1},Z]) @> \mu  >> K_0(H_S(R[Z]))\\
@. @V \iota  VV @VV \lambda  V @VV \zeta_0  V\\
1 @>  >> K_1(\Lambda) @> \delta >> K_1(\calL) @> \nu  >> K_0(H_S(\Lambda))
\end{CD}\end{equation*}
in which $\zeta_0$ is bijective. In particular, $\delta$ is injective, as claimed in Proposition \ref{injectivedeltas}.
\end{proposition}
\begin{proof}
We first claim that there is an exact commutative diagram of abelian groups
\begin{equation*}\label{Weibel7}
\begin{CD}
K_1(H_S(R[Z])) @> \eta >> K_1(R[Z]) @>  \theta >> K_1(R[Z^{-1},Z]) @> \mu  >> K_0(H_S(R[Z]))\\
@V  \zeta_1 VV @V \iota  VV @VV \lambda  V @VV \zeta_0  V\\
K_1(H_S(\Lambda)) @> \kappa >> K_1(\Lambda) @> \delta >> K_1(\calL) @> \nu  >> K_0(H_S(\Lambda))
\end{CD}\end{equation*}
in which $\zeta_0$ and $\zeta_1$ are induced by the scalar extension functor, which is exact by Proposition \ref{catequiv} below.

After noting that the localisations $S^{-1}R[Z]$ and $S^{-1}\Lambda$ identify with $R[Z^{-1},Z]$ and with $\calL$ respectively, \cite[Thm. V.7.1]{K-book} gives the exactness of both rows. The naturality of these exact localisation sequences with respect to the scalar extension from $R[Z]$ to $\Lambda$ is implicitly stated in \cite[Prop. V.7.5(2)]{K-book}, but we proceed to give a direct argument instead.

To first consider the right-hand square we use the explicit description of the connecting homomorphisms $\mu$, $\nu$ from \cite[Cor. III.3.1.1]{K-book}. Every element of $K_1(R[Z^{-1},Z])$ may be represented by a pair $[R[Z^{-1},Z]^n,s^{-1}\otimes_{R[Z]}\alpha]$, where $s$ is an element of $S$ and $\alpha$ is an injective endomorphism of $R[Z]^n$ with the property that $R[Z^{-1},Z]\otimes_{R[Z]}\cok(\alpha)$ vanishes. One then has
\begin{align*}(\zeta_0\circ\mu)([R[Z^{-1},Z]^n,s^{-1}\otimes_{R[Z]}\alpha])&=\zeta_0([\cok(\alpha)]\cdot[R[Z]/sR[Z]]^{-n})\\
&=[\Lambda\otimes_{R[Z]}\cok(\alpha)]\cdot[\Lambda\otimes_{R[Z]}(R[Z]/sR[Z])]^{-n}\\
&=[\cok(\Lambda\otimes_{R[Z]}\alpha)]\cdot[\Lambda/s\Lambda]^{-n}\\
&=\nu([\calL^n,s^{-1}\otimes_\Lambda(\Lambda\otimes_{R[Z]}\alpha)])\\
&=(\nu\circ\lambda)([R[Z^{-1},Z]^n,s^{-1}\otimes_{R[Z]}\alpha]).
\end{align*}

To next consider the left-hand square we write $\calP(R[Z])$, resp. $\calP(\Lambda)$, for the category of finitely generated, projective $R[Z]$-modules, resp. $\Lambda$-modules. Then the canonical maps $\eta$ and $\kappa$ decompose as $\eta=\eta_1^{-1}\circ\eta_2$ and $\kappa=\kappa_1^{-1}\circ\kappa_2$ where $\eta_1,\eta_2,\kappa_1,\kappa_2$ are respectively induced by the inclusions of categories $\calP(R[Z])\to H(R[Z])$, $H_S(R[Z])\to H(R[Z])$, $\calP(\Lambda)\to H(\Lambda)$ and $H_S(\Lambda)\to H(\Lambda)$ (see the proof of \cite[Lem. V.7.4]{K-book}). In these decompositions we have also used the fact that $\eta_1$ and $\kappa_1$ are bijective, by \cite[Thm. V.3.2]{K-book}. Using also the fact that the scalar extension functor $H(R[Z])\to H(\Lambda)$ is exact by Lemma \ref{Lambdaflat}, and thus induces a map $\zeta'$ on $K_1$-groups, the definition of the maps $\eta_1,\eta_2,\kappa_1,\kappa_2$ implies the commutativity of the diagram
\begin{equation*}
\begin{CD}
K_1(H_S(R[Z])) @>  \eta_2 >> K_1(H(R[Z])) @< \eta_1  << K_1(R[Z])\\
@V \zeta_1  VV @VV \zeta'  V @VV \iota  V\\
K_1(H_S(\Lambda)) @> \kappa_2 >> K_1(H(\Lambda)) @< \kappa_1  << K_1(\Lambda).
\end{CD}\end{equation*}
The commutativity $\kappa\circ\zeta_1=\kappa_1^{-1}\circ\kappa_2\circ\zeta_1=\kappa_1^{-1}\circ\zeta'\circ\eta_2=\iota\circ\eta_1^{-1}\circ\eta_2=\iota\circ\eta$ of the left-hand square thus follows from that of the last displayed diagram.

We next observe that $\zeta_0$ and $\zeta_1$ are bijective, by Proposition \ref{catequiv} below.

We finally observe that $\theta$ is injective by \cite[Cor. III.3.5.5]{K-book}. It only remains to show that $\delta$ is also injective. But $\eta$ must be the trivial map and $\kappa=\iota\circ\eta\circ\zeta_1^{-1}$, so $\kappa$ must also be the trivial map, hence $\ker(\delta)=\im(\kappa)=1$, as required.
\end{proof}

\begin{proposition}\label{catequiv} The scalar extension functor is an exact equivalence of categories $H_S(R[Z])\to H_S(\Lambda)$, whose inverse is the forgetful functor.\end{proposition}
\begin{proof}
Exactness follows from Lemma \ref{Lambdaflat} below. Equivalence is a special case of the general result \cite[Prop. V.7.5(1)]{K-book} but, for the sake of clarity and simplicity, we sketch here a more direct and explicit proof instead.
We implicitly use Lemma \ref{TMremark1}(i) for both $S=\Fq[Z]$ and $S=\Fq[\![Z]\!]$.

Since we only consider $S$-torsion modules, the inverse of the scalar extension functor is the forgetful functor, provided that it is a well-defined functor $H_S(\Lambda)\to H_S(R[Z])$.
We fix an object $M$ of $H_S(\Lambda)$ and show that $M$ also belongs to $H_S(R[Z])$. 

Now, $M$ is necessarily finitely generated over $R[Z]$. The existence of a finite resolution by finitely generated projective $\Lambda$-modules also implies that $M$ is $G$-c.t.

We now fix a short exact sequence of finitely generated $R[Z]$-modules
$$0\to\calK\to P\to M\to 0$$
in which $P$ is projective. Since both $M$ and $P$ are $G$-c.t., so must be $\calK$.

Since $\Fq[Z]$ is a principal ideal domain, $\calK$ is also $\Fq[Z]$-projective, so $\calK$ must be $R[Z]$-projective. This proves that $M$ belongs to $H_S(R[Z])$, as required. (In fact, picking the projective resolution of Lemma \ref{43.4}, with $Z$ for $t$, gives a direct proof.) 
\end{proof}

\begin{lemma}\label{Lambdaflat} $\Lambda$ is a flat $R[Z]$-module.
\end{lemma}
\begin{proof} By the equivalence of (i) and (iv) in \cite[Expos\'e IV,\S 5,p.98,Thm. 5.6]{SGA1}, it suffices to fix $n\in\NN$ and show that $\Lambda\otimes_{R[Z]}(R[Z]/Z^nR[Z])$ is a flat $R[Z]/Z^nR[Z]$-module.

We consider the commutative diagram of $R[Z]$-modules with exact rows
\begin{equation*}\minCDarrowwidth 1.8em
\begin{CD}
0 @> >> Z^nR[Z] @>  >> R[Z] @>   >> R[Z]/(Z^nR[Z]) @>  >>  0\\
@. @V  f_1 VV @V  f_2 VV @VV f_3 V\\
0 @>  >> Z^n\Lambda @>  >> \Lambda @>  >> \Lambda/(Z^n\Lambda) @>  >>  0.\\
@. @A g_1 AA @A  g_2 AA @AA g_3 A\\
@. \Lambda\otimes_{R[Z]}(Z^nR[Z]) @>  >> \Lambda\otimes_{R[Z]}R[Z] @>   >> \Lambda\otimes_{R[Z]}(R[Z]/(Z^nR[Z])) @>  >>  0.
\end{CD}\end{equation*}
We claim that both $f_3$ and $g_3$ are bijective, which implies that $\Lambda\otimes_{R[Z]}(R[Z]/Z^nR[Z])$ is a free $R[Z]/Z^nR[Z]$-module of rank one, and in particular flat.

Our claim follows from the snake lemma after observing that $g_1$ is surjective, that $g_2$ is bijective, that $f_2$ is injective and that the map
$$\cok(f_1)=(Z^n\Lambda)/(Z^nR[Z])\longrightarrow \Lambda/R[Z]=\cok(f_2)$$
induced by $Z^n\Lambda\subset\Lambda$ is bijective. 
\end{proof}

We finally state \cite[Thm. III.3.6]{K-book}, translated to our multiplicative notation.
\begin{proposition}\label{otherWeibel} The following sequence is exact:
$$1\to K_1(R)\xrightarrow{(\beta,\beta')}K_1(\Pi)\oplus K_1(R[Z])\xrightarrow{(\theta'/\theta)} K_1(R[Z^{-1},Z]).$$
Here $(\theta'/\theta)(y,w):=\theta'(y)\cdot\theta(w)^{-1}$.\end{proposition}

\subsubsection{} To complete the proof of Proposition \ref{injectivedeltas}, we finally prove that $\varepsilon$ is injective. We fix $\rho\in\ker(\varepsilon)$ and, using Lemma \ref{limKN}, we fix $x\in K_1(\Lambda)$ such that $\widehat\partial(x)=\rho$. Then $0=\varepsilon(\rho)=\varepsilon(\widehat\partial(x))=\partial(\delta(x))$, and we also fix $y\in K_1(\Pi)$ such that
$\gamma(y)=\delta(x)$.
\begin{lemma}\label{yw}\ 

\noindent{}(i)In the notation of Proposition \ref{otherWeibel}, there is $w\in K_1(R[Z])$ with $(y,w)\in\ker(\theta'/\theta)$.

\noindent{}(ii)For $w$ as in claim (i), one has $\iota(w)=x$.
\end{lemma}
\begin{proof}
By Proposition \ref{Weibel17Prop}, we have
\begin{align*}\mu(\theta'(y))=\zeta_0^{-1}(\zeta_0(\mu(\theta'(y))))=\zeta_0^{-1}(\nu(\lambda(\theta'(y))))=\zeta_0^{-1}(\nu(\gamma(y)))&=\zeta_0^{-1}(\nu(\delta(x)))\\ &=\zeta_0^{-1}(0)=0.\end{align*}
In the third equality, we have used the second equality from (\ref{relations}). Thus, using Proposition \ref{Weibel17Prop} again, there is $w\in K_1(R[Z])$ such that $\theta(w)=\theta'(y)$, i.e., $(y,w)\in\ker(\theta'/\theta)$.

For $w$ as in claim (i), we compute
$$\delta(x\cdot\iota(w)^{-1})=\delta(x)\cdot\delta(\iota(w))^{-1}=\gamma(y)\cdot\lambda(\theta(w))^{-1}=\gamma(y)\cdot\lambda(\theta'(y))^{-1}=\gamma(y)\cdot\gamma(y)^{-1}=1.$$
In the second, resp. fourth, equality, we have used the third, resp. second, equality in (\ref{relations}).
Since $\delta$ is injective by Proposition \ref{Weibel17Prop}, we now deduce that $\iota(w)=x$. \end{proof}

We are ready to combine Lemma \ref{yw} with Proposition \ref{otherWeibel} in order to finally show that $\rho$ is trivial. Indeed, there is $z\in K_1(R)$ such that $\beta'(z)=w$ with $\iota(w)=x$ (and $\beta(z)=y$).
But then
$$\rho=\widehat\partial(x)=\widehat\partial(\iota(w))=\widehat\partial(\iota(\beta'(z)))=\widehat\partial(\alpha(z))=0.$$
In the fourth equality we have used the first equality from (\ref{relations}).

\subsection{A finitely generated representative of the complex of units}

In this section we prove that the complex $C^\cm_{E,K/k}$ belongs to the category $D^{\rm p,c}(A[G],\Fi)$, as claimed in Theorem \ref{MT}. To do so we assume given a fixed taming module $\cm$ for $K/k$.

In the process we obtain a representative of $C^\cm_{E,K/k}$ that is more amenable to the computation, in \S \ref{completionMT}, of the refined Euler characteristic of $(C^\cm_{E,K/k},\lambda^{\cm,-1}_{E,K/k})$.

\begin{theorem}\label{PBT} Fix $A[G]$-modules $M^1$ and $M^2$ as in Theorem \ref{POdiag} (i). Fix a pull-back square of $A[G]$-modules \begin{equation}\label{PBdiag}
\begin{CD}
P^1 @> d >> P^2\\
@V \alpha  VV @VV \beta V\\
M^1 @> {\rm exp}_E^1 >> M^2,
\end{CD}\end{equation}
in which $P^2$ is finitely generated and free, $\beta$ is surjective and ${\rm exp}_E^1$ denotes the restriction of ${\rm exp}_E$ to $M^1$.
Then $P^1$ is locally-free of finite rank and the complex $$P^1\xrightarrow{(d,0)}P^2\oplus\cm,$$ 
with the first term in degree one,
is isomorphic in $D(A[G])$ to $C^\cm_{E,K/k}$. In particular, $C^\cm_{E,K/k}$ belongs to $D^{\rm p,c}(A[G],\Fi)$, as claimed in Theorem \ref{MT}. 
\end{theorem}
\begin{proof}
We first observe that the sequence
\begin{equation}\label{presentationh}0\to P^1\xrightarrow{(d,\alpha)}P^2\oplus M^1\xrightarrow{\beta-{\rm exp}_E^1}M^2\to 0.\end{equation}
is exact. Indeed, it is exact at the first two terms by \cite[Lem. III.1.1]{HiltonS}, and exact at $M^2$ because $\beta$ is assumed to be surjective, and $\beta(\pi)=(\beta-{\rm exp}_E^1)(\pi,0)$ for all $\pi\in P^2$.

We next recall that $M^1$, $P^2$ (by Lemma \ref{TMremark1}(i)) and $M^2$ (by exactness of the second row of the diagram in Theorem \ref{POdiag}) are $G$-c.-t.
By the exactness of (\ref{presentationh}), $P^1$ is $G$-c.-t., and also $A$-free of finite rank. Thus $P^1$ is finitely generated over $A[G]$, and $A[G]$-projective by Lemma \ref{TMremark1}(i).

In fact, $P^1$ is locally-free by Swan's Theorem (Lemma \ref{TMremark1}(iii)) because (\ref{presentationh}) induces an isomorphism $P^1_{\F}\cong(P^2\oplus M^1)_{\F}$, and $P^2$ and $M^1$ are locally-free.

We next set $\mathbb{U}:=\expm$ and $H:=H(E/\cm)$ and we note that the fixed pull-back square (\ref{PBdiag}) induces an exact commutative diagram of the form
\begin{equation}\label{PBExt}
\begin{CD}
0 @>  >> \mathbb{U} @>  >> P^1 @> d >> P^2 @>  >> H @>  >> 0\\
@. @\vert @V \alpha  VV @VV \beta V @\vert @.\\
0 @>  >> \mathbb{U} @>  >> M^1 @> {\rm exp}_E^1 >> M^2 @>  >> H @>  >> 0.\end{CD}\end{equation}

We write $D^\bullet$ for the complex of $A[G]$-modules $$K_\infty\,\xrightarrow{{\rm exp}_E}\,E(K_\infty)/E(\cm),$$ $M^\bullet$ for the complex $M^1\xrightarrow{{\rm exp}_E^1}M^2$ and $P^\bullet$ for the complex $P^1\stackrel{d}{\longrightarrow}P^2$, in all cases
with the first term placed in degree one.

Then, since
$$H^j(D^\bullet)=\begin{cases}\mathbb{U}, &\qquad j=1,\\ 
H, &\qquad j=2,\\
0, &\qquad j\neq 1,2,\end{cases}$$
these three complexes give rise to respective classes $\gamma_{D^\bullet}$, $\gamma_{M^\bullet}$ and $\gamma_{P^\bullet}$ in the Yoneda Ext-group $\Ext^2_{A[G]}(H,\mathbb{U})$, with $\gamma_{M^\bullet}=\gamma_{P^\bullet}$ via (\ref{PBExt}).

To complete the proof of the Theorem, after recalling Remark \ref{regRK}, it only remains to show that 
$D^\bullet\cong P^\bullet$ in $D(A[G])$. But, in order to prove that, in fact, $D^\bullet\cong M^\bullet\cong P^\bullet$, it is enough to prove that $\gamma_{D^\bullet}=\gamma_{M^\bullet}=\gamma_{P^\bullet}$ or, equivalently, that $\gamma_{D^\bullet}=\gamma_{M^\bullet}$.

For this we use the canonical commutative triangle of abelian groups
\begin{equation*}
\xymatrix{& \Ext^2_{A[G]}(H,\mathbb{U})  \\ \Ext^1_{A[G]}(H,M^1/\mathbb{U}) \ar@{->}[ur]^{\delta}  \ar@{->}[rr]^{\iota}  & &\Ext^1_{A[G]}(H,K_\infty/\mathbb{U}),  \ar@{->}[ul]_{\delta'}}
\end{equation*}
and we denote by $\epsilon$, resp. $\epsilon'$, the class in $\Ext^1_{A[G]}(H,M^1/\mathbb{U})$, resp. $\Ext^1_{A[G]}(H,K_\infty/\mathbb{U})$, of the first, resp. second, column of the diagram in Theorem \ref{POdiag}.

Then the left-hand square of the diagram in Theorem \ref{POdiag} is a push-out diagram by Lemma \ref{POlemma} below, which implies that $\iota(\epsilon)=\epsilon'$. It follows that $$\gamma_{D^\bullet}=\delta'(\epsilon')=\delta'(\iota(\epsilon))=\delta(\epsilon)=\gamma_{M^\bullet},$$ where the first and last equalities follow, for example, from the general description of connecting homomorphisms given in \cite[Lem. 3]{bufl96}.
\end{proof}

\begin{remark}\label{lflf}{\em The proof of Theorem \ref{PBT} uses that $M^1$ is locally-free, but does not require it to be free. Thus, by Remark \ref{TMremark2}, it remains valid for any projective $A[G]$-lattice $M^1$ as in Theorem \ref{POdiag} (i) and associated $A[G]$-module $M^2$.
}
\end{remark}

\begin{lemma}\label{POlemma} The left-hand square of the diagram in Theorem \ref{POdiag} is a push-out.
\end{lemma}
\begin{proof} This follows from the dual statement to \cite[Lem. III.1.3]{HiltonS}, which is left as Exercise III.1.3 and used throughout loc. cit., but not stated explicitly. To give a more direct argument, we write $\iota_1:M^1/\UU\to K_\infty/\UU$, $\iota_2:M^1/\UU\to M^2$, $\iota_3:M^2\to V_E$, $\pi_1:K_\infty/\UU\to K_\infty/M^1$ and $\pi_3:V_E\to K_\infty/M^1$ for the corresponding arrows in the diagram. Then by the dual statement to \cite[Lem. III.1.1]{HiltonS}, or by Exercise II.9.2.(ii) in loc. cit, the relevant square is a push-out diagram if and only if the sequence
$$(M^1/\UU)\xrightarrow{(\iota_1,\iota_2)}(K_\infty/\UU)\oplus M^2\xrightarrow{\exp_E-\,\iota_3}V_E\to 0$$
is exact. We proceed to prove this exactness claim.

Clearly $\im((\iota_1,\iota_2))\subseteq\ker(\exp_E-\,\iota_3)$. For the converse, assume that $\exp_E(\kappa)=\iota_3(m)$. Then $\pi_1(\kappa)=\pi_3(\iota_3(m))=0$, so there is $\mu\in M^1/\UU$ with $\iota_1(\mu)=\kappa$. Since then $\iota_3(\iota_2(\mu))=\exp_E(\kappa)=\iota_3(m)$ and $\iota_3$ is injective, we also have $\iota_2(\mu)=m$, as required.

It remains to show that $\exp_E-\,\iota_3$ is surjective. Let $v\in V_E$ and fix $\kappa\in K_\infty/\UU$ such that $\pi_1(\kappa)=\pi_3(v)$. Then $\pi_3(\exp_E(\kappa)-v)=\pi_1(\kappa)-\pi_1(\kappa)=0$, so there is $m\in M^2$ with $\iota_3(m)=\exp_E(\kappa)-v$, which means that $v=\exp_E(\kappa)-\iota_3(m)$, as required.
\end{proof}

\subsection{Completion of the proof of Theorem \ref{MT}}\label{completionMT}

We fix a pull-back square (\ref{PBdiag}) as in Theorem \ref{PBT}. By this result and Definition \ref{ECex}, $\chi_{G}(C^\cm_{E,K/k},\lambda^{\cm,-1}_{E,K/k})$ is equal to $-\left[P^1,\gamma,P^2\oplus\cm\right]$, where $\gamma$ denotes the composite isomorphism of $\FiG$-modules
$$P^1_{\Fi}\xrightarrow{(d,\alpha)_{\Fi}}P^2_{\Fi}\oplus M^1_{\Fi}\xrightarrow{{\rm id}\oplus\lambda}P^2_{\Fi}\oplus\cm_{\Fi}.$$

By Theorem \ref{POdiag} (ii), it is therefore enough to show that $$\left[P^1,\gamma,P^2\oplus\cm\right]\,=\,\left[M^1,\lambda,\cm\right]+\KFit_G(M^2)$$ or, equivalently, that
\begin{equation}\label{FIapplied}\left[P^1,(d,\alpha)_{\Fi},P^2\oplus M^1\right]\,=\,\KFit_G(M^2).\end{equation}
The latter equality is in turn a direct consequence of Lemma \ref{Fitindependence}, applied to the short exact sequence of $A[G]$-modules (\ref{presentationh}).
This completes the proof of Theorem \ref{MT}.

\section{The proof of Theorem \ref{mtII}}\label{7}

In this section we fix a finite Galois extension $K/k$ with the property that $G=\Gal(K/k)$ admits a decomposition of the form (\ref{semidirect}) (i.e., that $\ell$ does not divide the order of the commutator subgroup of $G$).

We then fix a Drinfeld module $E/\co_k$ and a taming module $\cm$ for $K/k$. We set $\UU:=\expm$ and $H:=H(E/\cm)$.

Since $H$ naturally surjects onto $H(E/\co_K)$, the final claim of Theorem \ref{mtII} follows directly upon combining the first claim with \cite[Thm. 2.2 (ii), (iii)]{jn}.

We now fix a projective $A[G]$-lattice $M^1$ as in Theorem \ref{POdiag} (i), as well as a corresponding $A[G]$-module $M^2$ and, by abuse of notation, we again identify $\lambda=\lambda^\cm_{E,K/k}$ with the isomorphism
$M^1_{\Fi}=\UU_{\Fi}\stackrel{\lambda}{\longrightarrow}\cm_{\Fi}$.

Now, since $M^2$ surjects onto $H$, and by applying \cite[Thm. 2.2 (iii)]{jn}, the proof of the first claim of Theorem \ref{mtII} is reduced to showing that
\begin{equation}\label{M2fit}Z(A[G])\cdot\left\{\theta^{\cm}_{E,K/k}\cdot R(\psi)\Bigm\vert\,\psi\in\Hom_{A[G]}\left(M^1,\cm\right)\right\}=\Fit_{A[G]}(M^2),\end{equation}
where the left-hand side denotes the $Z(A[G])$-submodule of $Z(\FiG)$ generated by the given subset, and
we have put
\begin{equation}\label{ell inv def} R(\psi)\, :=\, {\rm Nrd}_{\FiG}(\psi_{\Fi}\circ \lambda^{-1}).\end{equation}

We also note in passing that the equality (\ref{M2fit}), in combination with \cite[Thm. 2.2 (v)]{jn}, justifies the claim made in Remark \ref{finerexpectation}.

In order to now verify the equality (\ref{M2fit}), since non-commutative Fitting ideals commute with localisation (by \cite[Thm. 2.2 (vii)]{jn}), it is enough to show that, for every maximal ideal $P$ of $A$, 
one has
\begin{multline}\label{newenoughp}Z(A_{(P)}[G])\cdot\left\{\theta_{E,K/k}^\cm\cdot R(\psi)\Bigm\vert\,\psi\in\Hom_{A_{(P)}[G]}\left(A_{(P)}\otimes_A M^1,A_{(P)}\otimes_A\cm\right)\right\}\\ =\Fit_{A_{(P)}[G]}(A_{(P)}\otimes_A M^2),\end{multline}
where the left-hand side denotes the $Z(A_{(P)}[G])$-submodule of $Z(\FiG)$ generated by the given subset, and
with $R(\psi)$ defined exactly as in (\ref{ell inv def}).

We thus fix a maximal ideal $P$ of $A$, set $\B_P:=A_{(P)}[G]$ and, for every $A[G]$-module $M$, also $M_P:=A_{(P)}\otimes_A M$. We now proceed to verify the equality (\ref{newenoughp}).

By Remark \ref{TMremark2}, the $A[G]$-modules $\cm$ and $M^1$ are both locally-free of rank equal to $[k:\F]$. We thus may and do fix an isomorphism $\phi:M^1_P\cong\cm_P$ and observe that for any $\psi$ as in (\ref{newenoughp}), the difference
$$R(\psi)\cdot R(\phi)^{-1}\,=\,\Nrd_{\Fi[G]}((\psi\circ\phi^{-1})_{\Fi})\,=\,\Nrd_{\B_P}(\psi\circ\phi^{-1})$$
belongs to $Z(\B_P)$. 
To verify (\ref{newenoughp}) it therefore suffices to prove that 
\begin{equation}\label{newyetanotherreductionp}Z(\B_P)\cdot\theta_{E,K/k}^\cm\cdot R(\phi)=\Fit_{\B_P}(M^2_P).\end{equation}

We fix a pull-back square (\ref{PBdiag}) as in Theorem \ref{PBT} and abbreviate the map $(d,\alpha)$ in the sequence (\ref{presentationh}) to $h$.
We use the canonical commutative triangle of abelian groups
\begin{equation*}
\xymatrix{& K_0(A[G],\FiG) \ar@{->}[dr]^{\iota_P} \\ K_1(\FiG) \ar@{->}[ur]^{\partial_G}  \ar@{->}[rr]^{\partial_{G,P}}  & &K_0(\B_P,\FiG).  }
\end{equation*}

Then Theorem \ref{MT}, via Theorem \ref{POdiag} (ii) and the equality (\ref{FIapplied}), implies that
\begin{align*}\partial_{G,P}([h_{\Fi}])=&\iota_P\left(\KFit_G(M^2)\right)\\
=&\partial_{G,P}\left(\Theta^\cm_{E,K/k}\right)-[M^1_P,\lambda,\cm_P]\\
=&\partial_{G,P}\left(\Theta^\cm_{E,K/k}\right)+[\cm_P,\lambda^{-1},M^1_P]+[M^1_P,\phi_{\Fi},\cm_P]\\
=&\partial_{G,P}\left(\Theta^\cm_{E,K/k}\right)+[\cm_P,\phi_{\Fi}\circ\lambda^{-1},\cm_P]\\
=&\partial_{G,P}\left(\Theta^\cm_{E,K/k}\cdot \left[\phi_{\Fi}\circ\lambda^{-1}\right]\right).
\end{align*}
Now, since $$\Nrd_{\FiG}\left(\ker(\partial_{G,P})\right)=\Nrd_{\FiG}\left(\im\left(K_1(\B_P)\right)\right)=\Nrd_{\B_P}\left(K_1(\B_P)\right)=Z(\B_P)^\times,$$
with the last equality by Lemma \ref{passestoK}, we deduce the required equality (\ref{newyetanotherreductionp}) because
$$Z(\B_P)\cdot\theta_{E,K/k}^\cm\cdot R(\phi)=Z(\B_P)\cdot \Nrd_{\FiG}\left([h_{\Fi}]\right)=\Fit_{\B_P}(M^2_P).$$
Here the last equality holds by \cite[Prop. 3.4]{jn}.

This completes the proof of Theorem \ref{mtII}.

\Addresses
\end{document}